\documentclass[a4paper,11pt,leqno,english]{smfart}
\usepackage{aeguill}
\usepackage{enumerate}
\usepackage{amssymb,amsmath,latexsym,amsthm}
\usepackage[T1]{fontenc}
\usepackage{geometry}
\usepackage{url}
\usepackage[frenchb, main=english]{babel}
\usepackage[utf8]{inputenc}
\usepackage{mathrsfs}
\usepackage{xcolor}
\usepackage{comment}
\definecolor{violet}{rgb}{0.0,0.2,0.7}
\definecolor{rouge2}{rgb}{0.8,0.0,0.2}
\usepackage{tikz}
\usepackage{empheq}
\usepackage{tikz-cd}
\usetikzlibrary{matrix,arrows,decorations.pathmorphing}
\usepackage{hyperref}
\usepackage{mathpazo}
\hypersetup{
    bookmarks=true,         
    unicode=false,          
    pdftoolbar=true,        
    pdfmenubar=true,        
    pdffitwindow=false,     
    pdfstartview={FitH},    
    pdftitle={},    
    pdfauthor={},     
    colorlinks=true,       
   linkcolor=rouge2,          
    citecolor=violet,        
    filecolor=black,      
    urlcolor=cyan}           
\usepackage{enumitem}
\usepackage{appendix}

 \theoremstyle{plain}    
 \newtheorem{thm}{Theorem}[section]
\theoremstyle{plain} 
\newtheorem{bigthm}{Theorem}
\newtheorem{bigpro}[bigthm]{Proposition}

 \numberwithin{equation}{section} 
 \numberwithin{figure}{section} 
  \newtheorem{assumption}[thm]{Assumption} 
 \newtheorem{cor}[thm]{Corollary} 
 \theoremstyle{plain}    
 \newtheorem{prop}[thm]{Proposition} 
 \theoremstyle{plain}    
 \newtheorem{lem}[thm]{Lemma} 
 \theoremstyle{remark}
  \newtheorem{claim}[thm]{Claim} 
 \theoremstyle{remark}
 \newtheorem{rem}[thm]{Remark}
 \theoremstyle{definition}
\newtheorem{exa}[thm]{Example}
\theoremstyle{plain}  
\newtheorem{set}[thm]{Setting}
\theoremstyle{plain}

\newtheorem{pb}{Problem}
\theoremstyle{definition}
\newtheorem{defi}[thm]{Definition}

\newtheorem{conj}[thm]{Conjecture}

\newcommand{\N}{{\mathbb{N}}}

\newcommand{\Q}{{\mathbb{Q}}}
\newcommand{\R}{{\mathbb{R}}}

\newcommand{\bD}{{\mathbb{D}}}

\newcommand{\cO}{{\mathcal{O}}}

\newcommand{\cX}{{\mathcal{X}}}
\newcommand{\cY}{{\mathcal{Y}}}

\newcommand{\psh}{{\mathrm{PSH}}}

\def\1{\mathbf{1}}

\newcommand{\Krel}{K_{\cX/\bD}}

\newcommand{\xtr}{X_{t}^{\rm reg}}

\renewcommand{\a}{\alpha}

\newcommand{\e}{\varepsilon}
\newcommand{\om}{\omega}
\newcommand{\f}{\varphi}
\newcommand{\uf}{\underline {\varphi_t}}

\newcommand{\p}{\psi}

\newcommand{\omf}{\om_{\varphi_t}}

\newcommand{\xmb}{X\setminus B}

\newcommand{\vp}{\varphi}

\newcommand{\vpi}{\varphi_{\infty}}
\newcommand{\uvpi}{\underline{\varphi_{\infty}}}
\newcommand{\ep}{\varepsilon}

\newcommand{\Amp}{\mathrm{Amp}}

\newcommand{\Ric}{\mathrm{Ric}}

\newcommand{\reg}{\mathrm{reg}}

\renewcommand{\ge}{\geqslant}
\renewcommand{\le}{\leqslant}

\newcommand{\vol}{\operatorname{vol}}

\newcommand{\tr}{\operatorname{tr}}

\newcommand{\ca}{\operatorname{Cap}}
\newcommand{\PSH}{\operatorname{PSH}}
\newcommand{\MA}{\operatorname{MA}}

\title{Families of singular K\"ahler-Einstein metrics}
\date{\today}

\author{Eleonora Di Nezza}

  \address{Institut Math\'ematique de Jussieu, Sorbonne Universit\'e, France}
\email{eleonora.dinezza@imj-prg.fr}

\author{Vincent Guedj}

\address{Institut de Mathématiques de Toulouse; UMR 5219, Université de Toulouse; CNRS, UPS, 118 route de Narbonne, F-31062 Toulouse Cedex 9, France \quad}
\email{vincent.guedj@math.univ-toulouse.fr}

\author{Henri Guenancia}
\address{Institut de Mathématiques de Toulouse; UMR 5219, Université de Toulouse; CNRS, UPS, 118 route de Narbonne, F-31062 Toulouse Cedex 9, France}
\email{henri.guenancia@math.cnrs.fr}

\begin{document}
 

\begin{abstract}  
Refining Yau's and Kolodziej's techniques, 
we establish very precise uniform a priori estimates for degenerate complex Monge-Amp\`ere equations on compact K\"ahler manifolds,
that allow us to control the blow up of the solutions as the cohomology class and the complex structure both vary.

We apply these estimates to the study of various families of 
possibly singular K\"ahler varieties endowed with twisted K\"ahler-Einstein metrics, by 
analyzing the behavior of canonical densities, 
establishing uniform integrability properties, and
developing the first steps of a pluripotential theory in families. 
This provides interesting information on the moduli space of stable varieties, extending works by Berman-Guenancia and Song,
as well as on the behavior of singular Ricci flat metrics on (log) Calabi-Yau varieties, generalizing works by Rong-Ruan-Zhang, Gross-Tosatti-Zhang,
Collins-Tosatti and 
Tosatti-Weinkove-Yang.
\end{abstract} 

\maketitle

\tableofcontents

\section*{Introduction}

 Let $p: X \rightarrow Y $ be a proper, surjective holomorphic map with connected fibers between K\"ahler varieties.
 It is a central question in complex geometry to relate the geometry of $X$ to the one of $Y$ and the fibers 
 $X_y$ of $p$. 
 An important instance of such a situation is when  one can endow $X_y$  with a   K\"ahler-Einstein metric and study the geometry of $X$ induced by the properties of the resulting family of metrics. This is the main theme of this article.

\smallskip

Einstein metrics are a central object of study  in differential geometry. 
A K\"ahler-Einstein metric on a complex manifold is a K\"ahler metric whose
Ricci tensor is proportional to the metric tensor.
This notion still makes sense on midly singular varieties as was observed in \cite[section 7]{EGZ}.
The solution of the (singular) Calabi Conjecture  \cite{Yau78,EGZ} provides a very powerful existence theorem for 
K\"ahler-Einstein metrics with negative or zero Ricci curvature.
It is important to study the ways in which these canonical metrics behave when they are moving in families.
  In this paper we consider the case when both the complex structure and the K\"ahler class  vary 
  and we try and understand how the corresponding metrics can degenerate.
 
\smallskip

Constructing singular K\"ahler-Einstein metrics on a midly singular variety $V$ boils down to solving degenerate
complex Monge-Amp\`ere equations of the form
$$
(\omega+i \partial \overline{\partial} \f)^n=f e^{\lambda \f} dV_X,
$$
where 
\begin{itemize}
\item $\pi:X \rightarrow V$ is a resolution of singularities, $dV_X$ is a volume form on $X$,
\item $\omega=\pi^* \omega_V$ is the pull-back of a K\"ahler form on $V$,
\item the sign of $\lambda \in \R$ depends on that of $c_1(V)$,
\item $f \in L^p(X)$ with $p>1$ if the singularities of $V$ are  mild (klt singularities),
\end{itemize}
and 
$\f$ is the unknown. The latter should be $\omega$-plurisubharmonic
($\omega$-psh for short),
 i.e. it is locally the sum of a psh and a smooth function, and  satisfies
$\omega+i \partial \overline{\partial} \f \geq 0$ in the   weak sense of currents.
We let $PSH(X,\omega)$ denote the set of all such functions.

\subsection*{The uniform estimate}

A crucial step in order to prove the existence of a solution to the above equation
is to establish a uniform a priori estimate.
In order to understand the behavior of the solution $\f$ as the cohomology
class $\{\omega_V\}$ and the complex structure of $V$ vary,
we revisit the proof by Yau \cite{Yau78}, as well as its recent generalizations \cite{Kolo,EGZ},
and establish the following (see Theorem \ref{thm:uniform1}):

\begin{bigthm}
\label{thmA}
Let $X$ be a compact K\"ahler manifold of complex dimension $n \in \N^*$
and let $\omega$ be a semi-positive form   such that
$
V:= \int_X \omega^n >0.
$
Let $\nu$ and $\mu=f \, \nu$ be probability measures, with
$0 \le f \in L^p(\nu)$ for some $p>1$.
Assume the following  assumptions are satisfied:
\begin{enumerate}
\item[(H1)] there exists $\a>0$ and $A_\a>0$ such that for all $\p \in \PSH(X,\omega)$,
$$
\int_X e^{-\a (\p-\sup_X \p) } d\nu \le A_\a;
$$
\item[(H2)] there exists $C>0$ such that $\left( \int_X |f|^p \, d\nu \right)^{1/p} \le C$.
\end{enumerate}

Let $\f$ be the unique $\omega$-psh solution $\f$ to the complex Monge-Amp\`ere equation
$$
{V}^{-1} (\omega+i \partial \overline{\partial} \f)^n =\mu,
$$
normalized by $\sup_X \f=0$. Then $-M \le \f \le 0$ where
$$
M=1+ C^{1/n}A_\a^{1/nq} \, e^{\a/nq} b_n \left[5+e\a^{-1} C (q !)^{1/q} A_{\a}^{1/q}  \right] ,
$$
$1/p+1/q=1$ and $b_n$ is a constant such that $\exp(-1/x) \le b_n^n x^{2n}$ for all $x>0$.
\end{bigthm}

\bigskip

\begin{rem}
\label{no charge pp}
Let us observe that the condition (H1) in Theorem~\ref{thmA} above guarantees that the measure $\nu$ does not charge pluripolar sets, since any such set can be included in the polar locus of a \textit{global} $\om$-psh function by \cite[Thm.~7.2]{GZ05}. The existence (and uniqueness) of the solution $\vp$ in Theorem~\ref{thmA} follows from \cite[Thm.~A]{BEGZ}.
\end{rem}

We also establish slightly more general versions of Theorem~\ref{thmA} 
valid for less regular densities (Theorem   \ref{thm:uniform3}) or 
big cohomology classes (Theorem   \ref{thm:uniform2}).
We then move on to checking hypotheses (H1) and (H2) in various geometrical contexts.

\smallskip

\noindent $\bullet$ {\it Hypothesis (H1).}
If $\pi:{\mathcal X} \rightarrow \mathbb D$ is a projective family whose fibers
$X_t=\pi^{-1}(t)$ have degree $d$ with respect to a given projective embedding ${\mathcal X} \subset \mathbb{P}^N\times \bD$,
and $\omega=\omega_t$ is the restriction of the Fubini-Study metric,
we observe in Proposition \ref{prop:uniformint_proj} that 
$$
V=\int_{X_t}  \omega_t^n=\int_{\mathbb{P}^N} \omega_{FS}^n \wedge [X_t]=d
$$
is independent of $t$ and the following uniform integrability holds. 

\begin{bigpro}
For for all $\p \in \PSH(X_t,\omega_t)$,
 $$
  \int_{X_t} e^{-\frac{1}{nd}(\p-\sup_{X_t} \p)} \omega_t^n\le (4n)^n\cdot d \cdot 
  \exp\left\{-\frac{1}{nd}\int_{X_t} (\p-\sup_{X_t} \p) \, \omega_t^n\right\}.
  $$
  \end{bigpro}
The hypothesis (H1) is thus satisfied in this projective setting, with $\a=1/nd$, as soon as we can uniformly control
the $L^1$-norm of $\p$. We take care of this in Section \ref{sec:normalization}. This non-trivial
control requires the varieties $X_t$ to be irreducible
(see Example \ref{exa:irred}).

Bypassing the projectivity assumption, we show that (H1) is actually satisfied for 
many K\"ahler families of interest, by generalizing a uniform integrability result of Skoda-Zeriahi \cite{Sko72,Zer01}
(see Theorem \ref{prop:uniformint}). This is the content of Theorem~\ref{H1}.\\

\smallskip

\noindent $\bullet$ {\it Hypothesis (H2).}
We analyze (H2) 
 in section \ref{sec:H2}. We show that, up to shrinking the base, it is always satisfied if
the $f_t$'s are canonical densities associated to a proper, holomorphic surjective map  $\pi:\cX\to \bD$ from a normal, $\Q$-Gorenstein Kähler space $\cX$ to the unit disk
 such that the central fiber has only canonical singularities, cf Lemma~\ref{canon} and its application to families of Calabi-Yau varieties, Theorem~\ref{thmCY}. 

While previous works tend to use sophisticated arguments from Variations of Hodge Structures (see e.g. the Appendix by Gross in
\cite{RZ}), we use here direct elementary computations in adapted coordinates, in the spirit of \cite[section 6]{EGZ}.

In the context of families of varieties with negative curvature though, it is essential to allow worse singularities than the ones described above, cf Setting~\ref{fset} for the precise context. The trade-off is that the canonical densities do not satisfy condition (H2) anymore, reflecting the fact that the local potentials of the Kähler-Einstein metrics at stake need not be bounded anymore. This legitimizes the introduction of 
a weaker condition (H2') (see Theorem \ref{thm:uniform3} and Lemma~\ref{corintcontrol}). This allows us to derive an almost optimal control of the potentials of Kähler-Einstein metrics along a stable family, cf Theorem~\ref{thmC} below. 

Let us end this paragraph by emphasizing that our approach enables us to work with singular families (i.e. families where the generic fiber is singular, cf Theorems~\ref{thmC} and \ref{thmCY}) as opposed to all previously known results on that topic, requiring to approximate a singular variety by smooth ones using either a smoothing or a crepant resolution. 

We now describe more precisely four independent geometric settings to which we apply the uniform estimate provided by Theorem~\ref{thmA}.

\subsection*{Families of manifolds of general type}

Let $\mathcal X$ be an irreducible and reduced complex space endowed with a Kähler form $\omega$ and a proper, holomorphic map $\pi:{\mathcal X} \rightarrow \mathbb D$. We assume that for each $t\in \mathbb D$, the (schematic) fiber $X_t$ is a $n$-dimensional K\"ahler manifold $X_t$ of general type, i.e. such that its canonical bundle $K_{X_t}$ is big. In particular, $\mathcal X$ is automatically non-singular and the map $\pi$ is smooth. 

We fix $\Theta$ a closed differential $(1,1)$-form on ${\mathcal X}$ 
which represents $c_1(K_{\mathcal X/\mathbb D})$
and set $\theta_t=\Theta_{|X_t}$.

It follows from \cite{BEGZ}, a generalization of the Aubin-Yau theorem \cite{Aubin,Yau78}, that there exists
a unique K\"ahler-Einstein current on $X_t$. This is a positive closed current $T_t$ in $c_1(K_{X_t})$ which
 is a smooth K\"ahler form in the ample locus $\Amp(K_{X_t})$, where it satisfies
 the K\"ahler-Einstein equation
 $$
 \Ric(T_t)=-T_t.
 $$
 
  It can be written
$T_t=\theta_t+dd^c \f_t$, where $\f_t$ is the unique $\theta_t$-psh function with minimal singularities
that satisfies the complex Monge-Amp\`ere equation
$$
(\theta_t+dd^c \f_t)^n=e^{\f_t+h_t}  \omega_t^n  \quad \mbox{on } \Amp(K_{X_t}),
$$
where $h_t$ is such that $\Ric( \om_t)- dd^c h_t=-\theta_t$ and $\int_{X_t} e^{h_t}\om_t^n= \vol(K_{X_t})$. 
For $x\in \cX$, set 
\begin{equation}
\label{phi}
\phi(x):=\f_{\pi(x)}(x)
\end{equation}
 and consider
\begin{equation}
\label{Vtheta}
{V}_{\Theta}=\sup \{u  \in \PSH(\mathcal X,\Theta);  \,\, u \le 0 \}.
\end{equation}
We prove that conditions (H1) and (H2) are satisfied in this setting. It then follows from Theorem~\ref{thmA}  and the plurisubharmonic variation of the $T_t$'s (\cite[Thm.~A]{JHM1}) that $\phi-{V}_{\Theta}$ is uniformly bounded
on compact subsets of   ${\mathcal X}$, cf Theorem~\ref{thmgt} and Remark~\ref{remgt}:

\bigskip

\begin{bigthm}
\label{thmB}
Let $\pi:\cX\to \bD$ be a smooth Kähler family of manifolds of general type, let $\Theta\in c_1(K_{\cX/\bD})$ be a smooth representative and let $\phi$ be the Kähler-Einstein potential as in \eqref{phi}. Given any
 compact subset ${\mathcal K} \Subset {\mathcal X}$, there exists a constant $M_{{\mathcal K}}$ such that 
the following inequality
$$
-M_{{\mathcal K} } \le\phi-V_{\Theta} \le M_{{\mathcal K} }
$$
holds on $\mathcal K$, where $V_{\Theta}$ is defined by \eqref{Vtheta}. 
\end{bigthm}

\bigskip

The same results can be proved if the family $\pi:\cX\to \mathbb D$ is replaced by a smooth family $\pi:(\mathcal X,B)\to \mathbb D$ of pairs 
$(X_t,B_t)$ of log general type, i.e. such that $(X_t,B_t)$ is klt and $K_{X_t}+B_t$ is big for all $t\in \mathbb D$.

\bigskip

 \subsection*{Stable families}

 A stable variety is a projective variety $X$ such that 
  $X$ has semi-log-canonical singularities and 
  the $\mathbb Q$-line bundle $K_X$ is ample. 
We refer to \cite{Kov,Kollar} for a detailed account of these varieties and their connection to moduli theory.

In \cite{BG}, it was proved that a stable variety admits a unique Kähler-Einstein metric $\omega$, i.e.
  a smooth Kähler metric on $X_{\reg}$ such that,  if $n=\dim_{\mathbb C}X$, 
$$
\Ric(\om)=-\om \quad \mbox{and} \quad \int_{X_{\rm reg}}\om^n=(K_X^n).
$$

  The metric $\om$ extends canonically across $X_{\rm sing}$ to a closed, positive current in the class $c_1(K_X)$. It is desirable to understand the singularities of $\om$ near $X_{\rm sing}$. In \cite[Thm.~B]{GW}, it is proved that $\omega$ has cusp singularities near the double crossings of $X$. Moreover, it is proved in \cite{Song17} that the potential $\vp$ of $\om$ with respect to a given Kähler form $\om_X\in c_1(K_X)$,
   i.e. $\omega=\omega_X+dd^c \f$,  is locally bounded on the klt locus of $X$. In order to refine the latter assertion, let us introduce the following invariant. 
   
   Given a log resolution of singularities $f:Y\to X$ write $K_Y=f^*K_X+\sum_{i\in I} a_iE_i$ where $E=\sum_{i\in I}E_i$ is a divisor with simple normal crossings whose components are either exceptional or sit above the divisorial components of $\mathrm{Sing}(X)$. We set $J:=\{i\in I; a_i=-1\}$ which we assume to be non-empty and define the non-klt locus of $X$ by $\mathrm{Nklt}(X)=\cup_{j\in J}f(E_j)$; it is independent of $f$. Given $K\subset I$, we set $E_K=\cap_{i\in K}E_i$ and introduce the invariants
   \begin{equation}
   \label{nu0}
   \nu(f):=\max\{|K|; K\subset J, E_K\neq \emptyset\}, \quad \mbox{and} \quad \nu:=\inf_f \{\nu(f)\}
   \end{equation}
    where the infimum is taken over all the log resolutions of $X$. We have $\nu \in \{1, \ldots, n\}$. 
   
 We establish  the following (cf. Proposition~\ref{minoration}). 

\begin{bigpro}
For any $\ep>0$, there is a constant $C_{\ep}$ such that 
\begin{equation}
\label{minoration_intro}
C_1 \ge \vp \ge -(n+\nu+\ep) \log (-\log |s|)-C_{\ep}
\end{equation}
where $(s=0)$ is any reduced divisor containing the non-klt locus of $X$ and $\nu$ is as in \eqref{nu0}.
\end{bigpro}

This estimate is almost optimal as illustrated by the Baily-Borel-Satake compactification of a ball quotient or a Hilbert modular surface, cf Examples~\ref{BBS1}-\ref{BBS2}. \\

A slight refinement of Theorem~\ref{thmA} (cf. Theorem \ref{thm:uniform3}) allows us to establish a uniform family version of the estimate \eqref{minoration_intro}. In order to state it, let $\cX$ be a normal Kähler space and let $\pi:\cX \to \bD$ be a proper, surjective, holomorphic map such that each fiber $X_t$ has slc singularities and $\Krel$ is an ample $\mathbb Q$-line bundle. If $\om_{\cX}\in c_1(\Krel)$ is a relative Kähler form and $\om_{X_t}:=\omega_\cX|_{X_t}$, then the Kähler-Einstein metric of $X_t$ can be written as $\om_{X_t}+dd^c \f_t$ where $\f_t$ is uniquely determined by the equation \eqref{KEt} from section~\ref{sec:negative}. 

\smallskip

   Given a semistable model $f:\cY\to \cX$ write $K_\cY+Y_0=f^*(K_\cX+X_0)+\sum_{i\in I} a_iE_i$ where $Y_0$ is the strict transform of $X_0$. We set $J:=\{i\in I; a_i=-1\}$ which we assume to be non-empty and define the non-klt locus of $(\cX,X_0)$ by 
   \begin{equation}
   \label{nonkltlocus}
   \mathrm{Nklt}(\cX,X_0)=\cup_{j\in J}f(E_j);
   \end{equation}
    it is independent of $f$. Given $K\subset I$, we set $E_K=\cap_{i\in K}E_i$ and introduce the invariants
   \begin{equation}
   \label{nunu}
   \nu(f):=\max\{|K|; K\subset J, E_K\neq \emptyset\}, \quad \mbox{and} \quad \nu:=\inf_f \{\nu(f)\}
   \end{equation}
   where the infimum is taken over all the semistable models $f:\cY\to \cX$. We have $\nu \in \{1, \ldots, n\}$.

The behavior of $\f_t$ is then described by the following (see Theorem \ref{thm:stablefamily})

\begin{bigthm} \label{thmC}
Let $\cX$ be a normal Kähler space and let $\pi:\cX \to \bD$ be a proper, surjective, holomorphic map such that
\begin{enumerate}
\item[$\bullet$] Each schematic fiber $X_t$ has semi-log-canonical singularities.
\item[$\bullet$] $\Krel$ is an ample $\mathbb Q$-line bundle.
\end{enumerate} 
In particular, $X_t$ is a stable variety for any $t\in \bD$. Assume additionally that the central fiber $X_0$ is irreducible. 

Let $\om_{X_t}+dd^c \f_t$ be the Kähler-Einstein metric of $X_t$ and let $D=(s=0)\subset \cX$ be a divisor which contains $\mathrm{Nklt}(\cX,X_0)$, cf \eqref{nonkltlocus}.
Fix some smooth hermitian metric $|\cdot|$ on $\mathcal O_\cX(D)$. 
Up to shrinking $\bD$, then for any $\ep>0$, there exists 
$C_{\ep}>0$  such that the inequality
\begin{equation*}
C_1 \ge \vp_t \ge -(n+\nu+\ep) \log (-\log |s|)-C_{\ep}
\end{equation*}
holds on $X_t$ for any $t\in \bD$, where $\nu$ is as in \eqref{nunu}. 
\end{bigthm}
 
  Let us finally mention the very recent results of Song, Sturm and Wang \cite[Proposition 3.1]{SSW} where similar bounds are derived in the context of smoothings of stable varieties over higher dimensional bases, with application towards Weil-Petersson geometry of the KSBA compactification of canonically polarized manifolds.

  \subsection*{Families of $\mathbb Q$-Calabi-Yau varieties}

A $\mathbb Q$-Calabi-Yau variety is a compact, normal Kähler space $X$ with canonical singularities such that the $\mathbb Q$- line bundle $K_X$ is torsion. Up to taking a finite, quasi-étale cover referred to as the index $1$ cover (cf e.g. \cite[Def.~5.19]{KM}), one can assume that $K_X\sim_{\mathbb Z} \mathcal O_X$.  Given any Kähler class $\alpha$ on $X$, it follows from \cite{EGZ} and \cite{Paun} that there exists a unique singular Ricci flat Kähler metric $\omega_{\rm KE}\in \alpha$, i.e. a closed, positive current $\omega_{\rm KE} \in \alpha$ with globally bounded potentials inducing a smooth, Ricci-flat Kähler metric on $X_{\reg}$. 
  
  Now, we can consider families of such varieties and ask how the bound on the potentials vary. This is the content of the following (see Theorem~\ref{thm:CYfamily} and Remark~\ref{cansing})

\begin{bigthm}
 \label{thmCY} 
 Let  $\cX$ be a normal, $\mathbb Q$-Gorenstein Kähler space and let $\pi:\cX\to \bD$ be a proper, surjective, holomorphic map. Let $\alpha$ be a relative Kähler cohomology class on $\cX$ represented by a relative Kähler form $\om$.
Assume additionaly that
\begin{enumerate}
\item[$\bullet$] The relative canonical bundle $\Krel$ is trivial.
\item[$\bullet$] The central fiber $X_0$ has canonical singularities. 
\item[$\bullet$] Assumption~\ref{aspt} is satisfied.
\end{enumerate} 
Up to shrinking $\bD$, each fiber $X_t$ is a $\mathbb Q$-Calabi-Yau variety. Let $\omega_{{\rm KE}, t}=\om_t+dd^c \vp_t $ be the singular Ricci-flat Kähler metric in $\alpha_t$, normalized by $\int_{X_t} \vp_t \om_t^n=0$. Then, given any compact subset $K\Subset \mathbb D$, there exists 
$C=C(K)>0$  such that one has
\begin{equation*}
\mathrm{osc}_{X_t} \vp_t \le C
\end{equation*}
 for any $t\in K$, where $\mathrm{osc}_{X_t}(\vp_t)=\sup_{X_t}\vp_t-\inf_{X_t}\vp_t$. 
\end{bigthm}

In the case of a projective smoothing (i.e. when $\cX$ admits a $\pi$-ample line bundle and $X_t$ is smooth for $t\neq 0$), the result above has been obtained previously by Rong-Zhang \cite{RZ} by using Moser iteration process.

  \subsection*{Log Calabi-Yau families}

 Let $X$ be 
a  compact Kähler manifold and let $B=\sum b_iB_i$ be an effective $\mathbb R$-divisor such that the pair 
 $(X,B)$ has klt singularities and $c_1(K_X+B)=0.$
 
 It follows from \cite{Yau78,EGZ,BEGZ} that one can find a unique Ricci flat metric  in each K\"ahler  class $\a_t$.
 A basic problem is to understand the asymptotic behavior of these   metrics as $\a_t$ approaches the boundary of the 
 K\"ahler cone.
 Despite  motivations coming from mirror symmetry, not much 
is known when the norm of $\a_t$ converges to $+\infty$
(this case is expected to be the mirror of a large complex structure limit,
see \cite{KS01}). We thus only consider the case
when $\a_t \rightarrow \a_0 \in \partial {\mathcal K}_X$.

The non-collapsing case ($\vol(\a_0)>0$) can be  easily understood by using  Theorem~\ref{thmA} (see Theorem \ref{thm:non-collapsing}).
We describe here a particular instance of the more delicate collapsing case
 $\vol(\a_0)=0$.
Let $f:X\to Z$ be a surjective holomorphic map with connected fibers, where $Z$ is a normal Kähler space. 
 Let $\omega_X$ (resp. $\om_Z$) be a Kähler form on $X$ (resp. $Z$).
 Set $\omega_t:=f^*\om_Z+t \omega_X$.
There exists a unique singular Ricci-flat current $\omf:=\omega_t+dd^c \f_t$  in $\{f^*\om_Z+t\om_X\}$ for $t>0$,
where $\int_{X}\f_t\om_X^n=0$.  It satisfies
$$
 \omf^n=V_t\cdot \mu_{(X,B)},
\; \; \text{ where } \; \; 
\mu_{(X,B)}= (s\wedge \bar s)^{\frac 1m} e^{-\phi_B}.
$$
 Here, $s\in H^0(X,m(K_X+B))$ is any non-zero section (for some $m\ge 1$) and $\phi_B$ is the unique singular psh weight on $\mathcal O_X(B)$ solving $dd^c \phi_B=[B]$ and normalized by 
 $$
 \int_X (s\wedge \bar s)^{\frac 1m} e^{-\phi_B}=1.
 $$ 
 
 The probability measure $f_*\mu_{(X,B)}$ has $L^{1+\ep}$-density with respect to $\om_Z^m$ thanks to \cite[Lem.~2.3]{EGZ16}. 
 It follows therefore from \cite{EGZ} that there exists a unique current $\om_{\infty}\in \{\om_Z\}$ solution of the Monge-Ampère equation 
$$
\om_{\infty}^m= f_*\mu_{(X,B)}.
$$
 In the case where $X$ is smooth, $B=0$ and $c_1(X)=0$, the Ricci curvature of   $\om_{\infty}$
coincides with the Weil-Petersson form of the fibration $f$ of Calabi-Yau manifolds.

Understanding the asymptotic behavior of the $\omf$'s as $t \rightarrow 0$ is an important problem with
 a long history, we refer the reader to the thorough survey \cite{TosSurvey}
for references. We prove here the following:

 \bigskip

\begin{bigthm}
\label{thmD}
Let $(X,B)$ be a log smooth klt pair such that $c_1(K_X+B)=0$ and such that $X$ admits a fibration $f:X\to Z$. 
 With the notations above, the conic Ricci-flat metrics $\omf \in \{f^*\om_Z+t\om_X\}$ converge to $f^*\om_{\infty}$ as currents on $X$ when $t$ goes to $0$. 
\end{bigthm}

\bigskip

When $B=0$ is empty, it has been shown in \cite{Tos10,GTZ13,TWY18,HT18} that he metrics $\omf$ converge to $f^*\om_{\infty}$
in the ${\mathcal C}^\alpha$-sense on compact subsets of $X \setminus S_X$ for some $\alpha>0$, where  $S_X=f^{-1}(S_Z)$ and
 $S_Z$ denotes the smallest proper analytic subset $\Sigma \subset Z$ such that  $\Sigma$ contains the singular locus $Z_{\rm sing}$ of $Z$ and the map $f$ is smooth on $f^{-1}(Z\setminus \Sigma)$.
  
The proof of Theorem~\ref{thmD} follows the  strategy developed by the above papers with several twists that notably
require the extensive use of Theorem~\ref{thmA} and conical metrics.


  \subsection*{Acknowledgements.}    
 We thank S.Boucksom, M.P\u{a}un, J.Song and A.Zeriahi for several interesting discussions. We are grateful to the anonymous referees for a remarkably careful reading, for suggesting many improvements as well as for pointing out a gap in ~\S3 of an earlier version. Finally, we thank Rui Tang for pointing out an error in the published version.
 The authors are partially supported by the ANR project GRACK.


\section{Chasing the constants}

Our goal in this section is to establish the following a priori estimate
which is a refinement of the main result of Kolodziej \cite{Kolo} (see also \cite{ EGZ,EGZ08,DemPal}):

\begin{thm} \label{thm:uniform1}
Let $(X,\omega_X)$ be a compact K\"ahler manifold of complex dimension $n \in \N^*$
and let $\omega$ be a semi-positive form which is big, i.e. such that
$$
V:={\rm{Vol}}_\omega(X)=\int_X \omega^n >0.
$$
Let $\nu$ and $\mu=f \, \nu$ be probability measures, with
$0 \le f \in L^p(\nu)$ for some $p>1$.
Assume the following two assumptions are satisfied:
\begin{enumerate}
\item[(H1)] there exists $\a>0$ and $A_\a>0$ such that for all $\p \in \PSH(X,\omega)$,
$$
\int_X e^{-\a (\p-\sup_X \p) } d\nu \le A_\a;
$$
\item[(H2)] there exists $C>0$ such that $\left( \int_X |f|^p \, d\nu \right)^{1/p} \le C$.
\end{enumerate}

Let $\f$ be the unique $\omega$-psh solution $\f$ to the complex Monge-Amp\`ere equation
$$
{V}^{-1} (\omega+dd^c \f)^n =\mu,
$$
normalized by $\sup_X \f=0$. Then $-M \le \f \le 0$ where
$$
M=1+ C^{1/n}A_\a^{1/nq} \, e^{\a/nq} b_n \left[5+e\a^{-1} C (q !)^{1/q} A_{\a}^{1/q}  \right] ,
$$
$1/p+1/q=1$ and $b_n$ is a constant such that $\exp(-1/x) \le b_n^n x^{2n}$ for all $x>0$.
\end{thm}

Here $d=\partial+\overline{\partial}$ and $d^c=\frac{i}{2} (\partial-\overline{\partial})$
so that $dd^c =i\partial\overline{\partial}$. Recall that a function $\f:X \rightarrow \R \cup \{-\infty\}$
is $\omega$-plurisubharmonic ($\omega$-psh for short) if it is locally given as the sum of a 
smooth and a psh function, and such that $\omega+dd^c \f \ge 0$ in the weak sense of currents.
We let $\PSH(X,\omega)$ denote the set of all $\omega$-psh functions.

The non-pluripolar Monge-Amp\`ere measure of arbitrary $\omega$-psh functions has been defined in \cite{BEGZ}.
It follows from assumption (H1) that the measure $\mu$ does not charge pluripolar sets,
since the latter can be defined by $\omega$-psh functions (as follows easily from \cite[Thm. 7.2]{GZ05} since a big class contains a Kähler current).
The existence of a unique normalized $\omega$-psh solution to
${V}^{-1} (\omega+dd^c \f)^n =\mu$ follows from \cite[Theorem A]{BEGZ}
(the case of K\"ahler forms had been earlier treated in \cite{GZ07,Din09}).

\medskip

We will use this result to obtain uniform a priori estimates on normalized solutions $\f_t$ to families of complex Monge-Amp\`ere  equations
$$
{V_t}^{-1} (\omega_t+dd^c \f_t)^n =\mu_t,
$$
when the hypotheses (H1,H2) are satisfied, i.e. the constants $1/\a_t, A_{\alpha_t},q_t,C_t$ in the theorem are actually bounded from above by
 uniform constants $1/\a, A,q,C$ independent of $t$.
Here $q$ denotes the conjugate exponent of $p>1$, $1/p+1/q=1$.
The assumption on this exponent is thus that $p>1$ stays bounded away from 1.

The reader should keep in mind that assumption (H1) is the strongest of all.
In some applications one can assume  $f \equiv 1$
hence (H2) is trivially satisfied.

\medskip

We are going to eventually obtain a version of Theorem \ref{thm:uniform1}
that applies to big cohomology classes, extending \cite[Theorem B]{BEGZ}. The proof is almost identical but 
explaining the statement requires to introduce various notions and technical  notations, so we
first treat the case of semi-positive classes and postpone 
this to section \ref{sec:big}.

\subsection{Preliminaries on capacities}


Let $K\subset X$ be a Borel set and consider
$$
V_{K,\omega}:=\left(\sup \{ \p \, | \, \p \in \PSH(X,\omega) \text{ and } \p \le 0
\text{ on } K \} \right)^*,
$$
where $^*$ denotes the upper semi-continuous regularization.

The Alexander-Taylor capacity is the following:
$$
T_\omega(K):=\exp\left(-\sup_X V_{K,\omega} \right).
$$
It is shown in \cite[Lem.~9.17]{GZ17} that If $K$ is pluripolar then $V_{K,\omega} \equiv +\infty$ and $T_\omega(K)=0$.
When $K$ is not pluripolar then 
\begin{itemize}
\item $0 \le V_{K.\omega} \in \PSH(X,\omega)$ and $V_{K,\omega} =0$ on $K$ off a pluripolar set;
\item the Monge-Amp\`ere measure $\MA(V_{K,\omega} )$ is concentrated on $E$.
\end{itemize}
 We denote here and in the sequel by 
$$
\MA(u)=\frac{1}{V}(\omega+dd^c u)^n
$$ 
the normalized Monge-Amp\`ere measure
of a $\omega$-psh function $u$, where
$
V=\int_X \omega^n =\{ \omega \}^n
$
is the volume  of the cohomology class $\{ \omega \}$. It is defined for any $\om$-psh function $u$, cf. e.g. \cite[\S~1.1]{GZ07}.
For a Borel set $K\subset X$, the Monge-Amp\`ere capacity is
$$
\ca_\omega(K):=\sup \left\{ \int_K \MA(u) \; ; \; u \in \PSH(X,\omega) \text{ and } 
0 \le u \le 1 \right\}.
$$
This capacity also characterizes pluripolar sets, i.e.
$$
\ca_\omega^*(P)=0
\Longleftrightarrow P
\text{ is pluripolar}.
$$
Here $\ca_\omega^* $ is the outer capacity associated to $\ca_\omega$ defined for any  set $E\subset X$ as
$$\ca_\omega^* (E):=\inf\{\ca_\omega(G)\; ;\; G \;{\rm open}, E\subset G\}.$$ Moreover, if $K\subset X$ is a compact set than $\ca_\omega^* (K)= \ca_\omega (K)$.

The Monge-Amp\`ere and the Alexander-Taylor capacities compare as follows:

\begin{lem} \label{lem:comparisoncap}
$$
T_\omega(K) \le \exp\left[ 1-\frac{1}{\ca_\omega(K)^{1/n}} \right].
$$
\end{lem}

We refer the reader to \cite[Proposition 7.1]{GZ05} for a proof which also provides a reverse inequality that is not needed in the sequel.

\subsection{Proof of Theorem \ref{thm:uniform1}}

\subsubsection{Domination by capacity}

It follows from H\"older inequality and (H2) that
$$
\mu \le C \nu^{1/q},
$$
 where $q$ is the conjugate exponent, $1/p+1/q=1$.

Let $K \subset X$ be a non pluripolar Borel set. Recall that 
$V_{K,\omega}(x) =0$ for $\nu$-almost every point $x \in K$.
The hypothesis (H1)  therefore implies that
$$
\nu(K) \le \int_X e^{-\a \, V_{K,\omega}} \, d\nu
\le A_\a  \, T_\omega(K)^\a.
$$
Combining previous information we obtain
$$
\mu(K) \le C A_\a^{1/q} e^{\a/q} \exp \left[ 
-\frac{\a/q}{\ca_{\omega}(K)^{1/n}} \right]
\le D \, \ca_{\omega}(K)^2,
$$
where
$$
D=b_n^n C A_\a^{1/q} e^{\a/q},
$$
with $b_n$ a numerical constant such that
$\exp(-1/x) \le b_n^n x^{2n}$ for all $x>0$.

\smallskip

We now need to relate the Monge-Amp\`ere capacity of the sublevel sets of a 
$\omega$-psh function to the Monge-Amp\`ere measure of similar sublevel sets:

\begin{lem} \label{lem:capma}
Let $\f$ be a bounded $\omega$-psh function. 
For all $s>0$ and $0 < \delta<1$,
$$
\delta^n \, \ca_{\omega}\left( \{ \f<-s-\delta\} \right)
\le \MA(\f) \left( \{ \f<-s\} \right)
$$
\end{lem}

We refer to \cite[Lemma 2.2]{EGZ} for a proof.

\subsubsection{Bounding the solution from below}
\label{Bounded}

Under our assumptions (H1,H2), it follows from general arguments that there is
a unique bounded $\omega$-psh solution $\f$ of 
$\MA(\f)=\mu$ normalized by $\sup_X \f=0$, cf Remark~\ref{no charge pp}. 
The non-expert reader could even think that $\f$ is smooth:
the point here is to establish a uniform a priori bound from below.

 We let $f:\R^+ \rightarrow \R^+$
denote the function defined by
$$
f(s):=-\frac{1}{n} \log \ca_\omega \left( \{ \f<-s\} \right)
$$
Observe that $f$ is non decreasing and such that $f(+\infty)=+\infty$.
It follows from our previous estimates that for all
$s>0$ and $0 < \delta < 1$,
$$
f(s+\delta) \ge 2 f(s)+\log \delta-\frac{\log D}{n}.
$$
Our next lemma guarantees that such a function reaches $+\infty$ in finite time:

\begin{lem} \label{lem:giorgio}
$f(s)=+\infty$ for all $s \ge 5 D^{1/n}+s_0$, where
$$
s_0=\inf \{ s>0 \, | \, eD^{1/n} \ca_\omega \left( \{ \f<-s\} \right)<1 \}.
$$
\end{lem}

\begin{proof}
We define a sequence $(s_j)$ of positive reals by induction as follows,
$$
 s_{j+1}=s_j+\delta_j \; \; \text{ with }  \; \; \delta_j=eD^{1/n} \exp(-f(s_j)).
$$

We fix $s_0$ large enough (as in the statement of the Lemma) so that $\delta_0<1$.
It is straightforward to check, by induction, that the sequence $(s_j)$ is increasing, while
$(\delta_j)$ is decreasing. Thus $0 < \delta_j <1$ and
$$
f(s_{j+1}) \ge f(s_j)+1 ,
\; \; \text{hence} \; \; 
f(s_j) \ge j.
$$

We infer $\delta_j \le eD^{1/n} \exp(-j)$ and 
$$
s_{\infty}=s_0+\sum_{j \ge 0} (s_{j+1}-s_j) \le s_0+\sum_{j \ge 0} eD^{1/n} \exp(-j) \le s_0+5D^{1/n}.
$$
\end{proof}

It remains to obtain a uniform bound on $s_0$.
It follows from Chebyshev inequality and Lemma \ref{lem:capma} 
(used with $\delta=1$) that for all $s>0$,
$$
\ca_\omega \left( \{ \f<-s-1\} \right) \le \frac{1}{s}   \int_X (-\f) d\mu,
$$
since $\MA(\f)=\mu$. H\"older inequality and $(H2)$ yield
$$
\int_X (-\f) d\mu \le C \left( \int_X (-\f)^q d\nu \right)^{1/q}.
$$
Observe that for all $t \ge 0$,
$$
t^q \le \frac{q !}{\a^q} \exp( \a t)
$$
and use $(H1)$ to conclude that
$$
\ca_\omega \left( \{ \f<-s-1\} \right) \le \frac{C (q !)^{1/q} A_{\a}^{1/q}}{\a s}.
$$
Thus 
$$
s_0=1+eD^{1/n} \frac{C (q !)^{1/q} A_{\a}^{1/q}}{\a}
$$ 
is a convenient choice.
This yields the desired a priori estimate and concludes the proof.

\subsection{More general densities} \label{sec:logdensity}

The setting of Theorem \ref{thm:uniform1} is the most commonly used in geometric applications, as it allows e.g. to construct
K\"ahler-Einstein currents on varieties with log-terminal singularities (see section \ref{sec:calabi}).
For varieties  of general type with semi log-canonical singularities (see section \ref{sec:stable}),
one has to deal with slightly more general densities.
 The following result is a refinement of \cite[Theorem 2.5.2]{Kolo}
 and \cite[Theorem A]{EGZ}.

\begin{thm} \label{thm:uniform3}
Let $(X,\omega_X)$ be a compact K\"ahler manifold of complex dimension $n \in \N^*$
and let $\omega$ be a semi-positive form  with
$
V:={\rm{Vol}}_\omega(X)=\int_X \omega^n >0.
$
Let $\nu$ and $\mu=f \, \nu$ be probability measures, with
$0 \le f \in L^1(\nu)$.
Assume the following  assumptions are satisfied:
\begin{enumerate}
\item[(H1)] there exists $\a>0$ and $A_\a>0$ such that for all $\p \in \PSH(X,\omega)$,
$$
\int_X e^{-\a (\p-\sup_X \p) } d\nu \le A_\a;
$$
\item[(H2')] there exists $C,\e>0$ such that $ \int_X |f| | \log f |^{n+\e} \, d\nu \le C$.
\end{enumerate}

Let $\f$ be the unique $\omega$-psh solution $\f$ to the complex Monge-Amp\`ere equation
$$
{V}^{-1} (\omega+dd^c \f)^n =\mu,
$$
normalized by $\sup_X \f=0$. Then $-M \le \f \le 0$ where
$
M=M(C,\e,n,A_\a).
$
\end{thm}

\begin{proof}
The proof follows the same lines as that of Theorem \ref{thm:uniform1}, so we only emphasize the main technical differences and focus on the case $\ep=1$.
Set, for $t \geq 0$,
$$
\chi(t)=(t+1) \sum_{j=0}^{n+1} (-1)^{n+1-j} \frac{(n+1) !}{j !}( \log (t+1))^j.
$$
Observe  that $\chi$ is a convex function such that $\chi(0)=0$ and $\chi'(t)=(\log (t+1))^{n+1}$. Its Legendre
transform is
$$
\chi^*(s)=\sup_{t>0} \{ s \cdot t -\chi(t) \}=s t(s) -\chi(t(s)),
$$
where $1+t(s)=\exp( s^{\frac{1}{n+1}})$ satisfies $s=\chi'(t(s))$, thus
$$
\chi^*(s)=P(s^{\frac{1}{n+1}}) \exp (s^{\frac{1}{n+1}})-s,
$$
where $P$ is the following polynomial of degree $n$,
$$
P(X)=\sum_{j=0}^{n} (-1)^{n-j} \frac{(n+1) !}{j !}X^j.
$$

\smallskip

We let the reader check that (H2') is equivalent to 
$
||f||_{\chi} \leq C',
$
where $||f||_{\chi}$ denotes the Luxemburg norm of $f$,
$$
||f||_{\chi}:=\inf \left\{ r>0, \; \int_X \chi( |f|/r) d\nu \leq 1 \right\}.
$$
Let $K \subset X$ be a non pluripolar Borel set.
It follows from H\"older-Young inequality \cite[Proposition 2.15]{BBEGZ} that
$$
\mu(K) \leq 2C' ||\1_{K}||_{\chi^*},
$$
where
$
||\1_{K}||_{\chi^*}=\inf \{ r>0, \; \nu(K) \chi^*( 1/r ) \leq 1 \}=r_K,
$
with
$$
\chi^*(1/r_K)=\frac{1}{\nu(K)}.
$$

We are interested in the behavior of this function as $\nu(K)$ approaches zero, i.e. for small values of $r_K$.
Observe that $\chi^*(s) \leq \exp(2 s^{\frac{1}{n+1}})$ for $s \geq 1/r_n$, hence
$$
\nu(K) \leq \delta_n \Longrightarrow 
\mu(K) \leq 2C' r_K \leq \frac{2^{n+2}C'}{(-\log \nu(K))^{n+1}}.
$$

Recall that (H1) and Lemma \ref{lem:comparisoncap} yield
$$
\nu(K) \leq A_{\a} e^{\a} \exp \left( -\frac{\a}{\ca_\omega(K)^{1/n}}\right)
$$
It follows that for $\nu(K) \leq \delta_n$,
$$
\mu(K) \leq C'' \ca_\omega(K)^{1+1/n},
$$
and we can then conclude by reasoning as in Lemma \ref{lem:giorgio}.
This completes the proof when $\e=1$. The proof for arbitrary $\e>0$ is similar, the crucial point being
the domination of $\mu$ by a multiple of $\ca_\omega^{1+\e/n}$, with an exponent 
$1+\e/n>1$.
\end{proof}

\subsection{Big cohomology classes} \label{sec:big}

 We now consider a similar situation where the reference cohomology class
 $\a$ is still big but no longer semi-positive.
 We assume for convenience that the ambient manifold $(X,\omega_X)$ is again compact K\"ahler,
 but one could equally well develop this material when $X$ belongs to the
 Fujiki class (i.e. when $X$ is merely bimeromorphic to a K\"ahler manifold).
 
 By definition $\a$ is big if it contains a {\it K\"ahler current}, i.e.
there is a 
positive current $T \in \a$ and $\e>0$ such that $T \ge \e \omega_X$.
It follows from
\cite{D2}
that one can further assume that $T$ has {\it analytic singularities},
i.e. it can be locally written $T=dd^c u$, with
$$
u=\frac{c}{2}\log \left[ \sum_{j=1}^s |f_j|^2 \right]+v,
$$
where $c>0$, $v$ is smooth and the $f_j$'s are holomorphic functions.

\begin{defi}
We let $\Amp(\alpha)$ denote the {\it ample locus} of $\a$, i.e. the Zariski open subset of all points
$x \in X$ for which there exists a K\"ahler current in $\a$ with analytic singularities which is smooth in
a neighborhood of $x$.
\end{defi}

It follows from the work of Boucksom \cite{DZD} that one can find a single K\"ahler current $T_0$
with analytic singularities in $\a$ such that 
$$
\Amp(\a)=X \setminus \rm{Sing} \, T_0.
$$
 
 \smallskip
 
 We fix $\theta$ a smooth closed differential $(1,1)$-form representing $\a$.
Following Demailly, one defines the following $\theta$-psh function with {\it minimal singularities}:
$$
V_{\theta}:=\sup \{ u \, ; \, u \in \PSH(X,\theta) \text{ and } u \leq 0\}.
$$

\begin{defi}
A $\theta$-psh function $\f$ has minimal singularities if for every other
$\theta$-psh function $u$, there exists $C \in \R$ such that
$u \le \f+C$.
\end{defi}

There are plenty of such functions, which play the role here of bounded functions
when $\a$ is semi-positive. 
Demailly's regularization result  \cite{D2}
insures that $\a$ contains many $\theta$-psh functions which are smooth in  $\Amp(\a)$.
In particular a $\theta$-psh function $\f$ with minimal singularities is locally bounded 
in $\Amp(\a)$. The Monge-Amp\`ere measure $(\theta+dd^c \f)^n$ is thus well defined in $\Amp(\a)$
in the sense of Bedford and Taylor \cite{BT82}. 

\begin{defi}
It follows from the work of Boucksom \cite{Bou02}
that 
$$
\int_{\Amp(\a)} (\theta+dd^c \f)^n=:V_{\a}>0
$$
is independent of $\f$, it is the {\it volume} of the cohomology class $\a$.
\end{defi}

One can therefore  develop a pluripotential theory   in
the Zariski open set $\Amp(\a)$. This was done in \cite{BEGZ},
where the following properties have been established:
\begin{itemize}
\item the class $\PSH(X,\theta)$ enjoys several compactness properties;
\item the operator $\MA(\f)=V_\a^{-1} (\theta+dd^c \f)^n$ is a well defined probability measure on
the set of $\theta$-psh functions with minimal singularities;
\item the extremal functions
$V_{K,\theta}=\sup \{ u \, ; \, u \in \PSH(X,\theta) \text{ and } u \le 0 \text{ on } K \}$ and
the Alexander-Taylor capacity
$
T_{\theta}(K)=\exp \left( -\sup_X V_{K,\theta} \right) 
$
enjoy similar properties as in the semi-positive case;
\item in particular it compares similarly to the Monge-Amp\`ere capacity 
$$
\ca_{\theta}(K):=\sup \left\{ \int_K \MA(u) \, ; \, u \in \PSH(X,\theta) \text{ and } 0 \le u-V_{\theta} \le 1 \right\};
$$
\item the comparison principle holds so Lemma~\ref{lem:capma} holds here as well.
\end{itemize}
 
The same proof as above therefore produces the following uniform a priori estimate,
which is a refinement of \cite[Thm.~4.1]{BEGZ}:

\begin{thm} \label{thm:uniform2}
Let $(X,\omega_X)$ be a compact K\"ahler manifold of complex dimension $n \in \N^*$.
Let $\a$ be a big cohomology class of volume $V_\a>0$ and
fix $\theta$ a smooth closed differential $(1,1)$-form representing $\a$.

Let $\nu$ and $\mu=f \, \nu$ be probability measures, with
$0 \le f \in L^p(\nu)$ for some $p>1$.
Assume the following assumptions are satisfied:
\begin{enumerate}
\item[(H1)] $\exists \a>0,A_\a>0$ such that $\forall \p \in \PSH(X,\theta)$,
$
\int_X e^{-\a (\p-\sup_X \p) } d\nu \le A_\a;
$
\item[(H2)] there exists $C>0$ such that $\left( \int_X |f|^p \, d\nu \right)^{1/p} \le C$.
\end{enumerate}
Let $\f$ be the unique $\theta$-psh function with minimal singularities such that
$$
{V_\a}^{-1} (\theta+dd^c \f)^n =\mu,
$$
and $\sup_X \f=0$. Then $-M \le \f -V_{\theta} \le 0$ where
$$
M=1+ C^{1/n}A_\a^{1/nq} \, e^{\a/nq} b_n \left[5+e\a^{-1} C (q !)^{1/q} A_{\a}^{1/q}  \right] ,
$$
where $b_n$ is a uniform constant such that $\exp(-1/x) \le b_n^n x^{2n}$ for all $x>0$.
\end{thm}

 \begin{rem}
We also have an analogue of Theorem \ref{thm:uniform3} in the big setting. 
 \end{rem}


 \section{Uniform integrability}

  We wish to apply the previous uniform estimates 
when the complex structure of the underlying manifold is moving.  
In this section we pay a special attention to assumption (H1), by generalizing
an integrability result of Skoda-Zeriahi \cite{Sko72,Zer01}.

\subsection{Notations}

In all what follows, given a positive real number $r$, we denote by $\mathbb D_r:=\{z\in \mathbb C; |z|<r\}$ the open disk of radius $r$ in the complex plane. If $r=1$, we simply write $\mathbb D$ for $\mathbb D_1$. 

\begin{set}
\label{set}
Let $\mathcal{X}$ be an irreducible and reduced complex K\"ahler space. We let $\pi:{\mathcal X} \rightarrow \mathbb D$ 
denote a proper, surjective holomorphic map such that each fiber $X_t=\pi^{-1}(t)$ is a $n$-dimensional, reduced, irreducible, compact K\"ahler space, for any $t\in \mathbb D$. 
\end{set}
For later purposes, we pick a covering $\{U_\alpha\}_\alpha$  of $\mathcal{X}$ by open sets admitting an embedding $j_\alpha:U_\alpha \hookrightarrow \mathbb{C}^{N}$ for some $N\ge n+1$. Moreover, we fix a Kähler form $\omega$ on $\mathcal X$. Up to refining the covering, the datum of $\omega$ is equivalent to the datum of Kähler metrics on open neighborhoods of $j_{\alpha}(U_{\alpha}) \subset \mathbb C^N$ that agree on each intersection $U_{\alpha}^{\rm reg} \cap U_{\beta}^{\reg}$. Equivalently, $\om$ is a genuine Kähler metric on $\mathcal X_{\rm reg}$ such that $(j_{\alpha})_*(\omega|_{U_{\alpha}^{\rm reg}})$ is the restriction of a Kähler metric defined on a an open neighborhood of  $j_{\alpha}(U_{\alpha}) \subset \mathbb C^N$. 

Let us point out that this definition of a Kähler metric on a singular space $X$ is much more restrictive than merely asking for a Kähler metric on $X_{\rm reg}$ (even say, by requiring that the latter has local potentials near $X_{\rm sing}$, and that those are bounded). One important property that Kähler metrics satisfy is that their pull back under a modification is a smooth form (i.e. locally the restriction of a smooth form under a local embedding in $\mathbb C^N$); in particular, it is dominated by a Kähler form. 

For each $t\in \mathbb D$, we set
$$
\omega_t:=\omega_{|X_t}.
$$

\noindent
An easy yet important observation is the following.

\begin{lem}
\label{constant}
In the Setting~\ref{set} and using the notation above, the quantity $\int_{X_t} \om_t^n$ is independent of $t\in \mathbb D$. We will denote it by $V$ in the following. 
\end{lem}

\begin{proof}
The function $\mathbb D \ni t \mapsto \int_{X_t} \om_t^n$ coincides with the push-forward current $\pi_*\om^n$ of 
bidimension $(1,1)$. Its distributional differential is zero as $d$ commutes with $\pi_*$ and $\omega$ is closed. 
\end{proof}

We fix a smooth, closed differential $(1,1)$-form $\Theta$ on ${\mathcal X}$ and set $\theta_t=\Theta_{|X_t}$. Up to shrinking $\mathbb D$, one will always assume that there exists a constant $C_{\Theta}>0$
such that
\begin{equation}
\label{Ctheta}
-C_{\Theta} \omega \le \Theta \le C_{\Theta} \omega.
\end{equation}
In particular, one has the inclusion $\psh(X_t, \theta_t)\subseteq \psh(X_t, C_{\Theta}\omega_t)$.
We assume that the cohomology classes $\{\theta_t\} \in H^{1,1}(X_t,\R)$ are psef, i.e. the
sets $\PSH(X_t,\theta_t)$ are non-empty for all $t$. 
The notions of (quasi)-plurisubharmonic functions, positive currents and Monge-Amp\`ere measure are well defined on singular spaces \cite{Dem85}.



\subsection{Uniform integrability index}
Recall from \cite[Déf.~3]{Dem82} that if $T$ is a closed, positive current of bi-dimension $(p,p)$ on a complex space $X$ and if $x\in X$ is a closed point, then the Lelong number of $T$ at $x$ is defined as the limit
\begin{equation}
\label{lelong}
\nu(T, x):=\lim_{r\to 0} {\!}^{\downarrow} \frac{1}{ r^{2p}} \int_{\{\psi <r\}} T \wedge (dd^c \psi)^{p}
\end{equation}
where $\psi:=\sum_{i\in I} |g_i|^2$ and $(g_i)_{i\in I}$ is a (finite) system of generators of the maximal ideal $\mathfrak m_{X,x}\subset \mathcal O_{X,x}$. It is proved in \textit{loc. cit.} that the limit above is a decreasing limit, independent of the choice of the generators. Moreover, one has the formula 
\begin{equation}
\label{formula}
\nu(T, x)= \int_{\{x\}} T \wedge (dd^c \log \psi)^{p}
\end{equation} 
cf \cite[bottom of p.~45]{Dem82}. Finally, if $\varphi$ is a $\theta$-psh function on $X$ for some smooth, closed $(1,1)$-form $\theta$, then the Lelong number of $\varphi$ at a given point $x\in X$ is defined to be the quantity $\nu(\theta+dd^c\varphi, x)$.

\begin{prop} \label{pro:uniformLelong}
In the Setting~\ref{set}, let $\f_t \in \PSH(X_t,\theta_t)$ be a collection of $\theta_t$-psh functions on $X_t$.
Then
$$
\sup_{t \in  \overline{\mathbb{D}}_{1/2}}\sup_{x \in X_t } \nu(\f_t,x) <+\infty.
$$
\end{prop}

\begin{proof}
Let $U'_{\alpha} \Subset U_{\alpha}$ be a relatively compact open subset such that the $U'_{\alpha}$ are still a covering of $\mathcal X$. Up to adding more elements to the initial covering, one can always assume that one can find such a refinement. One picks cut-off functions $\chi_{\alpha}$ such that $\chi_{\alpha}\equiv 1$ on $U'_{\alpha}$ and $\mathrm{Supp}(\chi_{\alpha}) \subset U_{\alpha}$.  
Now, let $x\in\mathcal X$; there exists $\alpha=\alpha(x)$ such that $x\in U'_\alpha$. 
Recall that we have an embedding $j_\alpha: U_\alpha \rightarrow \mathbb{C}^{N}$; 
we set $x':=j_{\alpha}(x)$ and $G_{x'}:\mathbb C^{N}\ni z\mapsto \log(\sum_{i=1}^{N}|z_i - x_i'|^2)$. One can easily check that there exists a constant $A>0$, independent of the point $x$ now ranging in the compact set $ \pi^{-1}(\overline {\mathbb D}_{1/2})$, such that the function 
$$H_x:=\chi_{\alpha} \cdotp j_{\alpha}^* G_{x'}$$
defines an $A\om$-psh function on the whole $\mathcal X$. By the formula \eqref{formula}, one has
\begin{eqnarray*}
\nu(\f_t, x) &=& \int_{\{x\}} (\theta_t+dd^c \f_t)\wedge (dd^c (j_{\alpha}^*G_{x'})|_{X_t})^{n-1}\\
& \le & \int_{U_\a'\cap X_t} (\theta_t+dd^c \f_t)\wedge (dd^c H_x)^{n-1}\\
&\le & \int_{U_\a'\cap X_t} (\theta_t+dd^c \f_t)\wedge (A\omega_t +dd^c H_x)^{n-1}\\
&\le & \int_{ X_t} (C_{\Theta}\omega_t+dd^c \f_t) \wedge (A\omega_t +dd^c H_x)^{n-1}\\
&= &  C_{\Theta} A^{n-1}\cdotp V.
\end{eqnarray*}
The conclusion follows.
\end{proof}

\noindent
 It follows from Skoda's integrability theorem \cite{Sko72} that the Lelong number $\nu(\f_t,x)$ controls the
 local integrability index $\a(\f_t,x)$ of a $\theta_t$-psh function $\f_t$, 
 $$
 \a(\f_t,x):=\sup \left\{ c>0 \, ; \, e^{-c\f_t} \in L_{\rm loc}^2(X_t,x) \right\},
 $$
via
$$
\frac{1}{\nu(\f_t,x)} \le \a(\f_t,x) \le \frac{n}{\nu(\f_t,x)}.
$$
Proposition \ref{pro:uniformLelong} thus yields:

\begin{cor} \label{cor:uniformexp}
In the Setting~\ref{set}, the following quantity
$$
\a(\Theta):=\inf \Big\{ \a(\f_t,x); \; t \in \overline{\mathbb{D}}_{1/2}, \;x \in X_t,\; \f_t \in \PSH(X_t,\theta_t) \Big\}
$$
is positive. 
\end{cor}


\subsection{Skoda's integrability theorem in families: the projective case}
Zeriahi \cite{Zer01} has established a uniform version of Skoda's integrability theorem.
We now  further generalize Zeriahi's result by establishing its  family version.

We first provide a very explicit result in the projective case which does not rely on Corollary~\ref{cor:uniformexp} unlike its general Kähler analogue that will be given later, cf. Theorem~\ref{prop:uniformint}. This should also help the reader in following the somehow tricky computations in the general K\"ahler case.

 \begin{prop} \label{prop:uniformint_proj}
 Let $V\subseteq \mathbb{P}^N$ be a projective variety of complex dimension $n$ and degree $d$. 
 Let $\omega=\omega_{\rm FS}|_V$ and $\varphi\in \psh(V,\om)$ be such that $\sup_V \varphi=0$. Then 
 $$ \int_V e^{-\frac{1}{nd}\varphi} \omega^n\le (4n)^n\cdot d \cdot \exp\left\{-\frac{1}{nd}\int_V \varphi \omega^n\right\}.$$
 \end{prop}

To our knowledge, the inequality given in Proposition~\ref{prop:uniformint_proj} above is new. 

\begin{rem}
\label{projcase}
When $\pi:\mathcal{X}\to \mathbb D$ is a projective family whose fibers have degree $d$  with respect to a given projective embedding, the above result gives the integrability of $e^{-\frac{1}{nd}\varphi_t}$ on $V_t:=\pi^{-1}(t)$. In particular, one gets $\a(\om_{\rm FS}) \ge \frac{1}{2nd}$. 
  \end{rem}

 \begin{proof} 
 Embedding $\mathbb{P}^1$ in $\mathbb{P}^2$ if necessary, we assume without loss of generality that $N \geq 2$.
 We first claim that it is enough to prove the Proposition when $\f$ is \textit{smooth}. Indeed, thanks to \cite[Cor.~C]{CGZ13}, there exists a sequence of smooth functions $\f_n\in \PSH(V,\om_{\rm FS})$ decreasing pointwise to $\f$. Let $\ep_n:=\sup_V \f_n$; by Hartog's theorem, we have $\ep_n\to 0$. If the Proposition holds for smooth functions, we will have $$\int_V e^{-\frac{1}{nd}\varphi_n} \omega^n\le e^{\frac{\ep_n\cdot (d-1)}{nd}} (4n)^n\cdot d \cdot \exp\left\{-\frac{1}{nd}\int_V \varphi_n \omega^n\right\}$$
 Using Fatou Lemma and the monotone convergence theorem, we deduce the expected inequality for $\varphi$. From now on, one assumes that $\f$ is smooth.\\

 The projective logarithmic kernel on $\mathbb{P}^N\times \mathbb{P}^N$ is defined by the following formula
 $$
 G(x,y):= \log\left(\frac{||x\wedge y||}{||x|| \cdot ||y||} \right),\quad  x,y\in\mathbb{P}^N,
 $$
 writing $x,y$ in homogeneous coordinates.
 By \cite[Lem.~4.1]{AAZ18}, for any fixed $y$, $x\mapsto G(x,y)$ is a non positive $\omega_{\rm FS}$-psh function in $\mathbb{P}^N$
 such that $(\omega_{\rm FS}+dd^c_x G(\cdot, y))^N= \delta_y$. 
We set $g=G|_V$ and $g_y=g(\cdot, y)$. By definition, $g_y$ has Lelong number one at $y$. Therefore, it follows from \cite[Cor.~4.8]{Dem85} that $\omega_{g_y}^n:=(\omega+dd^c g(\cdot, y))^n \ge \delta_y$. From Stokes formula (cf Lemma~\ref{stokess} below) it follows that
 \begin{eqnarray*}
 \varphi(y) &\ge & \int_{V} \varphi \omega_{g_y}^n= \int_{V}\vp (\omega+dd^c g_y)\wedge \omega_{g_y}^{n-1}\\
 &=&  \int_{ V} \varphi \omega\wedge \omega_{g_y}^{n-1} +\int_{ V} g_y (\omega+dd^c \varphi)\wedge \omega_{g_y}^{n-1} -\int_{ V} g_y \omega\wedge \omega_{g_y}^{n-1}\\
 &\ge & \int_{ V} \varphi \omega\wedge \omega_{g_y}^{n-1}+ \int_{ V} g_y \omega_{\varphi}\wedge \omega_{g_y}^{n-1} ,
 \end{eqnarray*}
using that $g_y \le 0$.
One obtains similarly  
 \begin{eqnarray*}
\int_{ V} \varphi \omega\wedge \omega_{g_y}^{n-1}  
 &\ge &    \int_{ V} \varphi \omega^2\wedge \omega_{g_y}^{n-2} + \int_{ V} g_y \omega \wedge \omega_\f\wedge \omega_{g_y}^{n-2}\\
 &\ge & \int_{ V} \varphi \omega^2\wedge \omega_{g_y}^{n-2}+ \int_{ V} g_y \omega_{\varphi}\wedge \omega_{g_y}^{n-1} ,
 \end{eqnarray*} 
where the second inequality follows from 
 $$
 \int_{ V} g_y \omega \wedge \omega_\f \wedge \omega_{g_y}^{n-2}
  = \int_{ V} g_y \omega_\varphi \wedge  \omega_{g_y}^{n-1} +\int_V d g_y \wedge d^c g_y \wedge \omega_\varphi \wedge \omega_{g_y}^{n-2} \ge \int_{ V} g_y \omega_\varphi \wedge  \omega_{g_y}^{n-1}.
  $$
 Iterating the process $n$ times we end up with
 $$
 \varphi(y) \ge  \int_{V} \varphi \omega^n +n  \int_{V} g_y \omega_\varphi \wedge  \omega_{g_y}^{n-1}. 
 $$
 Hence
$$
  \int_{V} e^{-\frac{1}{nd}\varphi} \omega^n  \le   
  \exp\left\{-\frac{1}{nd}\int_V \varphi \omega^n\right\} \cdot I
  $$
  where
  $$
  I:=   \int_{y\in V} \exp\left \{-\frac{1}{d}\int_{x\in V} g_y (x)\omega_\varphi(x) \wedge \omega_{g_y}(x)^{n-1} \right\}   \omega(y)^n
$$
The $(n,n)$-form $ \frac1d \cdot {\omega_\varphi \wedge \omega_{g_y}^{n-1}}$ induces a probability measure on $V$ given that 
 $$
  \int_{V}  \omega_\varphi \wedge \omega_{g_y}^{n-1} =   \int_{\mathbb{P}^N}  \omega_\varphi \wedge \omega_{g_y}^{n-1} \wedge [V] =  \{\omega_{\rm FS}\}^n \cdot \{V\}=d.
 $$
From Jensen's inequality, one can then derive
 $$I \le \frac{1}{d}\int_{y\in V} \int_{x\in V}  e^{-g(x,y)} \,\omega_\varphi(x) \wedge (\omega(x)+dd^c_x g(x,y))^{n-1} \wedge \omega(y)^n. $$
 Lemma~\ref{lemma_ineq} $(i)$ yields
 $$ 
 \omega_\varphi(x) \wedge (\omega(x)+dd^c_x g(x,y))^{n-1} \le  e^{-2(n-1)g(x,y)} \omega_\varphi(x)\wedge \omega(x) ^{n-1}.
 $$ 
 Lemma~\ref{lemma_ineq} $(ii)$ below (for $\delta=1/2n$) now yields
 \begin{eqnarray*}
 I &\le & \frac{1}{d} \int_{y\in V} \int_{x\in V}  e^{(-2n+1) g(x,y)} \omega_\varphi(x)\wedge \omega(x) ^{n-1} \wedge \omega(y)^n\\
 &=& \frac{1}{d} \int_{x\in V} \left(\int_{y\in V} \Big[e^{-2(1-\frac{1}{2n}) g(x,y)} \omega(y)\Big]^n \right) \omega_\varphi(x)\wedge \omega(x) ^{n-1}\\
 &\le & {(4n)^n} \int_{x\in V}  \left(\frac{1}{d}\int_{y\in V} (\omega+ dd^c \chi_{\frac{1}{2n}}\circ g_x)^n \right)\omega_\varphi(x)\wedge \omega(x) ^{n-1}\\
 &=& (4n)^n\int_{x\in V} \omega_\varphi(x)\wedge \omega(x) ^{n-1} = (4n)^n\cdot d.
  \end{eqnarray*}
 \end{proof} 
 
 \begin{rem}
The same arguments as above show that for any $\gamma\in (0,2)$
 $$
  \int_V e^{-\frac{\gamma}{nd}\varphi} \omega^n\le C_\gamma \cdot d \exp\left\{-\frac{\gamma}{nd}\int_V \varphi \omega^n\right\},
  $$
 where $C_\gamma>0$ depends on $n$ and $\gamma$. We have fixed $\gamma=1$ in the above proposition to simplify the statement.
 \end{rem}

  \begin{lem}
  \label{lemma_ineq} 
  With the notations of the proof of Proposition~\ref{prop:uniformint_proj} above, we fix a point $y\in V$ and set $g:=g_y$. Moreover, let $\delta\in (0,1)$ be a given number. Then, the following set of inequalities hold as currents on $V$.  
  \begin{itemize}
      \item[$(i)$] $\omega_g \le  e^{-2g} \omega$
               \item[$(ii)$] $\frac \delta 2 e^{-2(1-\delta)g}\omega \le \omega+ dd^c \chi_\delta\circ g$
  \end{itemize}
Here,  $\chi_{\delta}$ is the function defined on $\mathbb R$ by the expression $\chi_\delta (t):= \frac{e^{2\delta t}}{4\delta}$. 
  \end{lem}
  
 
\noindent It is understood here that we take derivatives w.r.t. $x$ and the estimates are uniform both in $x$ and $y$.
\begin{proof}
We proceed in three steps. \\

\noindent
\textit{Step 1. Reduction to a computation on $\mathbb C^N$}.

\noindent 
First of all we observe that the function $g$ as well as the $(1,1)$-currents $\om$ and $\omega_g$ are the restriction to $V$ of a function or $(1,1)$-currents on $\mathbb P^N$. As positivity is preserved by restriction to a subvariety, it is enough to prove the inequalities of currents above on the whole $\mathbb P^N$ where they make sense as well. 

Now, recall that $\mathrm{PU}(N,\mathbb C)$ acts transitively on $\mathbb{P}^N$ by transformations preserving $\omega_{\rm FS}$ and an isometry $u$ sends $G_y$ to $G_{u(y)}$. Therefore it suffices to prove all the inequalities above on $\mathbb P^N$, for the special point $y=[1:0:\cdots :0]$.
 We work in the affine chart $(U_1, z)$ where $U_1:=\{x\in \mathbb{P}^N \,:\, x_1\neq 0\}$ and $z:=(z_j)_j$, $z_j= x_j/x_1$. In these coordinates $\omega_{\rm FS}|_{U_1}= \frac{1}{2} dd^c \log(1+\|z\|^2)$. Note that $U_1$ is dense in $\mathbb P^N$ and both $\om_{\rm FS}, \omega_G$ are smooth on the complement $\mathbb P^N\setminus U_1$; thus it is sufficient to prove the inequalities on $U_1\simeq \mathbb C^N$. 
 
 We actually claim that is is sufficient to prove the inequalities on $U_1\setminus \{y\}$, where all the currents involved are smooth differential forms. This is because neither of the positive currents $ e^{-2G} \omega_{\rm FS}$ and $\omega_{\rm FS}+ dd^c \chi_\delta\circ G$ on $\mathbb P^N$ puts any mass on $\{y\}$. This follows from the integrability of $e^{-2G}$ for the first one (recall that $N\ge 2$) and the boundedness of $\chi_{\delta}\circ G$  for the second one. 
 
 As observed in \cite[Lem.~4.1]{AAZ18}, for $(x,y:=[1:0\cdots:0])\in U_1\times U_1 $ we have
 $$G(x,y)= N(z,0)-\frac{1}{2}\log(1+\|z\|^2)$$ where $z=z(x)$ and  $N(z,0):=\frac{1}{2} \log{\|z\|^2}$.  Thus in $U_1$ we have $e^{-2G}= 1+\frac{1}{\|z\|^2}$ and
 $$
 \omega(x)+dd_x^c G_y(x)= dd_z^c N(z,0)= \frac{1}{2}dd_z^c \log\|z\|^2. 
 $$
 Let us define $\beta:=dd^c \|z\|^2=i\sum_{k=1}^Ndz_k\wedge d\bar z_k$ and let $\alpha_1:=\sum_{k=1}^N \bar z_k dz_k$. \\
 
\noindent
\textit{Step 2. Proof of Item} $(i)$. 

\noindent
Standard computations give  
$$(\om_{\rm FS})_{j\bar k}= \frac{(1+\|z\|^2)\delta_{j\bar k} - \bar z_jz_k}{2(1+\|z\|^2)^2}  \quad \mbox{and} \quad  N_{j\bar{k}} =\frac 12 \cdot \frac{\|z\|^2 \delta_{j\bar{k}} - \bar{z}_j z_k}{\|z\|^4}$$ 
or equivalently 
$$\om_{\rm FS}=\frac 12 \left(\frac 1{1+\|z\|^2}\beta- \frac{1}{(1+\|z\|^2)^2} i\alpha_1\wedge \bar\alpha_1\right)\quad \mbox{and} \quad \om_G= \frac 12 \left(\frac 1{\|z\|^2}\beta- \frac{1}{\|z\|^4}i\alpha_1\wedge \bar \alpha_1\right) $$
The matrix $A(z):=(z_i\bar z_j)_{ij}$ is semipositive with rank at most one and trace $\|z\|^2$. Therefore, if $\lambda, \mu \in \mathbb R$ (they can depend on $z$), the matrix $\lambda\mathrm{Id}+\mu A$ is hermitian with eigenvalues $\lambda$ (with multiplicity $N-1$) and $\lambda + \|z\|^2 \cdot \mu$ (with multiplicity one). In particular, it is semipositive if and only if $\lambda \ge \max(0, -\|z\|^2\cdot \mu)$.

\noindent
The computations above show that the eigenvalues of the $(1,1)$-form $\lambda \beta+\mu i\alpha_1\wedge \bar \alpha_1$ with respect to $\beta$ are $\lambda$ and $\lambda+\|z\|^2\cdot \mu$. Now, if $C$ is some non-negative constant, the $(1,1)$-form $Ce^{-2g}\om_{\rm FS}-\om_G$ can be rewritten as follows
$$\frac{1}{2(1+\|z\|^2)\|z\|^4}\cdot \left[(C-1)\|z\|^2(1+\|z\|^2) \cdot  \beta+ \left[(1+\|z\|^2)-C\|z\|^2\right]\cdot i\alpha_1\wedge \bar \alpha_1\right].$$
The latter form is semipositive if and only if $C\ge 1$.
This proves $(i)$. \\

\noindent
\textit{Step 3. Proof of Item} $(ii)$. 

\noindent
Observe that $\chi_\delta$ is convex increasing with $0 \le \chi'_\delta \le 1/2$ for $t \le 0$.
 Standard computations give $dd^c \chi_{\delta} \circ G = \chi'_\delta \circ G\,  dd^c G +\chi''_\delta \circ G \,dG \wedge d^c G$. 
Next, we have
$$dd^c G=\frac{1}{2\|z\|^2(1+\|z\|^2)}\left[ \beta - \frac{1+2\|z\|^2}{\|z\|^2(1+\|z\|^2)}\cdotp i\alpha_1\wedge \bar \alpha_1\right]$$
with the notation introduced in Step 1. Similarly, one finds
$$ dG\wedge d^c G=\frac{1}{4\|z\|^4(1+\|z\|^2)^2}i\alpha_1\wedge \bar \alpha_1.$$
To lighten notation, we will from now on write $\chi'$ (resp. $\chi''$) to denote $\chi'_{\delta}\circ G$ (resp. $\chi''_{\delta}\circ G$). One has 
$$\omega_{\rm FS}+dd^c \chi_{\delta} \circ G = \frac{1}{2(1+\|z\|^2)}\left[\Big(1+\frac{\chi'}{\|z\|^2}\Big)\beta+ \frac{\frac 12\chi''-\chi'(1+2\|z\|^2)}{\|z\|^4(1+\|z\|^2)}i\alpha_1\wedge \bar \alpha_1\right].$$
As a result, the two eigenvalues $\lambda, \mu$ of $\omega_{\rm FS}+dd^c \chi_{\delta} \circ G$ with respect to $\om_{\rm FS}$ are given by 
$$\lambda = 1+\frac{\chi'}{\|z\|^2}$$
and 
$$
\mu = (1+\|z\|^2)\cdot \left(1+\frac{\chi'}{\|z\|^2}+\frac{\frac12\chi''-\chi'(1+2\|z\|^2)}{\|z\|^2(1+\|z\|^2)}\right)
= (1+\|z\|^2-\chi')+\frac{\chi''}{2\|z\|^2}
$$
Using the definition of $\chi$ and the fact that $e^{-2G}=1+\frac{1}{\|z\|^2}$, one easily sees that $\lambda \ge \frac 12 e^{-2(1-\delta)G}$ and $\mu \ge \frac \delta 2 e^{-2(1-\delta)G}$. The conclusion follows. 
\end{proof}

 \subsection{Skoda's integrability theorem in families: the general case}

In this section, we bypass the projectivity assumption and establish a quite general family version of Skoda's integrability theorem, valid for families of
compact K\"ahler varieties:

 \begin{thm}
  \label{prop:uniformint}
 In Setting~\ref{set}, let us choose a positive number $\alpha \in  (0, \a(\Theta))$, which is possible thanks to Corollary~\ref{cor:uniformexp}. Then, there exist constant $A_\a, C>0$ such that for all $t\in \overline{\mathbb{D}}_{1/2}$ and for all $\f_t\in \psh(X_t, \theta_t)$ with
 $\sup_{X_t} \f_t=0$,
 \begin{equation}\label{SkodaUnif}
 \int_{X_t}e^{-\a \f_t} \omega_t^n \le C\exp\left\{-A_\alpha\int_{X_t} \varphi_t \omega_t^n\right\}.
 \end{equation}
\end{thm}

\begin{proof}
The proof follows the same strategy as in \cite{Zer01}, as presented in \cite[Thm.~2.50]{GZ17}.  
There exists a finite number of trivializing charts $\{U_\tau\}$ of $\mathcal{X}$ such that $\pi^{-1}( \overline{\mathbb{D}}_{1/2})\subset \cup_\tau U_\tau$. The statement will then follow if we prove the bound for the integral on the left-hand side replacing $X_t$ by $X_t\cap U_\tau$. Moreover, w.l.o.g we can assume that we have an immersion $j_\tau : U_\tau\hookrightarrow \mathbb{B}$, where $\mathbb{B}$ is the unit ball in $\mathbb{C}^{N}$. Up to shrinking $U_{\tau}$, one can also assume that there exists a smooth function $\rho$ on $\mathbb B$ such that $\sup_{\mathbb B} \rho =-2$ and $\Theta|_{U_{\tau}}=dd^c j_{\tau}^*\rho$. We define $\rho_t:=(j_{\tau}^*\rho)|_{U_{\tau}\cap X_t}$; this is a potential of $\theta_t|_{U_{\tau}\cap X_t}$. Note that $\psi_t:= \f_t+\rho_t$ is a non-positive psh function in $U_{\tau} \cap X_t$ such that 
\begin{equation}
\label{borneinf}
\varphi_t-2 \ge \psi_t \ge \varphi_t-C_{\tau}
\end{equation}
for some constant $C_\tau>0$ depending only on $U_\tau$. It is also clear that proving \eqref{SkodaUnif} is equivalent to showing that 
\begin{equation}
\label{borneinf2}
\int_{U_{\tau}\cap X_t} e^{-\a \psi_t} \omega_t^n \le C_{\tau} \exp\left\{-A_{\alpha, \tau}\int_{U_\tau \cap X_t} \psi_t \omega_t^n\right\},
\end{equation}
for some constants $C_{\tau}, A_{\a, \tau}$ that do not depend on $t$. 

\begin{claim}
\label{claim} 
It is sufficient to prove \eqref{borneinf2} for \textit{smooth}, non-positive psh functions $\psi_t$ on $U_{\tau}\cap X_t$ such that 
\begin{equation}
\label{spsh}
dd^c \psi_t\ge (j_\tau^*dd^c \|z\|^2)|_{X_t}.
\end{equation}
\end{claim}
\begin{proof}[Proof of Claim~\ref{claim}]
Indeed, as
$$
\int_{U_{\tau}\cap X_t} e^{-\a \psi_t} \omega_t^n\le e^{\alpha} \int_{U_{\tau}\cap X_t} e^{-\alpha (\psi_t+j_{\tau}^*\|z\|^2)} \omega_t^n,
$$
we can replace $\psi_t$ by the function $\psi_t+j_{\tau}^*\|z\|^2$, bounded above by $-1$. Next, thanks to a result of Fornaess-Narasimhan \cite[Thm.~5.5]{FN}, one can write $\psi_t$ as a decreasing limit of non-positive, smooth psh functions on $U_{\tau}\cap X_t$ (up to shrinking $U_{\tau}$ possibly). The combination of the monotone convergence theorem and the integrability of $e^{-\alpha \vp_t}$ on $X_t$ provided by Corollary~\ref{cor:uniformexp} settles the claim. 
\end{proof}

From now on, we assume that $\psi_t$ is smooth, and we work exclusively on $U_{\tau}$ that we view inside the unit ball $\mathbb B$ of $\mathbb C^N$. By abuse of notation, we will denote by  $\mathbb{B}\cap X_t$ the set $ U_{\tau}\cap X_t$. In the same vein, we will identify the coordinate functions $z=(z_1, \ldots, z_N)$ on $\mathbb B\subset \mathbb C^N$ with their pull-back by $j_\tau$ on $U_{\tau}$. \\

Let us pick some number $t\in  \overline{\mathbb{D}}_{1/2}$ and some point $x\in \mathbb{B}\cap X_t$. 
We denote by $\Phi_x$ the automorphism of the unit ball $\mathbb{B}$ that sends $x$ to the origin and consider 
$$
G_x(z):= \log \|\Phi_x(z)\|
$$ 
the pluricomplex Green function of the unit ball $\mathbb{B}$. Recall that $G_x$ is the unique plurisubharmonic function in $\mathbb{B}$ such that $(dd^c G_x)^{N}= \delta_x$ in the weak sense of currents, $G_x\le 0$ and $G_x$ is identically zero on $\partial \mathbb{B}$. Standard computations yield
\begin{equation}\label{Computation}
dd^c G_x\le \frac{ C_0}{ \|\Phi_x(z)\|^{2}} \, dd^c \|z\|^2\qquad  \rm{on} \; \mathbb{B}.
\end{equation}
for some dimensional constant $C_0=C_0(N)>0$.\\
Since $[X_t|_{\mathbb B}]$ is a positive $(N-n,N-n)$-current on $\mathbb B$ and the singular set of the restriction of the Green function $G_x|_{X_t}$ is compact (it is indeed equal to $\{x\}$), the mixed Monge-Amp\`ere measure $ (dd^c G_x)^n \wedge [X_t]$ is well defined \cite[Prop.~3.15]{GZ17} and it has a Dirac mass with coefficient $\ge 1$ at the point $x$. Since $\psi_t\le 0$ we then have 
$$
\psi_t(x)\ge  \int_{\mathbb{B}} \psi_t (dd^c G_x)^n \wedge [X_t] =\int_{\mathbb{B}\cap X_t} \psi_t (dd^c G_x)^n.$$
Now, we have the following result, which is Stokes' formula in a context of isolated singularities. 

\begin{lem}
\label{stokess}
Let $X\subset B_{\mathbb C^N}(0,2)$ be a a proper, $n$-dimensional complex subspace of the ball of radius $2$ in $\mathbb C^N$, center at the origin. Let $u,v,w$ be psh functions on $B_{\mathbb C^N}(0,2)$ with isolated singularities, i.e. they are smooth outside a discrete set of points in $B_{\mathbb C^N}(0,2)$ which we assume does not meet $\partial B_{\mathbb C^N}(0,1)$. Finally, let $\mathbb B:= B_{\mathbb C^N}(0,1) \cap X$. 
Then, we have 
 \begin{equation}
 \label{stokes}
 \int_{\partial \mathbb B }(ud^cv-vd^cu) \wedge (dd^c w)^{n-1} = \int_{\mathbb B } (udd^cv-vdd^c u) \wedge (dd^c w)^{n-1} 
 \end{equation}
\end{lem}

We include a proof for the reader's convenience.

  \begin{proof}[Proof of Lemma~\ref{stokess}]
  By using a (regularized) maximum operation, we can find a family of smooth psh functions $u_\ep$ (resp. $v_\ep, w_\ep$) decreasing to $u$ (resp. $v,w$) and which coincide
with their limit outside a compact set $K_\ep \Subset \mathbb B$ which collapses onto a finite set $S\Subset \mathbb B$.   By the usual Stokes' formula, one has
$$ \int_{\partial \mathbb B }(u_\ep d^cv_\ep-v_\ep d^cu_\ep) \wedge (dd^c w_\ep)^{n-1} = \int_{\mathbb B } (u_\ep dd^cv_\ep-v_\ep dd^c u_\ep) \wedge (dd^c w_\ep)^{n-1} $$
The left-hand side of the formula above is identical to the left-hand side of \eqref{stokes}. To prove that the right-hand side above converges to the right-hand side of \eqref{stokes}, we prove that the current $(u_\ep dd^cv_\ep-v_\ep dd^c u_\ep) \wedge (dd^c w_\ep)^{n-1}$ on $\mathbb B$ converges to  $(udd^cv-vdd^c u) \wedge (dd^c w)^{n-1}$   weakly, globally on $\mathbb B$ and locally smoothly on $\overline{\mathbb B}   \setminus S$. The local smooth convergence outside $S$ is obvious. As for the global weak convergence, it follows from the convergence $u_\ep dd^cv_\ep \wedge (dd^c w_\ep)^{n-1}\rightharpoonup udd^cv \wedge (dd^c w)^{n-1}$ (and its symmetrical version swapping $u$ and $v$), proved by Demailly, cf e.g. \cite[Thm. 2.6 and Rem.~2.10]{Dem85}.
\end{proof}

Applying Lemma~\ref{stokess} 
to $X=X_t$, $u=\psi_t$, $v=w=G_x$ (recall that $G_x|_{\partial \mathbb B}\equiv 0$), we get
$$
\int_{\mathbb{B}\cap X_t} \psi_t \,(dd^c G_x)^n = \underbrace{\int_{\mathbb{B}\cap X_t} G_x \,dd^c \psi_t \wedge (dd^c G_x)^{n-1}}_{=:I_t} + \underbrace{ \int_{\partial\mathbb{B}\cap X_t} \psi_t \, d^c G_x\wedge  (dd^c G_x)^{n-1}}_{=:J_t}
$$
 By Lemma~\ref{lemG}, in order to get a lower bound for $J_t$, it is enough to bound from above the quantity $\int_{\partial\mathbb{B}\cap X_t} (-\psi_t)\, d^c\|z\|^2 \wedge (dd^c \|z\|^2)^{n-1} $. 
 Applying \eqref{stokes} to $u=-\psi_t$, $v=w=\|z\|^2-1$, we find

  \begin{eqnarray*}
  \int_{\partial \mathbb B \cap X_t}  (-\psi_t)\, d^c\|z\|^2 \wedge (dd^c \|z\|^2)^{n-1} &=& 
  \int_{\mathbb B \cap X_t}  (-\psi_t)\,  (dd^c \|z\|^2)^{n}+\\
  && \int_{\mathbb B \cap X_t}(\|z\|^2-1) \, dd^c \psi_t \wedge  (dd^c \|z\|^2)^{n-1} \\
  &\le &\int_{ \mathbb B \cap X_t}  (-\psi_t)\,  (dd^c \|z\|^2)^{n} \\
  &\le & C_1^n \left[\int_{X_t}  (-\varphi_t)\,\omega_t^n+C_{\tau}\cdot V \right],
  \end{eqnarray*}
  where $C_1$ is such that $dd^c \|z\|^2\le C_1 \omega$ on $\mathbb{B}$ and $C_{\tau}$ is given in \eqref{borneinf}.\\
  
We now take care of the most singular term $I_t$. 
 Set $$\gamma_t(x):= \int_{\mathbb B } dd^c \psi_t\wedge (dd^c G_x)^{n-1}\wedge [X_t] $$ 
 so that $\mu:= \gamma_t^{-1} dd^c \psi_t\wedge (dd^c G_x)^{n-1}\wedge [X_t] $ is a probability measure on $\mathbb B$ (depending on $x$). We claim that for any $x\in \mathbb{B}$ there exists a constant $ \nu>0$ independent of $t$ and $x$ such that $1 \le \gamma_t < \nu$. The uniform upper bound follows from the same computations in the proof of Proposition \ref{pro:uniformLelong}. By \eqref{spsh} we can infer that
\begin{eqnarray*}
\int_{\mathbb{B}} dd^c \psi_t\wedge  (dd^c G_x)^{n-1}\wedge [X_t] & \ge & \int_{\mathbb{B}} dd^c \|z\|^2 \wedge  (dd^c G_x)^{n-1}\wedge [X_t] \\
&\ge  & \nu ((dd^c G_x)^{n-1}\wedge [X_t], x)\\
& \ge &\nu([X_t], x) =m(X_t,x)\ge 1
\end{eqnarray*}
In the second inequality 
 we used the fact that 
 $r\rightarrow \frac{1}{r^2}\int_{\mathbb{B}_r} dd^c \|z\|^2 \wedge T $ is decreasing to $\nu(T,x)$ when $r\downarrow 0$ (see \eqref{lelong}). The first equality follows from \eqref{formula} while the second one comes from Thie's theorem.
  Recall that the origin of $\mathbb{B}$ is identified with the point $x$.\\

\noindent
We now use Jensen's formula and \eqref{Computation} to obtain 
\begin{eqnarray*}
\exp (-\alpha I_t (x)) &=& \exp  \left( \int_{z\in \mathbb{B}} -\alpha\gamma_t G_x d\mu \right) \\
&\le & \frac{1}{\gamma_t} \int_{z\in \mathbb{B}} e^{-\alpha \gamma_t G_x}\, dd^c \psi_t \wedge (dd^c G_x)^{n-1}\wedge [X_t]\\
&=  & \frac{1}{\gamma_t} \int_{z\in \mathbb{B}} \frac{ dd^c \psi_t \wedge (dd^c G_x)^{n-1}\wedge [X_t]}{\|\Phi_x(z)\|^{\alpha\gamma_t}}\\
&\le & {C_0} \int_{z\in \mathbb{B}} \frac{dd^c \psi_t \wedge (dd^c \|z\|^2)^{n-1}\wedge [X_t] }{\|\Phi_x(z)\|^{\alpha\nu+2n-2}},
\end{eqnarray*}
where we can assume w.l.o.g. that $\alpha \nu <2$. By Fubini's theorem, we have
{\footnotesize
\begin{eqnarray*}
\int_{x\in \mathbb{B}_{1/2}} e^{-\alpha \psi_t} \omega^n \wedge [X_t] & \le  & \int_{x\in \mathbb{B}_{1/2} } e^{-\alpha(I_t+J_t)} \omega^n \wedge [X_t]\\
&\le & K \cdot  \int_{x\in \mathbb{B}_{1/2}}  e^{-\alpha I_t} \omega^n \wedge [X_t] \\
&\le & C_0\cdot K\cdot \int_{x\in \mathbb{B}_{1/2}} \left(    \int_{z\in \mathbb{B}} \frac{dd^c \psi_t \wedge (dd^c \|z\|^2)^{n-1}\wedge [X_t] }{\|\Phi_x(z)\|^{\alpha\nu+2n-2}}   \right) \omega^n \wedge [X_t] \\
&\le & C_0\cdot K\cdot \  \int_{z\in \mathbb{B}} \left(    \int_{x\in \mathbb{B}_{1/2}} \frac{  (dd^c \|x\|^2)^n \wedge [X_t] }{\|\Phi_x(z)\|^{\alpha\nu+2n-2}}   \right)  dd^c \psi_t \wedge (dd^c \|z\|^2)^{n-1}\wedge [X_t],
\end{eqnarray*}
}

\noindent
where $K:=\exp \{- \alpha\, C_1^n \int_{X_t} \psi_t \,\omega_t^n\}$.\\
Moreover, using the same computation as in the proof of Lemma \ref{lem: computation3} below, one can check that if $\beta:=  \frac{2-\alpha \nu}{2n}>0$, there exists a constant $C_{\beta}>0$ such that the inequality of $(n,n)$-currents below holds on $\mathbb B$ 
\begin{equation}\label{computation2}
 C_{\beta}^{-1} \, (dd^c_x \|\Phi_x(z)\|^{2\beta})^n \le \frac{1}{\|\Phi_x(z)\|^{\alpha\nu+2n-2}}(dd^c \|x\|^2)^n  \le C_{\beta} \, (dd^c_x \|\Phi_x(z)\|^{2\beta})^n
\end{equation}
Fix $z\in \mathbb{B}$ and for any $x\in \mathbb{B}$ let $f_x(z):= \|\Phi_x(z)\|$. We define an extension of $f_x$ to $\mathcal{X}$ by 
$$
F_x(z):=
\begin{cases}
\chi \cdot f_x(z) & {\mbox{if}}\; x\in \mathbb{B}\\

0 & \mbox{else}.
\end{cases}
$$
Here, $\chi $ is a smooth cut-off function such that $\mathrm{Supp} (\chi) \subset \mathbb B$ and $\chi\equiv 1$ on $ \mathbb{B}_{1/2}$.
It is easy to check that $F_x$ is an $A\omega$-psh function on $\mathcal{X}$ for some $A=A_\tau$ big enough (that a priori depends on $U_\tau$ but can be chosen independently of $x\in \mathbb{B}_{1/2}$). Thus
\begin{eqnarray*}
\int_{x\in \mathbb{B}_{1/2}} \frac{1}{\|\Phi_x(z)\|^{\alpha\nu+2n-2}} \,  (dd^c \|x\|^2)^n \wedge [X_t] &\le &
  C_\beta \int_{x\in \mathbb{B}_{1/2}} (dd^c_x \|\Phi_x(z)\|^{2\beta})^n \wedge[X_t]\\
&\le & C_\beta \int_{x\in \mathcal{X} } (A\omega +dd^c_x F_x(z)^{2\beta})^n \wedge [X_t] \\
&\le & C_\beta \cdot  A^n \cdot V := C_2.
\end{eqnarray*}
It then follows that 
$$
\int_{x\in \mathbb{B}_{1/2}} e^{-\alpha \psi_t} \omega^n \wedge [X_t]  \le C_0\cdot C_2 \cdot K \cdot   \int_{z\in \mathbb{B}} dd^c \psi_t \wedge (dd^c \|z\|^2)^{n-1}\wedge [X_t]
\le  C_3 \cdot K,
$$
where $C_3:= C_0  C_2   C_{\Theta}   C_1^{n-1} \cdot V$. The last inequality follows from the fact that on $\mathbb B_t$, we have $dd^c \psi_t \wedge (dd^c \|z\|^2)^{n-1} \le (\theta_t+dd^c\varphi_t) \wedge (C_1 \om)^{n-1}$, and one can dominate the integral of the right-hand side on $\mathbb B_t$ by its integral on $X_t$ and use \eqref{Ctheta}. 
This is the conclusion.
\end{proof}

  \begin{lem}
  \label{lemG}
  With the notations introduced at the beginning of the proof of Theorem~\ref{prop:uniformint}, there exists a constant $C=C(n)>0$ such that for all
   $x\in \mathbb{B}_{1/2}\subset \mathbb C^N $ and $z \in X_t \cap \mathbb S^{2N-1}$,
  \begin{equation}\label{equiv. volume form}
  \frac{1}{C }  d^c \|z\|^2\wedge (dd^c \|z\|^2)^{n-1}\le d^c G_x \wedge (dd^c  G_x)^{n-1} \le C d^c \|z\|^2\wedge (dd^c \|z\|^2)^{n-1}
  \end{equation}

  \end{lem}
  \begin{proof}
One knows that there exists a neighborhood $U$ of $\mathbb{S}^{2N-1} \subset \mathbb C^{N}$ not containing $x$ such that $dd^c \|\Phi_x\|^2$ defines a K\"ahler form $\omega_x$ on $U$. This follows for instance from the fact that $\Phi_x$ can be extended as an holomorphic map to an open neighborhood of the closed ball \--- and that neighborhood can be chosen to be independent of $x\in \mathbb B_{1/2}$.  On $U$, $\omega_x$ is comparable to the euclidean metric on $\mathbb C^{N}$ and therefore, $\omega_x$ and $\omega_{\rm eucl}$ induce uniformly equivalent Riemannian metrics $g_x$ and $g_{\rm eucl}$ on $U\cap X_t$ first, and then as well on the real hypersurface $X_t \cap \mathbb S^{2N-1}$; we denote them by $g'_x$ and $g'_{\rm eucl}$ respectively. In particular their volume forms $dV_{g_x'}, dV_{g'_{\rm eucl}}$ are equivalent too. One has $dV_{g'_{\rm eucl}}= \iota_v dV_{g_{\rm eucl}}$ where $v$ is the restriction to $X_t$ of the unit outward radial vector
  $$\sum_{j=1}^{n+k} \left(z_j \frac{\partial}{\partial z_j}+\bar  z_j \frac{\partial}{\partial \bar z_j}\right).$$ 
  Hence, on $X_t\cap \mathbb S^{2N-1}$ one has
  $$dV_{g'_{\rm eucl}}=\iota_v (dd^c \|z\|^2)^n= 2 \Big(\frac{i}{\pi}\Big)^{n-1} d^c \|z\|^2 \wedge (dd^c \|z\|^2)^{n-1}.$$
  
  In the same way, $dV_{g'_{x}}= \iota_{v_x} dV_{g_x}$ where $v_x$ is the restriction to $X_t$ of the unit outward vector with respect to $dd^c \|\Phi_x\|^2$, hence $v_x=\Phi_x^*v$. 
  Therefore one has on $X_t \cap \mathbb S^{2N-1}$,
  \begin{eqnarray*}
  dV_{g'_{x}}= \iota_{v_x} (dd^c \|\Phi_x\|^2)^n= \Phi_x^*(\iota_v (dd^c \|z\|^2)^n)&=& 2 \Big(\frac{i}{\pi}\Big)^{n-1} d^c \|\Phi_x\|^2 \wedge (dd^c \|\Phi_x\|^2)^{n-1}\\
  &=& 2^{n+1}  \Big(\frac{i}{\pi}\Big)^{n-1} d^c G_x \wedge (dd^c G_x)^{n-1}. 
  \end{eqnarray*}
 given that $G_x= \frac 12 \log  \|\Phi_x\|^2$ vanishes on the sphere and that $d^c \log u \wedge (dd^c \log u)^{n-1}=\frac 1{u^n} d^c u \wedge (dd^c u)^{n-1}$ for any smooth function $u$.
 This shows that the above two volume forms on $X_t \cap \mathbb S^{2N-1}$ are uniformly equivalent on $X_t\cap \mathbb S^{2N-1}$ hence it ends the proof.   \end{proof}

 \begin{lem}\label{lem: computation3}
 Let $\beta>0$ and $\mathbb{B}\subset \mathbb{C}^n$ be the unit ball. Then $\|z\|^{2\beta}$ is psh on $\mathbb{B}$ and there exists a constant $C_{\beta}>0$ (that depends only on $\beta$) such that 
 $$\frac{C_{\beta}^{-1}}{\|z\|^{2(1-\beta)}} \cdotp dd^c \|z\|^{2} \le dd^c \|z\|^{2\beta}  \le \frac{C_{\beta}}{\|z\|^{2(1-\beta)}} \cdotp dd^c \|z\|^{2}. $$ 
 \end{lem}
 
 \begin{proof}
 Let $\chi:\mathbb{R}^+\rightarrow \mathbb{R}^+$ be defined as $\chi(t):= t^\beta$ and $u:=\|z\|^2$. One has
 $$
 dd^c \chi \circ u = \beta  u^{\beta-1}  \left( dd^c u - (1-\beta) u^{-1} du \wedge d^c u  \right).
 $$
Note that 
 $\min \{1,\beta\}\,\cdotp \, dd^c u \le dd^c u - (1-\beta) u^{-1} du \wedge d^c u \le \max\{1,\beta\} \, \cdotp\, dd^c u$.
Observe that he hermitian matrix associated to the $(1,1)$-form $du \wedge d^c u$ is $(\bar z_i z_j)_{i\bar j}$. The latter has rank one and its non-zero eigenvalue coincides with its trace, i.e. $u$. Therefore the eigenvalues of the hermitian matrix $A:=\mathrm I_n-(1-\beta)u^{-1}(\bar z_i z_j)_{i\bar j}$  are $1$ (with multiplicity $n-1$) and $\beta$ (multiplicity $1$). This ends the proof.  
 \end{proof}
 
 
 
 
  
 

\section{Normalization in families} \label{sec:normalization}

Previous section allows us to check hypothesis (H1), as soon as the mean value of sup-normalized $\theta_t$-psh functions is
uniformly controlled.
It is classical that one can 
 compare the supremum and the mean value of $\theta$-psh functions on a fixed compact K\"ahler variety (see \cite[Prop.~8.5]{GZ17}). We conjecture that the following results holds
 
 \begin{conj}  \label{conj}
In the Setting~\ref{set},    there exists a constant $C>0$ such that: 
the inequality
$$
\sup_{X_t} \varphi_t -C\le \frac{1}{V} \int_{X_t} \varphi_t \, \omega{_t}^n \le \sup_{X_t} \varphi_t 
$$
holds for all $t\in \overline{\mathbb{D}}_{1/2}$ and for every function $\varphi_t\in \psh(X_t, \theta_t)$.
\end{conj}

In a preprint version of this paper, we claimed a proof of the conjecture above but a referee, whom we thank, pointed out a gap. In this section, we propose a large class of families for which the conjecture holds. More precisely, let us consider the following

\begin{assumption}
\label{aspt}
In Setting~\ref{set}, we assume additionnally that one of the following conditions is satisfied by the family $\pi: \cX\to \mathbb D$. 
\begin{enumerate}
\item The map $\pi$ is projective. 
\item The map $\pi$ is locally trivial. 
\item The fibers $X_t$ are smooth for $t\neq 0$. 
\item The fibers $X_t$ have isolated singularities for every $t\in \mathbb D$.  
\end{enumerate}
\end{assumption}

Recall that $\pi$ is said to be 
\begin{itemize}
\item {\it projective} if we have a commutative diagram as below $$
\begin{tikzcd}
\mathcal X \arrow[hookrightarrow]{rr}{\iota} \arrow[labels=below left]{rd}{\pi}& & \mathbb P^N\times \mathbb D \arrow{ld}{\mathrm{pr}_2}  \\
 &\mathbb D & 
\end{tikzcd}
$$
\item  {\it locally trivial} if, up to shrinking $\mathbb D$, there exists a euclidean open cover $(U_{\alpha})_{\alpha}$ of $\mathcal X$ and a collection of isomorphisms $$F_{\alpha}:\mathcal X|_{U_{\alpha}}\overset{\simeq}{\longrightarrow} (U_{\alpha}\cap X_0) \times \mathbb D$$ such that the following diagram is commutative
 \begin{equation}
 \label{lt}
\begin{tikzcd}
\mathcal X|_{U_{\alpha}}  \arrow[labels=below left]{rd}{\pi} \arrow{rr}{F_{\alpha}} & & (U_{\alpha}\cap X_0)\times \mathbb D \arrow{ld}{\mathrm{pr}_2}  \\
 &\mathbb D & 
\end{tikzcd}
\end{equation}
For instance, if $\mathcal X$ is smooth and if the map $\pi$ is a holomorphic submersion, then $\pi$ is automatically locally trivial. 
\end{itemize}

The main result in this section is the following

\begin{prop}  \label{mainprop}
In the Setting~\ref{set} and if Assumption~\ref{aspt} is satisfied, then Conjecture~\ref{conj} holds. That is, there exists a constant $C>0$ such that: 
the inequality
$$
\sup_{X_t} \varphi_t -C\le \frac{1}{V} \int_{X_t} \varphi_t \, \omega{_t}^n \le \sup_{X_t} \varphi_t 
$$
holds for all $t\in \overline{\mathbb{D}}_{1/2}$ and for every function $\varphi_t\in \psh(X_t, \theta_t)$.
\end{prop}

We will prove Proposition~\ref{mainprop} in several independent steps. 
\begin{itemize}
\item[$\bullet$] In \S~\ref{loctriv}, we prove the locally trivial case. 
\item[$\bullet$] In \S~\ref{isolated}, we treat the case of isolated singularities. 
\item[$\bullet$] In \S~\!\ref{sec:sobpoin}-\ref{heatk}-\ref{sec:green} we introduce the material (Sobolev and Poincaré inequalities, heat kernels and Green's functions) that we will use in the final section. 
\item[$\bullet$] In \S~\ref{sec:green2}, we establish at the same time the projective case and the case of a smoothing, thereby completing the proof of Proposition~\ref{mainprop}.
\end{itemize}

By combining the above result with Theorem~\ref{prop:uniformint}, we get the following

\begin{thm}
  \label{H1}
 In Setting~\ref{set}, let us choose a positive number $\alpha \in  (0, \a(\Theta))$, which is possible thanks to Corollary~\ref{cor:uniformexp}. If Assumption~\ref{aspt} is satisfied, there exists a constant $ C_{\alpha}>0$ such that for all $t\in \overline{\mathbb{D}}_{1/2}$ and for all $\f_t\in \psh(X_t, \theta_t)$, we have
  \begin{equation*}\label{SkodaUnif2}
 \int_{X_t}e^{-\a (\f_t-\sup_{X_t} \f_t)} \omega_t^n \le C_\alpha.
 \end{equation*}
\end{thm}

\subsection{Irreducibility of the fibers}
The irreducibility of \textit {all} the fibers is a necessary assumption for the left-hand-side inequality in Conjecture~\ref{conj} to hold
 as the following example shows:

\begin{exa} \label{exa:irred}
Consider
$\mathcal X \subset \mathbb{P}^2 \times \mathbb C$ where
$$
\mathcal X:=\{([x:y:z],t)\,;\, \, xy-tz^2=0\}.
$$ 
The variety $\mathcal X$ is smooth and comes equipped with the proper morphism $\pi: \mathcal X \to \mathbb C$
 induced by the second projection $\mathbb{P}^2 \times \mathbb C \to \mathbb C$. 
 Set $X_t=\{[x:y:z]\in \mathbb{P}^2 \,:\, xy= tz^2\}.$ Note that $X_t$ is a smooth conic for $t\neq 0$ while $X_0=\{[x:y:z]\in \mathbb{P}^2 \,:\, xy= 0\}$ is the union of two lines. The quasi-psh function $\varphi$ on $\mathbb{P}^2 $ defined by
$$
\varphi([x:y:z])=\frac 14 \Big(\log(|x|^2+|z|^2)+\log |y|^2\Big)- \frac 12\log (|x|^2+|y|^2+|z|^2)+\frac{\log 2}{2} 
$$
clearly induces a $\omega$-psh function $\Phi$ on $\mathcal{X}$, where $\omega=\omega_{\rm FS}+dd^c |t|^2$,
$$
\Phi ([x:y:z], t)= \varphi([x:y:z]).
 $$
We set
$\varphi_t:=\Phi|_{X_t}$ and $\om_t:=\om|_{X_t}$. 
A simple computation shows that $\sup_{\mathcal{X}} \Phi=0$ and it is attained at points $([x:y:z],t)$ such that $|y|^2= |x|^2+|z|^2$. 
We also find that $\sup_{X_t}\varphi_t=0$ and the supremum is attained on the set
$$ 
S_t:= \left\{[x:1:z]\, :\, |x|=\frac1{2|t|}\cdotp \Big(\sqrt{4|t|^2+1}-1\Big),\,z^2= xt^{-1}\right \}.
$$
As $t\rightarrow 0$, $S_t$ becomes the circle $\mathcal C:=\{[0:1:e^{i\theta}]; \theta \in \mathbb R\}\subset X_0$ . Note also that $X_0= \ell \cup \ell'$, where $\ell:= \{[0:y:z]\}$ and $\ell':=\{ [x:0: z]\}$ and $\mathcal C \subset \ell$. 
The open annulus $U_t:=\{[z^2:t:z]; 1< |z|^2 <2\}\subset X_t$ satisfies
$$\int_{U_t}\om_t \ge  \delta$$ for some $\delta>0$ independent of $t$ as well as $$\varphi_t|_{U_t} \le \frac 12( \log |t| +1)$$
from which it follows that 
$$\lim_{t\to 0} \int_{X_t} \varphi_t \, \om_t = -\infty.$$
\end{exa}


\subsection{The locally trivial case}
\label{loctriv}

In this section, we prove Proposition~\ref{mainprop} under the assumption that $\pi$ is locally trivial; we borrow the notations from Diagram~\eqref{lt}. 

One can reduce the problem to showing that there exists a constant $C>0$ depending only on $\pi$ such that given any sequence of complex numbers $t_k\to 0$ and any functions $\f_k \in \PSH(X_{t_k}, \theta_{t_k})$ such that $\mathrm{sup}_{X_{t_k}} \f_k=0$, one has 
$$\int_{X_{t_k}} \vp_k\om_{t_k}^n \ge -C.$$

By compactness of $\pi^{-1}(\overline {\mathbb D}_{1/2})$, one can assume that $\alpha$ ranges among the finite set $\{1, \ldots, r\}$ and without loss of generality, one can assume that $U_{\alpha+1}\cap U_{\alpha}\neq \emptyset$, for any $\alpha \in \{1, \ldots, r-1\}$. Up to splitting the sequence $(\f_k)$ into (at most) $r$ subsequences, we can assume that for every $k$, $\f_k$ attains its maximum in the same set $U_{\alpha_0}$ for some fixed $\alpha_0\in\{1, \ldots, r\}$. By simplicity, we assume that $\alpha_0=1$. 

Let $G_{\alpha,k}:U_{\alpha}\cap X_0 \to U_{\alpha}\cap X_{t_k}$ be the biholomorphism defined as the inverse of the restriction of $F_{\alpha}$ to $U_\a\cap X_{t_k}$ and let us analyze the sequence of functions $\psi_{\alpha,k}:=G_{\alpha,k}^*\f_k$. As $F_{\alpha}^*(\om_0+i dt\wedge d\bar t)$ is commensurable to $\om$, there exists $C>0$ depending only on $\pi$ such that 
\begin{equation}
\label{comm}
C^{-1}\om_0 \le G_{\alpha,k}^*\omega_{t_k} \le C\om_0.
\end{equation}
In particular, up to increasing $C$, one can assume that $G_{\alpha,k}^*\theta_{t_k} \le C\om_0$. As a result, one has $\psi_{\alpha,k} \in \PSH(U_{\alpha}\cap X_0,C\om_0)$. 

The family $(\psi_{1,k})_k$ is a family of non-positive $C \om_0$-psh functions on the complex space $U_1\cap X_0$ attaining the value zero there, so it is relatively compact for the $L^1_{\rm loc}$ topology, cf e.g.  \cite[Proposition 8.5]{GZ17}. In particular, given any compact subset $U'_1 \Subset U_1$, the integral $\int_{U'_1} \psi_{1,k} \omega_0^n$ admits a lower bound depending only on $U'_1$ but not on $k$. 

Next, the family $(\psi_{2,k})_k$ is a family of non-positive $C\om_0$-psh functions on $U_2\cap X_0$. Therefore, either it converges locally uniformly to $-\infty$ or it is relatively compact on each compact subset. From \eqref{comm}, it follows that the family of automorphisms $H_k:=(G_{2,k}^{-1})|_{U_1\cap U_2 \cap X_{t_k}}\circ G_1|_{U_1\cap U_2\cap X_0}$ of $U_1\cap U_2\cap X_0$ satisfies
$$C^{-1}\om_0 \le H_{k}^*\omega_{0} \le C\om_0 \quad \mbox{and} \quad \psi_{2,k}=H_k^* \psi_{1,k}.$$
One deduces then easily that for any compact subset $U'_{12} \Subset U_1\cap U_2$, the integral $\int_{U'_{12}} \psi_{2,k} \omega_0^n$ admits a lower bound independent of $k$. In turn, this implies that $(\psi_{2,k})_k$ is relatively compact for the $L^1_{\rm loc}$ topology on the whole $U_2\cap X_0$. 

By iterating the argument, one finds that for any $\alpha$, the family $(\psi_{\alpha,k})_k$ is relatively compact for the $L^1_{\rm loc}(U_\a\cap X_0)$ topology and using the estimate \eqref{comm} as above, one concludes easily that $\int_{X_{t_k}} \f_k \om_{t_k}^n$ admits a uniform lower bound as claimed. 

This shows that Proposition~\ref{mainprop} holds whenever $\pi$ is locally trivial. An easy consequence is the following

\begin{cor}
\label{corol}
In Setting~\ref{set}, there exists a discrete set $Z\subset \mathbb D$ such that for every compact subset $K\Subset \mathbb D\setminus Z$, there exists a constant $C_K$ such that 
$$\int_{X_t} \vp_t \om_t^n \ge -C_K$$
for any collection of functions $\vp_t\in \PSH(X_t,\theta_t)$ such that $\sup_{X_t}\vp_t=0$. 

\noindent
Moreover, one can take $Z=\emptyset$ provided that the family $\pi:\mathcal X\to \mathbb D$ admits a simultaneous resolution of singularities, i.e. a proper, surjective holomorphic map $f:\mathcal Y \to \mathcal X$ from a Kähler manifold $\mathcal Y$ such that for any $t\in \mathbb D$, the induced morphism $f|_{Y_t}:Y_t\to X_t$ is a resolution of singularities, where $Y_t:=f^{-1}(X_t)$. 
\end{cor}

\begin{proof}
Let $f:\mathcal Y\to \mathcal X$ be a resolution of singularities of $\mathcal X$. One can assume that $\mathcal Y$ is a Kähler manifold; let us pick $\om_Y$ a Kähler form on $\mathcal Y$. The induced map $\rho:=\pi \circ f: \mathcal Y \to \mathbb D$ is surjective, hence by generic smoothness, it is smooth over the complement of a proper analytic subset $Z$ of $\mathbb D$. In particular, $Z$ is discrete. Note that over $Z$, the fibers of $\rho$ may have several irreducible components.

 We denote by $f_t$ the restriction $f|_{Y_t}:Y_t\to X_t$ of $f$ to the fiber $Y_t$, where $Y_t:=\rho^{-1}(t)$. For any $t\in \mathbb D \setminus Z$, the map $f_t$ is bimeromorphic, i.e. it is a resolution of singularities of $X_t$. Let us choose $K\Subset \mathbb D$ a compact subset. There exists a constant $C_K$ such that $f^*\om \le C_K\om_Y$ on $\rho^{-1}(K)$. In particular, for any $t\in K$, one has $f_t^*\vp_t \in \PSH(Y_t,C_KC_\Theta\, \om_Y)$ and $\sup_{Y_t} f_t^*\vp_t =0$. Now, if we additionally assume that $K\Subset \mathbb D \setminus Z$, we can apply the result above to the \textit{smooth} family $\rho|_{\rho^{-1}(K)}:\rho^{-1}(K)\to K$ to find another constant $C_K'>0$ satisfying 
$$\int_{Y_t} (f_t^*\vp_t) \, \om_Y^n \ge -C_K'$$
for any $t\in K$. As $\om_Y^n \ge C_K^{-n} \,f_t^*\om_t^n$, we deduce that 
$$\int_{X_t} \vp_t \, \om_t^n \ge -C_K'\cdot C_K^{n}$$
which concludes the first part of the proof. The second statement is an immediate consequence of the proof of the first one. Indeed, if $Y_t$ is smooth (as an analytic space), then $\pi \circ f$ is smooth in a neighborhood of $Y_t$ and the argument above can be run over a neighborhood of $t$. 
\end{proof}

\subsection{The case of isolated singularities}
\label{isolated}

In this section, we prove Proposition~\ref{mainprop} in the case where all fibers  $X_t$, $t\in \mathbb D$, have isolated singularities. 

\begin{rem}
\label{rem:isolated}
We would like to start with two observations. 
\begin{itemize}
\item[$\bullet$] This case includes the case where $n=\dim X_t=1$. 
\item[$\bullet$] If one only assumes that $X_0$ has isolated singularities, then it is easy to check that there exists $\ep>0$ such that $X_t$ has isolated singularities for any $t$ satisfying $|t|< \ep$. This is because  the locus $Z\subset \cX$ where $\pi$ is not smooth is an analytic set such that $\dim (Z\cap X_0)=0$ and by upper semi-continuity,  $Z$ has relative dimension $0$ over a neighborhood of $0\in\mathbb D$.
\end{itemize}
\end{rem}

We now proceed to proving Proposition~\ref{mainprop} in several steps. \\

\noindent
\textbf{Step 1.} Localization of the problem at $t=0$. 

\noindent
Let $f:\mathcal Y\to \mathcal X$ be a resolution of singularities $\mathcal X$. The induced family $\pi\circ f: \mathcal Y \to \mathbb D$ is generically smooth over $\mathbb D$ so for $r>0$ small enough, the restriction of $\pi\circ f$ to the inverse image of $\mathbb D_r$ has a most one singular fiber, corresponding to $t=0$. In particular, the family $\mathcal Y\to \mathbb D_r$ is locally trivial away from $Y_0$.  Applying the result in the locally trivial case (cf. \S~\ref{loctriv}) to the collection of $f^*\theta_t$-psh functions $f^*\f_t$, we see that for every compact subset $K\Subset \mathbb D_r^*$, there exists a constant $C_K$ independent of the chosen family such that $$\sup_{t\in K} \int_{X_t} (-\f_t)\om_t^n \le C_K,$$
cf also Corollary~\ref{corol}. This shows that it is enough to prove that for any sequence $t_k\to 0$ and any collection of sup-normalized $\theta_{t_k}$-psh functions $\vp_{t_k}$, one has 
$$\sup_{k\ge 1} \int_{X_{t_k}}(-\vp_{t_k}) \om_{t_k}^n <+\infty.$$

\noindent
\textbf{Step 2.} Choice of a good covering. 

\noindent
As the fibers are reduced, it follows from the jacobian criterion for smoothness that the smooth locus of $\pi$ coincides with the union of the smooth loci of $X_t$ when $t$ ranges in $\mathbb D$. Recall that $Z$, the singular locus of $\pi$,  is an analytic space of relative dimension at most zero. It has finitely many irreducible components (say when restricted to $\pi^{-1}(\mathbb D_{1/2}))$ and we can assume without loss of generality that this number is equal to the cardinality of $Z\cap X_0$.  Let $(V_{\alpha})_{\alpha}$ be a finite collection of (small) open balls in $\mathcal X$ centered at the (finitely many) singular points of $X_0$. Up to adding a finite amount of balls to the collection, one can assume that
\begin{enumerate}
\item[$(i)$] The reunion $V:=\cup_\a V_\a$ is an open neighborhood of $X_0\subset \mathcal X$. 
\item[$(ii)$] Each point of $Z\cap X_0$ belongs to exactly one element $V_\a$ of the covering.
\item[$(iii)$] For all $\alpha$, there exists $\rho_{\alpha}\in \mathcal C^{\infty}(V_{\alpha}, \mathbb R)$ such that $\om|_{V_{\alpha}}=dd^c \rho_\a$. 
\item[$(iv)$] There exists $r>0$ such that for all $\alpha$, one has $$Z \cap \partial V_{\alpha} \cap \pi^{-1}(\mathbb D_{r}) =\emptyset.$$
\end{enumerate}

Up to substracting a constant to $\rho_\a$, one can assume that $\rho_\a$ is non-negative. Moreover, there exists a constant $C_1>0$ such that 
$\rho_{\alpha}\le C_1$ on $V_{\a}$, for any $\alpha$. Let $(\chi_\a)_\a$ be a partition of unity associated to the covering $(V_{\a})_\a$. That means that $\sum_{\a} \chi_{\a}\equiv 1$ and $\mathrm{Supp}(\chi_\a) \subset V_\a$. Finally, let $\rho:=\sum \chi_\a \rho_\a$. 
If follows from $(ii)$ that one has  $\om=dd^c \rho$ in some neighborhood $W'_\a$ of each point of $Z\cap X_0$. We pick a relatively compact open subset $W_\a \Subset W'_\a$ and set $W:=\cup W_\a$. Up to decreasing $r$ a little, one can assume that $Z \cap \partial W \cap \pi^{-1}(\overline {\mathbb D}_{r}) =\emptyset$. In particular, there exists $\delta>0$ such that for any $t\in \overline{\mathbb D}_r$, one has $d_{\om}(\partial W \cap X_t, Z) \ge \delta$. In summary
\begin{equation}
\label{c1}
0\le \rho \le C_1,  \quad \om=dd^c \rho \, \mbox{ on } W, \quad d_{\om}(\partial W\cap X_t, Z) \ge \delta \, \mbox{ for all } t \in  \overline{\mathbb D}_r .
\end{equation} \\

\noindent
\textbf{Step 3.} Weak compactness locally outside $Z$. 

\noindent
Let $t_k$ be a sequence of numbers converging to zero, and let $\vp_{t_k}\in \PSH(X_{t_k}, \theta_{t_k})$ such that $\sup_{X_{t_k}} \vp_{t_k}=0$. We claim that there exists a sequence of points $x_k\in X_{t_k}$ and a constant $C_2>0$ such that 
\begin{enumerate}
\item[$(i)$] $\vp_{t_k}(x_k) \ge -C_2.$
\item[$(ii)$] $d_{\om}(x_k,Z) \ge \delta/2.$
\end{enumerate}
Indeed, let $y_k\in X_{t_k}$ be such that $\vp_{t_k}(y_k)=0$. If $d_{\om}(y_k,Z) \ge \delta/2$, then we are done. Otherwise, it means that we have $y_k\in W$ by the third item of \eqref{c1}. Now, the function $C_{\Theta}\rho+\vp_{t_k}$ is psh on $\overline W$ so by the maximum principle, there exists $x_k\in \partial W$ such that $(C_{\Theta}\rho+\vp_{t_k})(x_k) \ge (C_{\Theta}\rho+\vp_{t_k})(y_k) \ge 0$. By the first item of \eqref{c1}, we deduce $\vp_{t_k}(x_k) \ge -C_2$ where we set $C_2:=C_1C_{\Theta}$. \\

Let $U:=\{x\in \pi^{-1}(\mathbb D_r); d(x,Z) >\delta/2\}$. The map $\pi$ is smooth on $U$ and one can cover $U$ by \textit{finitely} many open subsets $(U_j)_{1\le j\le p}$ isomorphic to $(U_j \cap X_0) \times \mathbb D_r$ over $\mathbb D_r$. 
Because of $(i)$, we can argue as in the locally trivial case (cf. \S~\ref{loctriv}) by exporting the functions $\vp_{t_k}|_{U_j\cap X_{t_k}}$ to the fixed space $U_j\cap X_0$ and get relative compactness there. In particular, one can find a constant $C_3>0$ independent of $k$ such that
 \begin{equation}
 \label{c2}
  \int_{U\cap X_{t_k}}(-\vp_{t_k}) \om_{t_k}^n \le C_3.
  \end{equation}
  
 \noindent
\textbf{Step 4.} The integral bound. 

\noindent
On $W$, one has $\om=dd^c\rho$. This implies that $\om^n=(dd^c \rho)^n+T$ for some smooth, closed $(n,n)$-form $T$ on $\pi^{-1}(\mathbb D_r)$ such that $T|_W\equiv 0$. Let us introduce constants $C_4, C_5$ such that $-C_4\om^{n-1} \le (dd^c \rho)^{n-1} \le C_4\om^{n-1}$ and $T \le C_5 \om^n$. As the complement of $W$ in $\pi^{-1}(\mathbb D_r)$ is included in $U$, it follows from \eqref{c2} that 
\begin{equation}
\label{c3} 
  \int_{X_{t_k}}(-\vp_{t_k}) \,T \le C_5 C_3.
  \end{equation}
  Moreover, one has 
    \begin{align*}
   \int_{X_{t_k}}(-\vp_{t_k})(dd^c \rho)^n &=   \int_{X_{t_k}}-\rho dd^c\vp_{t_k}\wedge (dd^c \rho)^{n-1}\\
   &=- \int_{X_{t_k}}\rho(\theta_{t_k}+ dd^c\vp_{t_k})\wedge (dd^c \rho)^{n-1} + \int_{ X_{t_k}}\rho \theta_{t_k} \wedge (dd^c \rho)^{n-1}\\
   & \le C_4 C_1 \int_{X_{t_k}}(\theta_{t_k}+ dd^c\vp_{t_k})\wedge \om^{n-1}+C_{\Theta} C_4 C_1\cdot V\\
   &\le 2C_1C_4C_{\Theta}\cdot V.
  \end{align*}
  All in all, one finds
  $$   \int_{ X_{t_k}}(-\vp_{t_k}) \om_{t_k}^n \le C_6 $$ 
where $C_6=C_3C_5+2C_1C_4C_{\Theta}\cdot V.$



\subsection{Sobolev and Poincaré inequalities}
\label{sec:sobpoin}

In this section, we work in the Setting~\ref{set} above and we assume from now on that the relative dimension $n=\dim_{\mathbb C} X_t$ satisfies $n>1$, since the case $n=1$ has already been dealt with in \S~\ref{isolated}, cf. Remark~\ref{rem:isolated}.

For $t\in \bD$, we set $X_t:=\pi^{-1}(t)$ and denote by $\xtr$ the regular locus of $X_t$. We fix a Kähler form $\omega$ on $\mathcal X$ and set
$$
\omega_t:=\omega_{|X_t}.
$$

\label{poinbolev}
\begin{prop}
\label{sobolev}
Let $K\Subset \bD$. There exists $C_S=C_S(K)$ such that
$$\forall t\in K, \forall f \in \mathcal C^{\infty}_0(\xtr), \quad \left(\int_{X_t} |f|^{\frac{2n}{n-1}}\om_t^n\right)^\frac{n-1}{n} \le C_S \int_{X_t} (|f|^2+|df|^2_{\om_t}) \, \om_t^n .$$
\end{prop}

\begin{rem}
The inequality above extends immediately to the functions $f\in W^{1,2}(\xtr)$, i.e. such that $f,df\in L^2(\xtr,\om_t)$.
\end{rem}

\begin{proof}
Because of the existence of partition of unity, the statement above is local. That means that it is enough to show the above inequality for any $t\in K$ and any $f\in  \mathcal C^{\infty}_0(U_i\cap \xtr)$ where $U_i\subset \cX$ are open sets such that $\cup U_i=\cX$. 

We fix such an open set $U_i$ and we drop the index $i$ in what follows. Without loss of generality, one can assume that there exists an embedding $U_i\hookrightarrow \mathbb C^N$ and that $\om|_U$ and $\omega_{\mathbb C^N}|_U$ are quasi-isometric. Because Sobolev inequality is essentially insensitive to quasi-isometry, it is enough to show the inequality replacing $\om_t$ by $\omega_{\mathbb C^N}|_{U_t}$ where $U_t:=U\cap X_t$. 

Now, the isometric embeddings $(U_t^{\rm reg}, \omega_{\mathbb C^N}|_{U_t}) \hookrightarrow (\mathbb C^N, \omega_{\mathbb C^N})$ provide a family of minimal submanifolds (i.e. with zero mean curvature vector) of the euclidean space by virtue of Wirtinger inequality. The expected inequality is now a direct application of Michael-Simon's Sobolev inequality \cite[Thm.~2.1]{MS}.
\end{proof}


\begin{prop}
\label{poincare}
Let $K\Subset \bD$. There exists $C_P=C_P(K)$ such that
$$\forall t\in K, \forall f \in \mathcal W^{1,2}_0(\xtr), \quad \int_{X_t} |f|^2\om_t^n \le C_P \int_{X_t} |df|^2_{\om_t} \, \om_t^n .$$
\end{prop}

In the statement above, the space $W^{1,2}_0(\xtr)$ is defined as the space of functions $f$ on $\xtr$ such that $f,df\in  L^2(\xtr,\om_t)$ and $\int_{X_t} f\om_t^n=0$.

\begin{proof}
First, we claim that for each $t\in \bD$, there exists such a Poincaré constant $C_{P,t}$. Indeed, thanks to \cite[Thm.~0.2]{Bei}, the Laplacian $\Delta_{\om_t}$ is positive, self-adjoint and its spectrum is discrete. It remains to show that its kernel is one-dimensional. Now, if $f\in W^{1,2}(\xtr)$ is such that $\Delta_t f=0$, it means that for every $u\in W^{1,2}(\xtr)$, we have $\langle \nabla u, \nabla f\rangle=0$. In particular, taking $u=f$ shows that $f$ is locally constant on $\xtr$. As $X_t$ is irreducible, $\xtr$ is connected and the result follows.  

Given the absolute case explained above, the family version of Poincar\'e inequality
follows from  Proposition~\ref{sobolev} and the irreducibility of the fibers:  we refer the reader
 to \cite[Prop. 3.2]{RZ0} for a detailed argument (the projectivity assumption made by Ruan-Zhang being unnecessary for this part of the argument).
\end{proof}

\subsection{Heat kernels and Green's functions}
\label{heatk}
In this section as well as in the following section~\ref{sec:green}, we go back to the absolute case and consider an \textit{irreducible} and reduced Kähler space $(X,\omega)$ of dimension $n=\dim_{\mathbb C} X$ satisfying $n>1$.
 
 When $X$ is smooth, it is well-known (cf e.g. \cite[\S~VI]{Chavel}) that there exists a smooth, positive function $H:X\times X \times (0,+\infty)$, symmetric in its space variable and such that if $\Delta:=\tr_\om dd^c$, one has
\begin{enumerate}
\item[$\bullet$] $(-\Delta_y+\partial_t)H(x,y,t)=0$.
\item[$\bullet$] For every $x\in X$, one has weak convergence $H(x,\cdot, t)\om^n \underset{t\to 0}{\longrightarrow} \delta_x$.
\end{enumerate} 
In the general case where $X$ may have singularities, one can introduce $X_{\ep}=X\setminus V_{\ep}$ where $V_{\ep}$ is a closed $\ep$-neighborhood of $X_{\rm sing}$ with smooth boundary. Then, there exists a unique smooth, positive function $H_{\ep}$ on $X_\ep \times X_\ep \times  (0,+\infty)$ such that 
\begin{enumerate}
\item[$\bullet$] $(-\Delta_y+\partial_t)H_\ep(x,y,t)=0$.
\item[$\bullet$] $H_{\ep}(x,y,t)\to 0$ whenever $x$ or $y$ approaches $\partial X_\ep$.
\item[$\bullet$] For every $x\in X_\ep$, one has weak convergence $H_\ep(x,\cdot, t)\om^n \underset{t\to 0}{\longrightarrow} \delta_x$.
\end{enumerate}

Moreover, given $(x,y,t)\in X_{\ep_0}\times X_{\ep_0} \times (0,+\infty)$, the function $(0,\ep_0)\ni \ep \mapsto H_{\ep}(x,y,t)$ is decreasing. Using \cite[VIII.2 Thm.~4]{Chavel} and its proof, we additionally see that the limit $H:=\lim_\ep H_\ep$ is everywhere finite and satisfies

\begin{enumerate}
\item[$\bullet$]  $H$ is positive and smooth on $X_\reg\times X_\reg \times (0,+\infty)$.
\item[$\bullet$] $(-\Delta_y+\partial_t)H(x,y,t)=0$.
\item[$\bullet$] For all $x,y\in X_\reg$ and $t,s>0$, one has 
\begin{equation}
\label{doubleint}
H(x,y,t+s)= \int_{X}H(x,\cdot,t) H(\cdot, y,s) \om^n.
\end{equation}
\item[$\bullet$] For any $x\in X_\reg$, one has $H(x,\cdot, t)\om^n \underset{t\to 0}{\longrightarrow} \delta_x$ weakly.
\end{enumerate}


When $X\subset \mathbb P^N$ is {\it projective} and $\om=\om_{\rm FS}|_X$, Li and Tian have showed in \cite{LT95} that there is an absolute inequality
\begin{equation}
\label{ineqLT}
H(x,y,t) \le H_{\mathbb P^n}(d_{\mathbb P^N}(x,y),t)
\end{equation}
for any $x,y\in X_\reg$ and $t\in (0,+\infty)$, where $H_{\mathbb P^n}$ is the heat kernel of $(\mathbb P^n, \omega_{\rm FS})$, whose dependence in the space variables $x,y$ is known to reduce to a single real variable, namely the distance between those two points. 

In particular, $H(x,\cdot, t)$ is \textit{bounded} on $X_{\rm reg}$ for any $x\in X_{\rm reg}$ and $t>0$. Since  $X_{\rm sing}$ has real codimension at least two, it admits cut-off functions whose gradient converges to zero in $L^2$,  and this allows one to perform integration by parts as in the compact case for {\it bounded} functions in $W^{1,2}$. We refer to \cite[Lem~3.1]{LT95} for more details; we will also rely on the latter result which states that $H(x,\cdot, t)\in W^{1,2}$ and that is satisfies the conservation property
$$\forall t>0, \int_X H(x,\cdot, t) \om^n = 1. $$


Below are a few more properties that will be useful later, which are certainly standard in the smooth case. For this purpose, one introduces the function 
$$G(x,y,t):=H(x,y,t)-\frac 1 V$$ 
where $V:=\int_X \om^n$. The key information for us will be given by the fourth item, for which the arguments are borrowed from \cite{ChengLi}, see also \cite[App.~A]{Siu87}.

\begin{lem}
\label{heat}
Assume either that $X$ is smooth or that $X\subset \mathbb P^N$ is projective and $\om=\omega_{\rm FS}|_{X}$. Let $x,y\in X_\reg$. We have
\begin{enumerate}
\item $G(x,y,t)>-\frac 1V$, $\int_X G(x,\cdot, t) \om^n =0$ and $\int_X |G(x,\cdot,t)|\om^n \le 2$. 
\item $|G(x,y,t)|^2 \le G(x,x,t)G(y,y,t)$.
\item $H(x,x,t)\to +\infty$ when $t\to 0$.
\item There exists a constant $C_0$ depending only on the Sobolev and Poincaré constant of $(X_\reg, \om)$ such that 
$$\left| G(x,y,t) \right| \le C_0t^{-n}$$
for any $x,y\in X_\reg$ and any $t>0$. 
\end{enumerate}
\end{lem} 

\begin{proof}
Under the assumptions on $X$, we know that $H(x,\cdot, t)$ is bounded on $X_{\rm reg}$, in $W^{1,2}$ and satisfies the conservation property. We will only rely on these non-quantitative properties to establish the items below, and not on the more precise inequality \eqref{ineqLT} which certainly does not hold if $X$ is not projective. 

$(1)$ a trivial consequence of the positivity of $H$ and the fact that $\int_X H(x,\cdot, t) \om^n = 1$.  

$(2)$ is classical when $X$ is smooth, so we assume for the time being that $X$ is projective.
Let $K_{\ep}$ be the {\it Neumann} heat kernel on $X_\ep$, let $V_\ep:=\int_{X_\ep} \om^n$ and let $\widetilde G_\ep:=K_\ep-\frac 1{V_\ep}$. Then we have
 $$K_\ep(x,y,t)= \sum_{i\ge 0} e^{-\lambda_{i,\ep} t} \phi_{i,\ep}(x)\phi_{i,\ep}(y)$$ 
 where $(\phi_{i,\ep})$ is an orthonormal basis of $L^2(X_\ep)$ consisting of Neumann eigenfunctions of $-\Delta$ with eigenvalues $\lambda_{i,\ep}$. Note that $\phi_{0,\ep}=\frac{1}{\sqrt{V_\ep}}$. 
 By Cauchy-Schwarz, we find that $|\widetilde G_{\ep}(x,y,t)|^2\le \widetilde G_\ep(x,x,t) \cdot  \widetilde G_\ep(y,y,t)$. Thanks to \cite[Lemma~3.2]{LT95}, $K_\ep$ converges to $H$ locally smoothly on $X_{\rm reg}^2\times (0,+\infty)$ when $\ep \to 0$, hence $\widetilde G_{\ep}\to G$ in the same way and we get the second item. 
 
$(3)$ Since $H\ge H_\ep$, It is enough to show the third claim for $H_\ep$. We consider a Sturm-Liouville decomposition as before
  $$H_\ep(x,y,t)= \sum_{i\ge 0} e^{-\mu_{i,\ep} t} \psi_{i,\ep}(x)\psi_{i,\ep}(y)$$ 
 but now, $(\psi_{i,\ep})$ is an orthonormal basis of $L^2(X_\ep)$ consisting of {\it Dirichlet} eigenfunctions of $-\Delta$ with eigenvalues $\mu_{i,\ep}$, cf \cite[VII (31)]{Chavel}. The sought property now follows since $\sum \psi_{i,\ep}(x)^2$ is the norm of the unbounded functional $L^2\cap \mathcal C^{\infty}(X_\ep)\ni f\mapsto f(x)$.

$(4)$ We start from the identity \eqref{doubleint}, which holds for $G$ as well as one checks easily. Taking $y=x$ and differentiating with respect to $s$ and eventually setting $s:=t$, one finds 
$$-G'(x,x,2t) = \|dG(x,\cdot, t)\|^2_{L^2} \ge (C_S(C_P+1))^{-1} \|G(x,\cdot,t)\|^2_{L^{\frac{2n}{n-1}}}$$
since integration by parts is legitimate as we explained above and $\int_X G(x,\cdot,t) \om^n=0$.   Moreover, the interpolation inequality gives
$$G(x,x,2t)=\|G(x,\cdot,t)\|_{L^2}^2 \le \|G(x,\cdot,t)\|_{L^1}^{\frac{2}{n+1}} \cdot \|G(x,\cdot,t)\|_{L^{\frac{2n}{n-1}}}^{\frac{2n}{n+1}}$$
hence 
$$ \|G(x,\cdot,t)\|^2_{L^{\frac{2n}{n-1}}} \ge 2^{-\frac 2{n}}G(x,x,2t)^{\frac{n+1}n} $$
and 
$$-\frac 1n G'(x,x,t) G(x,x,t)^{-1-\frac 1n }\ge C^{-1}$$
for $C=n4^{\frac 1 {n}}\cdot C_S(C_P+1)$. Integrating this inequality w.r.t. $t$ and using the second item, we get the fourth item \--- recall that $G(x,x,t)>0$ for any $x\in X_\reg$ given its expansion as power series, cf $(2)$.
\end{proof}


Under the assumptions of Lemma~\ref{heat} above, the integral
$$G(x,y):=\int_0^{+\infty}G(x,y,t)dt$$
is convergent whenever $x\neq y$ and defines a function $G$ on $X_\reg \times X_\reg$ such that $G(x,\cdot) \in L^1(X_\reg)$. Moreover, since $(-\Delta+\partial_t)G(x,\cdot, t)=0$, $G(x,\cdot, t)\underset{t\to +\infty}{\rightarrow} 0$  and $G(x,\cdot, t)\om^n \underset{t\to 0}{\rightarrow} \delta_x-\frac 1V$, we have 
$$dd^c G(x,\cdot) \wedge \om^n=\frac {\om^n}V-\delta_x, $$ 
i.e. for all $f\in C^{\infty}_0(X_\reg)$, we have
\begin{equation}
\label{greenid}
\int_{X}\Delta f \cdot  G(x,\cdot)\, \om^n= \frac1V \int_X f\, \om^n-f(x).
\end{equation}
Finally, the first and fourth item of Lemma~\ref{heat} enable us to find a lower bound of the Green function as follows
\begin{align}
G(x,y) &= \int_0^1 G(x,y,t)dt + \int_1^{+\infty}G(x,y,t)dt  \label{lowerbound}\\
& \ge -\frac1V-\frac C{n-1}\nonumber
\end{align}
where $C$ only depends on the Sobolev and Poincaré constants of $(X_\reg, \om)$.

\subsection{Green's inequality for general psh functions}
\label{sec:green}

In this section, we assume that the assumptions of Lemma~\ref{heat} are satisfied. 

Let us first generalize Formula~\eqref{greenid} to some functions $f\in \mathcal C^{\infty}(X_\reg)$ that are not necessarily compactly supported. For that purpose, let $p:Y\to X$ a log resolution of singularities, let $D$ be the exceptional divisor of $p$ and let $Y^\circ:=p^{-1}(X_\reg)=Y\setminus D$. We claim that for any  $f\in \mathcal C^{\infty}(X_\reg)$ such that $p^*f$ extends smoothly across $D$, the formula
\begin{equation}
\label{greenid2}
\int_{X_\reg}\Delta f \cdot  G(x,\cdot)\, \om^n= \frac1V \int_{X_\reg} f\, \om^n-f(x)
\end{equation}
holds. First observe that all the terms are well-defined as one sees by pulling back by $p$, which is an isomorphism over $X_\reg$. Indeed, recall that $x\in X_\reg$ and that $G(x,\cdot)$ is locally bounded near $X_{\rm sing}$ so that $p^*G(x,\cdot)$ is in $L^1(Y^\circ,\om_Y)$ for any Kähler form $\om_Y$ on $Y$.

Next, we choose a family $(\chi_\delta)_\delta$ of cut-off functions for $D$. As they are identically $0$ on $D$, they come from $X$ under $p$ and one can see them either as functions on $X$ or $Y$ interchangeably. It is classical (cf e.g. \cite[Sect.~9]{CGP}) that one can choose $\chi_\delta$ such that both $d\chi_\delta\wedge d^c\chi_\delta$ and $\pm dd^c \chi_\delta$ are dominated by some fixed Poincaré metric $\om_P$ (independently of $\delta$). In particular, using Cauchy-Schwarz and the dominated convergence theorem, one finds 
\begin{equation}
\label{error}
\lim_{\delta \to 0} \int_{X_\reg} G(x,\cdot)\Big[ f dd^c \chi_\delta+df\wedge d^c \chi_\delta+d\chi_\delta\wedge d^cf \Big] \wedge \om^{n-1}=0
\end{equation}
by the dominated convergence theorem. Formula \eqref{greenid2} is now a direct application of \eqref{greenid}.

The next result is the key for the proof of Proposition~\ref{mainprop}. 
\begin{claim}
\label{sup-int}
Under the assumptions of Lemma~\ref{heat}, let $\vp\in \PSH(X,\om)$, $V=\int_X\om^n$ and let $x\in X_\reg$. Then, one has
$$\frac1V \int_X \vp \om^n-\vp(x) \ge nV \cdot \inf_{X_\reg}G(x,\cdot).$$
\end{claim}

\begin{proof}
Replacing $\vp$ by $\max(\vp, -j)$ and letting $j\to +\infty$, one sees that it is enough to prove the claim for bounded functions $\vp$. Next, thanks to Demailly's regularization theorem, one can write $p^*\vp$ as a pointwise decreasing limit of smooth function $\psi_\ep$ satisfying $p^*\om+\ep \om_Y+dd^c\psi_\ep \ge 0$ for some fixed Kähler metric $\om_Y$ on $Y$. Using \eqref{greenid2} and setting $G_x:=G(x,\cdot)$, one finds
$$\frac1V \int_X \vp \om^n-\vp(x) = \lim_{\ep \to 0} \int_{Y^\circ} n p^*G_x dd^c \psi_\ep \wedge p^*\om^{n-1}.$$
Moreover, as $G_x$ have zero mean value, one has
{\small
\begin{align*}
\int_{Y^\circ}  p^*G_x dd^c \psi_\ep \wedge p^*\om^{n-1} =&\int_{Y^\circ} ( p^*G_x- \inf_{X_\reg}G_x) dd^c \psi_\ep \wedge p^*\om^{n-1}\\
 =& \int_{Y^\circ} ( p^*G_x- \inf_{X_\reg}G_x) (p^*\om+\ep\om_Y+dd^c \psi_\ep) \wedge p^*\om^{n-1} \\
 &-\int_{Y^\circ} p^*G_x \wedge (p^*\om+\ep\om_Y) \wedge p^*\om^{n-1}\\
& + \inf_{X_\reg}G_x\cdot \Big(V+\ep\int_Y\om_Y\wedge p^*\om^{n-1}  \Big) \\
  \ge&  \inf_{X_\reg}G_x\cdot V+\ep \cdot\Big( \inf_{X_\reg}G_x\cdot \int_Y\om_Y\wedge p^*\om^{n-1} -\int_{Y^\circ} p^*G_x \om_Y \wedge p^*\om^{n-1}\Big)
\end{align*}}
Taking the limit as $\ep\to 0$, we get the expected result.
\end{proof}

\subsection{Proof of Proposition \ref{mainprop}}
\label{sec:green2}

We can now finish the proof of Proposition~\ref{mainprop}. We are left to treating the cases where $\pi$ is projective or $X_t$ is smooth for $t\neq 0$. Moreover, we can assume that $n=\dim X_t \ge 2$ since otherwise, $X_t$ would have at most isolated singularities and we could then appeal to \S~\ref{isolated}, cf. Remark~\ref{rem:isolated}.

Moreover, the content of Proposition~\ref{mainprop} is insensitive to replacing $\om$ by another Kähler metric on $\cX$. In the case where $\pi$ is projective, i.e. if we have $\cX \subset \mathbb P^N \times \mathbb D$ such that $\pi$ commutes with the second projection, then we will assume that $\om = \om_{\rm FS}|_{\cX}$. 

Finally, in the case where $X_t$ is smooth for $t\neq 0$, it is sufficient to prove Proposition~\ref{mainprop} for $t\neq 0$ since it is already well-known that the $L^1$-sup comparison holds on the {\it fixed} irreducible complex space $X_0$.\\

We know from \S~\ref{poinbolev} that the Kähler manifolds $(\xtr, \om_t)$ admit uniform Poincaré and Sobolev constants. As the volume $V$ of $(X_t, \om_t)$ is constant, it follows from \eqref{lowerbound} that there exists $C_G>0$ independent of $t$ such that 
$$\forall x,y\in \xtr, \,\,G_t(x,y) \ge -C_G,$$ 
where $G_t(\cdot, \cdot)$ is the Green function of $(X_t,\om_t)$. As $\vp_t$ is sup-normalized and upper semi-continuous, there exists $x_t\in \xtr$ such that $\f_t(x_t)\ge -1$. Applying Claim~\ref{sup-int} to $\vp:=\vp_t$ and $x:=x_t$, we find 
$$\frac 1V\int_{X_t}  (-\varphi_t) \, \omega_t^n \le nV C_G+1.$$
The Proposition is proved.


%


\section{Densities along a log canonical map} \label{sec:H2}

We now pay attention to hypotheses (H2) and (H2'). We focus in this section on the integrability properties
of some canonical densities.




 
 

\subsection{Semi-stable model}

\begin{set}
\label{fset}
Let $\pi:\cX\to \bD$ be a proper, holomorphic surjective map from a Kähler space $\cX$ with connected fibers to the unit disk of relative dimension $n$. We make the following assumption
\begin{equation}
\label{assumption}
 \mbox{For each } t\in \bD, \mbox {the pair } (\cX,X_t) \mbox{ has log canonical singularities}
 \end{equation}
 where $X_t=\pi^{-1}(t)$ is the schematic fiber at $t\in \bD$, cf \cite[Def.~7.1]{KM}.
\end{set} 

\textit{About the singularities.} In Setting~\ref{fset}, the following properties hold
\begin{enumerate}
\item Every fiber is reduced, $K_{\cX/\bD}$ is $\mathbb Q$-Cartier and $\cX$ has log canonical singularities. 
\item The space $\cX$ has canonical singularities if and only if the general fiber $X_t$ has canonical singularities, cf \cite[Lem.~7.2]{KM}.
\item The condition \eqref{assumption} is preserved by finite base change from a smooth curve, cf \cite[Lem.~7.6]{KM}. 
\item If $(\cX,X_0)$ has lc singularities, then $(\cX, X_t)$ has lc singularities for $|t| \ll 1$, see  \cite[Cor.~4.10 (2)]{Kollar13} and \cite[Thm.~2.3]{Kol18}.
\item By \textit{loc. cit.}, the condition \eqref{assumption} is equivalent to asking $\cX$ to be normal, $\mathbb Q$-Gorenstein, and that each fiber $X_t$ has semi- log canonical singularities. \\
\end{enumerate}


  By \cite{KKMS}, one can find a semi-stable model of $\pi$\footnote{The reference \cite{KKMS} deals with the case of a proper morphism between algebraic varieties but the construction extends to the analytic case \textit{mutatis mutandis}, as stated in e.g. \cite[Thm.~7.17]{KM}}. More precisely, up to shrinking $\bD$, there exists a finite cover $\vp:t\mapsto t^k$ of the disk for some integer $k\ge 1$ and a proper, surjective birational morphism $f:\cX'\to \cX\times_{\varphi}\bD$ 
  \begin{equation}
\label{semistable}
\begin{tikzcd}
\cX'  \arrow[labels=below left]{rd}{\pi'} \arrow{r}{f} & \cX \times_{\varphi}\bD \arrow{r}{g} \arrow{d}{\mathrm{pr}_2} & \cX \arrow{d}{\pi}  \\
 &\mathbb D \arrow{r}{\vp} & \bD 
\end{tikzcd}
\end{equation}
such that $\cX'$ is smooth, $f$ is isomorphic over the smooth locus of $\pi$ and such that around any point $x'\in X_0'$, there exists an integer $p\le n+1$ and a system of coordinates $(z_0, \ldots, z_n)$ centered at $x'$ and such that $\pi'(z_0, \ldots, z_n)=z_0\cdots z_p$. \\

\textit{Additional assumption.} Up to shrinking $\bD$, one will assume that $\pi'$ is smooth away from $0$ so that for any $t\neq 0$, the induced morphism $(g\circ f)|_{X'_t}:X'_t\to X_t$ is a resolution of singularities. Note that $X_t'$ need not be connected.\\

Let $m\ge 1$ be an integer such that $m\Krel$ is a Cartier divisor. 
We can cover $\cX$ with open sets $U_i$ such that $U_i \cap \cX^{\reg}$ admits a nowhere vanishing section $\Omega_{U_i}\in H^0(U_i\cap \cX^{\reg}, m\Krel)$. For any $t\in \bD$, the restriction $\Omega_{U_i}|_{X_t^{\reg}}$ defines a nowhere vanishing section $\Omega_{U_i}|_{X_t^{\reg}}\in H^0(U_i\cap X_t^{\reg}, mK_{X_t})$. In particular, $mK_{X_t}$ is a Cartier divisor for all $t$. We want to understand the behavior of the volume forms $(\Omega_{U_i}\ \wedge \overline{\Omega_{U_i}})|_{X_t^{\rm reg}}^{\frac 1m}$ when $t\to 0$. In order to do so, it is enough to work on $\cX\times_{\vp} \bD$ directly as explained below. \\

\textit{Reduction step.}
The finite map $g$ induces an isomorphism of $\mathbb Q$-line bundles $K_{\cX\times_{\vp} \bD/\bD}\simeq g^*\Krel$. In particular, one can replace $\cX$ by $\cX\times_{\vp} \bD$ in the following, or equivalently assume that $\vp=\mathrm{Id}_{\bD}$; i.e $k=1$. By what was said above, the "new" family still satisfies the condition \eqref{assumption}.\\


\subsection{Analytic expression of the densities in a semi-stable model}

Let us start with some notation. Once and for all, we fix an open set $U:=U_{i_0}$ for some $i_0$. We set $\Omega:=\Omega_U$ and $\Omega_t:=\Omega|_{X_t^{\reg}}$. One can cover $f^{-1}(U)$ by a finite number of open subsets $V_j\subset \cX'$ isomorphic to the unit polydisk of $\mathbb C^{n+1}$ and endowed with a system of coordinates as above. We let $V:=V_{j_0}$ be one of them. The goal is to understand $f^*\Omega$ when restricted to $V$, using our preferred set of coordinates. Finally, we set $U_t:=U\cap X_t$ and $V_t:=V\cap X'_t$.

Next, we write 
\begin{equation}
\label{canformula}
K_{\cX'}+Y_0=f^*(K_{\cX}+X_0)+\sum_{i\in I} a_i  E_i
\end{equation}
where the $E_i$'s are $f$-exceptional divisors with $a_i\ge -1$ for all $i\in I$ and $Y_0$ is the strict transform of $X_0$. Note that some of the divisors $E_i$'s may be irreducible components of $X_0'$. The others surject onto $\bD$ thanks to the additional assumption made in the previous section. The divisor $E:=\sum_{i\in I} E_i$ is the exceptional locus of $f$ and $E+Y_0$ has simple normal crossing support. We set $J:=\{i\in I; a_i=-1\}$ which we assume to be non-empty and define the non-klt locus of $(\cX,X_0)$ by 
\begin{equation}
\label{nklt}
\mathrm{Nklt}(\cX,X_0):=f\big(\bigcup_{ j\in J } E_j\big)
\end{equation}
One can check that $\mathrm{Nklt}(\cX,X_0)$ does indeed not depend on the particular choice of $f$, hence the notation. Under our assumptions, the above analytic set contains the non-klt locus of every fiber $X_t$, $t\in \bD$. This is an easy consequence of the adjunction formula, at least when the $X_t$'s are normal. Given $K\subset I$, we set $E_K=\cap_{i\in K}E_i$ and introduce the invariants
\begin{equation}
\label{nu2}
   \nu(f):=\max\{|K|; K\subset J, E_K\neq \emptyset\}, \quad \mbox{and} \quad \nu:=\inf_f \{\nu(f)\}
\end{equation} 
where the infimum is taken over all the semistable models of $\cX\to \mathbb D$. \\

We now let $x'\in Y_0$ and we assume that the coordinates mentioned above are chosen such that $Y_0=(z_0\cdots z_r=0)$ locally for $0\le r \le p$ being the number of irreducible components of $Y_0$ minus one on that chosen open set.   \\


On $V_t$, $t\neq 0$, the functions $(z_1, \ldots, z_n)$ induce a system of coordinates and the form $f^*\Omega$ on $V$ can be seen as a collection of $m$-th powers of holomorphic $n$-forms 
$$
f^*\Omega_t=g_t(z_1, \ldots, z_n)(dz_1 \wedge \cdots \wedge dz_n)^{\otimes m}
$$
for some holomorphic function $g_t$ on $V_t\setminus E$, with poles of order at most $(-ma_i)_+$ along $E_i\cap X_t$. The form $\Omega \wedge \pi^*\big(\frac{dt}{t}\big)^{\otimes m}$ is trivialisation of $m(K_{\cX}+X_0)$ over $U^{\reg}$. 
The pull-back $f^*(\Omega \wedge \pi^*\big(\frac{dt}{t}\big)^{\otimes m})$ is a well-defined $m$-th power of a $(n+1)$-form on $f^{-1}(U^\reg)$ with logarithmic poles along $Y_0$ that extends meromorphically to $f^{-1}(U)$ with poles of order at most $(-ma_i)_+$ along $E_i$. 
As 
$$
f^*\pi^*\Big(\frac{dt}{t}\Big)=(\pi')^*\Big(\frac{dt}{t}\Big)=\sum_{i=0}^p \frac{dz_i}{z_i}
$$ 
on $V$,  the form $f^*(\Omega \wedge \pi^*\big(\frac{dt}{t}\big)^{\otimes m})$ is equal on that set to
$$(-1)^{mn} (z_1\cdots z_r)^{m} g_{\pi'(z)}(z_1, \ldots, z_n) \Big(\frac{dz_0}{z_0}\wedge \frac{dz_1}{z_1}  \wedge \cdots \wedge \frac{dz_r}{z_r} \wedge dz_{r+1} \wedge \cdots \wedge dz_n\Big)^{\otimes m}$$
so that the function $(V\setminus E\cup Y_0)\ni z \mapsto (-1)^{nm} (z_1\cdots z_r)^{m} g_{\pi'(z)}(z_1, \ldots, z_n)$ extends to a meromorphic function $h$ on $V$, holomorphic along $Y_0$ and with poles of order at most $(-ma_i)_+$ along $E_i$ and satisfying 
\begin{equation}
\label{fibre}
f^*\Omega_t=(-1)^{mn}\frac{h(z)}{(z_1\cdots z_r)^m}(dz_1 \wedge \cdots \wedge dz_n)^{\otimes m}
 \end{equation}
on $V_t$, for $t\neq 0$. When $t=0$, one can also obtain a formula as above for $f^*\Omega_0$ but it requires to first choose a component $Y_0^{(k)}$ of $Y_0$. Let $0\le i \le r$ such that 
$
Y_0^{(k)}\cap V_0=(z_i=0).
$
 On that set (say after removing $E$), one has
\begin{equation}
\label{fibre0}
f^*\Omega_0=(-1)^{i+mn}  \frac{h(z)}{(z_1\cdots \widehat z_i \cdots z_r)^m} \Big( dz_0  \wedge \cdots \wedge {\widehat{dz_i}} \wedge \cdots  \wedge dz_n\Big)^{\otimes m}.
 \end{equation}

%
%
\noindent
Note that if $X_0$ (or equivalenty, $Y_0$) is irreducible, then $r=0$ in the formula above. 
\begin{claim}
\label{canonical}
If $X_0$ has canonical singularities, then $r=0$ and the meromorphic function $V\ni z \mapsto h(z)$ is holomorphic on $V$. 
\end{claim}

\begin{proof}
As $X_0$ is normal, it is irreducible, hence $Y_0$ is smooth and irreducible. In particular, the map $f|_{Y_0}:Y_0\to X_0$ induces a resolution of singularities. 

As $X_0$ has canonical singularities, the pull-back $f^*\Omega_0$ of the form $\Omega_0$ on $X_0^{\reg} \cap U$ extends holomorphically across $Y_0 \cap E$. Given \eqref{fibre0}, it means that $h|{V\cap Y_0}$ extends holomorphically along each $E_i\cap Y_0$. As $h$ is holomorphic on $V$ and does not vanish outside $V_0$, its divisor is an $n$-dimensional variety supported on $V\cap E$, therefore $\mathrm{div}(h)=\sum b_iE_i$ for some integers $b_i$. As $E+Y_0$ is snc, the decomposition $\mathrm{div}(h|_{Y_0})=\sum b_i(E_i\cap Y_0)$ is the decomposition into irreducible components. As $h|_{Y_0}$ is holomorphic along the non-empty set $Y_0\cap E_i$, we have necessarily $b_i\ge 0$ for any $i$. The claim is proved.
\end{proof}

\subsection{Integrability properties of the canonical densities}

\begin{defi}
In Setting~\ref{fset}, let $\omega$ be a Kähler form on $\cX$. We define the function $\gamma$ on $U\cap \cX_{\reg}$ by 
$$(\Omega\wedge \overline \Omega)^{\frac 1m}=e^{-\gamma} \om^n.$$
\end{defi}

%
We want to analyze the integrability properties of $e^{-\gamma}$. Arguing as in the proof of \cite[Thm.~B.1(i)]{RZ} (see also \cite[Lem.~6.4]{EGZ}), it is easy to infer from the normality of $\cX$ that given any small open set $U'\subset U$, there exist bounded holomorphic functions $(f_1, \cdots, f_\ell)$ on $U'$ such that $V(f_1, \ldots, f_\ell)\subset U'_{\rm sing}$ and 
\begin{equation}
\label{gamma}
\gamma|_{U'_{\reg}} = \frac 1 m  \log \sum_i |f_i|^2.
\end{equation}
Let us pick a section $s_E\in H^0(\cX',\cO_{\cX'}(E))$ cutting out the exceptional divisor $E$ and let us choose $|\cdot |$ a smooth hermitian metric on $\cO_{\cX'}(E)$. Given \eqref{gamma}, there exists a constant $A>0$ such that \begin{equation}
\label{fgamma}
f^*\gamma \ge A \log |s_E|^2.
\end{equation}
We have the following

\begin{lem}
\label{canon}
Assume that $X_0$ has canonical singularities
and set $\om_t:=\om|_{X_t}$. Then up to shrinking $\bD$, there exists $p>1$ and a constant $C>0$ such that for any $t\in \bD$, one has
$$\int_{U_t}e^{-p\gamma}\om_t^n \le C.$$
\end{lem}

\begin{proof}

We set  $p:=1+\delta$ for some $\delta>0$ small enough to be chose later. Given \eqref{fgamma}, we have
$$
\int_{U_t}e^{-p\gamma}\om_t^n = \int_{f^{-1}(U_t)}e^{-\delta f^*\gamma}f^*(\Omega_t\wedge \overline \Omega_t)^{\frac 1m} 
 \le  \int_{f^{-1}(U_t)}|s_E|^{-2\delta A} f^*(\Omega_t\wedge \overline \Omega_t)^{\frac 1m} .
$$
Now, one can cover $f^{-1}(U_t)$ by finitely many open sets $V_t=V\cap X'_t$ as above. On $V$, the system of coordinates $(z_0, \ldots, z_n)$ induces a system of coordinates $(z_1, \ldots, z_n)$ such that we have
$$
|s_E|^{-2\delta A} f^*(\Omega_t\wedge \overline \Omega_t)^{\frac 1m} \le C \prod_{i=1}^p |z_i|^{-2\delta A} idz_1\wedge d\bar z_1 \wedge \cdots \wedge idz_n\wedge d\bar z_n
$$
for some uniform constant $C$ thanks to \eqref{fibre} and Claim~\ref{canonical}. 
Recall that $V= \prod_{i=0}^n \{|z_i|<1\} \subset \mathbb C^{n+1}$ and 
$$
V_t=V\cap \{z_0\cdots z_p=t\}\hookrightarrow \{(z_1, \ldots, z_n) \in \mathbb C^n; t \le |z_i| < 1\}\subset \bD^n.
$$
where the injective map is given by $\mathrm{pr}_{z_1, \ldots, z_n}|_{V_t}$, i.e. the restriction to $V_t$ of the projection map onto the last $n$ coordinates in $\mathbb C^{n+1}$. 
For $\delta$ small enough, the function $\mathbb D\ni z\mapsto |z|^{-\delta A}$ is integrable with respect to the area measure; this concludes the proof. 
\end{proof}

For the next lemma, we come back to the general case. We start by choosing a component $Y_0^{(k_0)}$ of $Y_0$, and we denote by $X_0^{(k_0)}$ the irreducible component of $X_0$ birational to  $Y_0^{(k_0)}$ via $f$. Next, we consider the reduced divisor $F$ on $\cX'$ whose support consists of the union of the other components $Y_0^{(k)}$, $k\neq k_0$, along with the divisors $E_i$ whose discrepancy $a_i$ is equal to $-1$, cf \eqref{canformula}. 

 Let $h_F$ be a smooth hermitian metric on $\mathcal O_{\cX'}(F)$ and let $s_F\in H^0(\cX', \mathcal O_{\cX'}(F))$ such that $\mathrm{div}(s_F)=F$. We let 
 \begin{equation}
 \label{psi}
 \psi_F:=-\log(-\log |s_F|_{h_F}^2).
 \end{equation}
  Similarly, let $F_{\rm klt}:=E-F\cap E$, and let $\psi_{\rm klt}:=\log |s_{F_{\rm klt}}|^2$. 
 
\begin{claim}
\label{intcontrol}
There exists $\delta>0$ small enough such that for any $\ep>0$, there exists a constant $C_{\ep}$ such that for any $t\in \bD$,
$$\int_{f^{-1}(U_t)} e^{(\nu(f)+\ep)\psi_F-\delta \psi_{\rm klt}} f^*(\Omega_t \wedge\overline \Omega_t)^{\frac 1m} \le C_{\e},$$
where $\nu(f)$ is defined in \eqref{nu2}.
\end{claim}

\begin{proof}
The statement is local on $\cX'$, so it is enough to control the integrals over $V_t$. Up to relabelling, one can assume that $Y_0^{(k_0)}\cap V=(z_0=0)$, $F\cap V=(z_1\cdots z_s=0)$ so that for $s+1\le i \le p$, $f^*\Omega_t$ has a pole of order at most $m-1$ along $(z_i=0)$. Given the definition of $\nu(f)$, we have 
\begin{equation}
\label{ineq nu}
 \nu(f)\ge s.
\end{equation}
We implicitly assumed that $V$ meets $Y_0^{(k_0)}$; it actually does not matter much for the computation which is insensitive to whether that condition is fulfilled or not. Using \eqref{fibre}, our integral is bounded by the following one
$$\int_{V_t}\frac{1}{\prod_{i=1}^s |z_i|^2}\cdot \frac{1}{(-\log \prod_{i=1}^s|z_i|)^{\nu(f)+\ep}} \cdot \prod_{i=s+1}^p \frac{1}{|z_i|^{2(\delta-a_i)}} \cdot d\lambda_{\mathbb C^n}$$
where $-1<a_i<0$ and $V= \prod_{i=0}^n \{|z_i|<1\} \subset \mathbb C^{n+1}$ and $V_t=V\cap \{z_0\cdots z_p=t\}$. 
By Fubini theorem, one can reduce the integral to $V_t^p:=V_t\cap \mathbb C^{p+1}$ (i.e. fixing $z_{p+1}, \ldots, z_n$). There is no harm in assuming that $\delta<\min_i {\frac{1+a_i}2}$ so that the integral is bounded by
$$\int_{V_t^p}\frac{1}{\prod_{i=1}^s |z_i|^2}\cdot \frac{1}{(-\log \prod_{i=1}^s|z_i|)^{\nu(f)+\ep}} \cdot \prod_{i=s+1}^p \frac{1}{|z_i|^{2(1- \delta/2)}} \cdot d\lambda_{\mathbb C^p}$$
Using polar coordinates, one can assume that $t$ is real (in $(0,1)$) and the integral becomes over $W_t:=\{(r_i)_{1\le i \le p}\in [0,1/2]^p; r_1\ldots r_p\ge t\}$
$$\int_{W_t} \frac{1}{\prod_{i=1}^s r_i \cdot (-\log \prod_{i=1}^s r_i)^{\nu(f)+\ep}} \cdot \prod_{i=s+1}^p \frac{1}{r_i^{1- \delta}} \cdot d\lambda_{\mathbb R^p}.$$
As $W_t\subset \prod_{i=1}^p \{t\le r_i\le 1/2\}$ and the function $r\mapsto\frac {1}{r^{1-\delta}}$ is integrable on $[0,1/2]$, one can appeal to Fubini's theorem to reduce the problem to showing convergence of the integral 
\begin{equation}
\label{integral}
\int_{[0,1]^s} \frac{dr_1\dots dr_s}{\prod_{i=1}^s r_i \cdot (-\log \prod_{i=1}^s r_i)^{\nu(f)+\ep}}
\end{equation}
Now, we inductively use the formula for $\alpha>1$:
\[\int_{[0,1]^2}\frac{dxdy}{xy(-\log xy)^\alpha}=\frac{1}{\alpha-1}\int_0^1 \frac{dx}{x (-\log x)^{\alpha-1}}\]
to see that the integral $\int_{[0,1]^s}\frac{dx_1\dots dx_s}{x_1\dots x_s(-\log x_1\dots x_s)^\alpha}$ converges if and only if $\alpha>s$. The conclusion now follows from the observation above and the inequality \eqref{ineq nu}. 
\end{proof}

The result above allows us to generalize Lemma~\ref{canon} when no assumption on the central fiber is made. To do so, we first need some notation. The function $\psi_F$ is well defined on $\cX'$ but it does not necessarily come from $\cX$. Given that $\mathrm{Nklt}(\cX,X_0)$ is an analytic set in $\cX$ and up to shrinking $\mathbb D$ a little, one can construct a function $\rho$ such that 
\begin{itemize}
\item[$\bullet$] $\rho \le -1$ on $\cX$.
\item[$\bullet$] $\rho$ is quasi-psh and has analytic singularities along  $\mathrm{Nklt}(\cX,X_0)$; in particular, it is identically $-\infty$ on that set.
\end{itemize}
We set 
 $$\psi:=-\log(-\rho )\quad \mbox{on } \cX.$$
  Up to scaling $\rho$, one can assume that 
  \begin{equation}
  \label{fpsi}
  f^*\psi \le \psi_F.
  \end{equation}
  Next, we introduce for $\ep>0$ the function $\gamma_{\ep}:=\gamma-(n+\nu(f)+2\ep)\psi$ defined on $U$. In other words, one has 
\begin{equation}
\label{densite2}
e^{(n+\nu(f)+2\ep)\psi}(\Omega\wedge \overline \Omega)^{\frac 1m}=e^{-\gamma_\ep} \om^n.
\end{equation}

\begin{lem}
\label{corintcontrol}
With the notation above, there exists a constant $\widetilde C_{\ep}$ such that 
$$\int_{U_t} |\gamma_{\ep}|^{n+\ep} e^{-\gamma_\ep}  \om_t^n \le \widetilde C_{\ep}$$
for any $t\in \bD$.
\end{lem}

\begin{proof}
In order to compute the integral, we pull it back by $f$ and work on $V_t$. We have successively 
\begin{align*}
|f^*\gamma_{\ep}|&\lesssim -\log |s_E|+\log(-\log |s_F|)\\
& \lesssim -\log |s_F|- \log |s_{F_{\rm klt}}|.
\end{align*}
The first inequality is a combination of \eqref{fgamma} and \eqref{fpsi}.  To obtain the second inequality, we use the fact that $E=F\cup F_{\rm klt}$ to split the term $\log |s_E|$ while $\log(- \log |s_F|)$ can be absorbed by the more singular $-\log |s_F|$. The integral to bound becomes
$$\int_{V_t}\left[(-\log |s_F|)^{n+\ep}+(- \log |s_{F_{\rm klt}}|)^{n+\ep}\right] e^{(n+\nu(f)+2\ep)\psi_F}f^*(\Omega_t\wedge \overline \Omega_t)^{\frac 1m}$$
which itself is controled by
$$\int_{V_t} e^{(\nu(f)+\ep)\psi_F}f^*(\Omega_t\wedge \overline \Omega_t)^{\frac 1m}+\int_{V_t} e^{(n+1)\psi_F-\delta \psi_{\rm klt}}f^*(\Omega_t\wedge \overline \Omega_t)^{\frac 1m}$$
for any given $\delta>0$. The lemma now follows from Claim~\ref{intcontrol}.
\end{proof}

 \section{Negative curvature} \label{sec:negative}
 
 In this section we apply our previous results to the study of 
 families of varieties with "negative canonical bundle": we consider families of
 manifolds of general type, as well as families of "stable varieties".
 
  
   \subsection{Families of  manifolds of general type}
   \begin{set}
   \label{gt}
Let $\mathcal X$ be an irreducible and reduced complex space endowed with a Kähler form $\omega$ and a proper, holomorphic map $\pi:{\mathcal X} \rightarrow \mathbb D$. We assume that for each $t\in \mathbb D$, the (schematic) fiber $X_t$ is a $n$-dimensional K\"ahler manifold $X_t$ of general type, i.e. such that its canonical bundle $K_{X_t}$ is big. 
In particular, $\mathcal X$ is automatically non-singular and the map $\pi$ is smooth. One can view the fibers $X_t$ as deformations of $X_0$. 
\end{set}

We fix $\Theta$ a closed differential $(1,1)$-form on ${\mathcal X}$ 
which represents $c_1(K_{\mathcal X/\mathbb D})\in H^{1,1}_{\partial \bar \partial}(\cX)$
and set $\theta_t=\Theta_{|X_t}$.
 Shrinking $\mathbb D$ if necessary and rescaling,
we can  assume without loss of generality that  
$$
-\omega \le \Theta \le  \omega.
$$

\begin{lem}
\label{volume}
In the Setting~\ref{gt}, the quantity $ \vol(K_{X_t})$ is independent of $t\in \mathbb D$.
\end{lem}

\begin{proof}
We work in two steps. First, we assume that the family $\pi :\mathcal X \to \mathbb D$ is projective, i.e. there exists a positive line bundle $L$ over $\mathcal X$. In that case,  we know that the invariance of plurigenera holds \cite{Siu98,Paun07} in that the function $t\mapsto h^0(X_t, mK_{X_t})$ is constant on $\mathbb D$, without even assuming that $X_t$ is of general type for all $t$. In particular, it would even be enough to assume that only $X_0$ is of general type from which it results that $X_t$ is of general type for all $t$ and that the volume $\vol(K_{X_t})$ is independent of $t$.

Coming back to the general case, we know that $K_{ \mathcal X/\mathbb D}$ is big. Thanks to Demailly's regularization theorem, there exists a Kähler current $T\in c_1(K_{\mathcal X/\mathbb D})$ with analytic singularities along $V(\mathcal I)$ for some ideal sheaf $\mathcal I \subset \mathcal O_{\mathcal X}$. Let $f:\mathcal X'\to \mathcal X$ be a log resolution of  $(\mathcal X, \mathcal I)$. By Hironaka's theorem, we know that one can construct such a morphism $f$ by a sequence of blow-ups along smooth centers only. We write $f^*T=T'+[F]$ for some smooth semipositive form $T'$ on $\mathcal X'$ and some effective divisor $F$. Remark that this sequence may be infinite; however, the centers project onto a locally finite family of subsets of $\mathcal X$. Up to co-restricting $f$ to $\pi^{-1}(K)$ for some compact subset $K\Subset \mathbb D$, one can assume that $f$ is a finite composition of blow-ups and that $T'\ge \delta \pi^*\om$ for some $\delta>0$ small enough.  

Let $E$ be the exceptional divisor of $f$, with irreducible components $E = \sum_{k =1}^N E_k$. A classical argument (cf e.g. \cite[Lem.~3.5]{DP}) allows one to find smooth $(1,1)$-forms $\theta_{E_k} \in c_1(E_k)$ with support in an arbitrarily small neighborhood of $E_k$ along with positive numbers $(a_k)$ such that the sum  $\theta = \sum_k a_k \theta_k$ defines a $(1, 1)$-form on $\mathcal X'$ which is negative definite along the fibers of $f$. 
It follows that for $\ep>0$ small enough, the smooth form $\pi^*\om-\ep \theta_E$ is Kähler. In particular, $T'-\delta\ep\theta$ is a Kähler form whose cohomology class belongs to $\mathrm{NS}_{\mathbb R}(\mathcal X')$. This implies that the Kähler cone of $\mathcal X'$ meets $\mathrm{NS}_{\mathbb Z}(\mathcal X')$, i.e. $\pi \circ f$ is projective. 

Let $X_t':=f^{-1}(X_t)$ and let $K^\circ \subset K$ be the set of regular value of $\pi \circ f$. For any $t\in K^\circ$, the map $f|_{X_t'}:X_t'\to X_t$ is birational hence $\vol(K_{X'_t})=\vol(K_{X_t})$. By the first step, the volume $\vol(K_{X'_t})$ is independent of $t\in K^\circ$, hence the same holds for $\vol(K_{X_t})$. The set $K\setminus K^\circ$ is finite and without loss of generality, one can assume that it consists of the single element $\{0\}$. The fiber $X_0'$ can be decomposed as $X_0'=Y_0+\sum E_i$ where $f|_{Y_0}:Y_0\to X_0$ is birational and $E_i$ is contracted by $f|_{X'_0}$. Let $Y_0'\to Y_0$ be a resolution of singularities. By \cite[Thm.~1.2]{Taka07}, we have $\vol(K_{Y'_0}) \le \vol(K_{X_t'})$ for $t\neq 0$. As $X_0$ and $Y_0'$ are smooth and birational, we have $\vol(K_{X_0})=\vol(K_{Y'_0}) \le \vol(K_{X_t})$. Finally, as the function $t\mapsto \vol(K_{X_t})$ is upper semi-continous, we have $\vol(K_{X_0})= \vol(K_{X_t})$ for any $t\in K$. The lemma is proved.
\end{proof}

\begin{rem}
In the last step of the proof of Lemma~\ref{volume}, we could also use the existence of relative minimal models, provided $\mathbb D$ is replaced by a quasi-projective smooth curve $C$. The general fiber of the projective morphism $  \mathcal X'\to C$ is a projective variety of general type, hence it admits a good minimal model over $\mathbb C$ by \cite{BCHM}. By \cite[Thm.~3.3]{Fuj16} and \cite[Cor.~1.2]{Taka19}, it follows that $  \mathcal X'\to C$ admits a birational model $\phi:\mathcal X' \dasharrow \mathcal X''$ over $C$ such that: $\phi^{-1}$ does not contract any divisor, every fiber $X_t''$ of $\mathcal X''\to C$ has canonical singularities and satisfies that $K_{\mathcal X_t''}$ is semiample and big. For any $t\in C$, one has $\vol(K_{X_t''})=(K_{X_t''}^n)$. By flatness, this quantity does not depend on $t$. 
Finally, we claim that $X_0''$ is birational to $X_0$.
This is a combination of the following two facts. First, the variety $X_0''$ has canonical singularities and $K_{X_0''}$ is big hence it is of general type and, in particular, it is not uniruled. Next, $X_0''$ is birational to a component of $X_0'$ and all of them but the strict transform of $X_0$ by $f$ are covered by rational curves as $f$ is a composition of blow-ups of smooth centers from a smooth manifold. 
\end{rem}

%
%


%

The positive $(n,n)$-forms $(\omega_t^n)_{t\in \mathbb D}$ induce a smooth hermitian metric on $-K_{\mathcal X/\mathbb D}$. 
Since $[\Theta] = c_1(K_{\cX/\mathbb D}) \in H^{1,1}_{\partial \bar \partial}(\cX)$; there exists a smooth function $\widetilde h$ on $\mathcal X$ such that
$$-dd^c_{\mathcal X} \log \om_t^n = - \Theta+dd^c_{\mathcal X} \widetilde h$$
We will denote by $\widetilde h_t:=\widetilde{h}|_{X_t}$ the restriction to the fiber $X_t$.  
The function $\widetilde h$ becomes unique (and remains smooth) after one imposes the following normalization
\begin{equation*}
\int_{X_t} \widetilde h_t \omega_t^n=0.
\end{equation*}
We define a function $h$ on $\mathcal X$ by imposing that $h_t:=h|_{X_t}$ satisfies
$$h_t=\widetilde h_t- \log \Big(\frac 1{V_t} \int_{X_t}e^{\widetilde h_t}\om_t^n\Big).$$
In particular, one has 
\begin{equation}
\label{normht}
\int_{X_t} e^{ h_t} \omega_t^n=V_t:=\vol(K_{X_t}).
\end{equation}
As $\widetilde h$ is smooth on $\mathcal X$, one has the following obvious consequence. 

\begin{lem}
\label{hbound}
Given any compact subset $K\Subset \mathbb D$, one has 
$$\sup_{t\in K} \| h_t\|_{L^{\infty}(X_t)}<+\infty.$$
\end{lem}

It follows from \cite{BEGZ}, a generalization of the Aubin-Yau theorem \cite{Aubin,Yau78}, that there exists
a unique K\"ahler-Einstein current on $X_t$. This is a positive closed current $T_t$ in $c_1(K_{X_t})$ which, by
 \cite{EGZ,BCHM}, is a smooth K\"ahler form in the ample locus $\Amp(K_{X_t})$, where it satisfies
 the K\"ahler-Einstein equation
 $$
 \Ric(T_t)=-T_t.
 $$
 
  It can be written
$T_t=\theta_t+dd^c \f_t$, where $\f_t$ is the unique $\theta_t$-psh function with minimal singularities
that satisfies the complex Monge-Amp\`ere equation
$$
(\theta_t+dd^c \f_t)^n=e^{\f_t+h_t}  \omega_t^n  \quad \mbox{on } \Amp(K_{X_t}).
$$
The minimal singularity assertion is equivalent to the following uniform bound:
for all $x \in X_t$,
$$
-M_t \le (\f_t(x)-\sup_{X_t} \f_t)-V_{\theta_t}(x) \le 0,
$$
where 
$$
V_{\theta_t}(x)=\sup \{ u_t(x) ; \; u_t \in \PSH(X_t,\theta_t) \text{ and } u_t \le 0 \}.
$$
We can choose $M_t$ independent of $t$ by using Theorem \ref{thm:uniform2}:

\begin{thm}
\label{thmgt}
In Setting~\ref{gt}, let $K\Subset \mathbb D$ be a compact subset. There exists a constant $M_K$ such that for all $x \in \pi^{-1}(K)$, one has
$$
-M_K \le \f_t(x)-V_{\theta_t}(x) \le M_K
$$
where $t=\pi(x)$. 
\end{thm}

\begin{proof}
From Lemma~\ref{volume}, it follows that the volume $V_t$ of $K_{X_t}$ is independent of $t$. We denote it by $V$. 

\noindent
Set $\mu_t=e^{h_t} \omega_t^n/V$ and recall that this is a probability measure, by our choice of normalization.
We first observe that 
\begin{equation}
\label{bornesup}
0 \le \sup_{X_t} \f_t \le -\inf_{\pi^{-1}(K)} h \le C_K.
\end{equation}

Let us first prove the left-hand side inequality. As the measures
$$
\frac{1}{V} (\theta_t+dd^c \f_t)^n=e^{\f_t} \mu_t
$$
have mass one, one has
$$
1 \le \int_{X_t} e^{\sup_{X_t} \f_t} d \mu_t =e^{\sup_{X_t} \f_t}
$$
and therefore, $ \sup_{X_t} \f_t \ge 0.$

To prove the inequality in the middle in \eqref{bornesup}, we observe that, since $\theta_t\le \omega_t$, $\f_t$ is a subsolution of the equation
$$
 (\omega_t+dd^c \f_t)^n \ge  (\theta_t+dd^c \f_t)^n=e^{\f_t+h_t}\omega_t^n,
$$
while the constant function $u_t(x)=-\inf_{\pi^{-1}(K)} h$ is a supersolution of the same equation,
$$
(\omega_t+dd^c u_t)^n =    \omega_t^n \le 
 e^{u_t+h_t} \omega_t^n.
$$
It follows  from the comparison principle \cite[Prop.~10.6]{GZ17} that $\f_t \le -\inf_{\pi^{-1}(K)} h$. The rightmost inequality in \eqref{bornesup} follows from Lemma~\ref{hbound} above. 

\smallskip

We can thus rewrite the complex Monge-Amp\`ere equation as
$$
\frac{1}{V} (\theta_t+dd^c \p_t)^n=e^{\p_t+\sup_{X_t} \f_t} \mu_t=f_t\mu_t,
$$
where $\p_t=\f_t-\sup_{X_t} \f_t$ and  $f_t=\exp(\p_t+\sup_{X_t} \f_t)$. Combining the inequalities $\p_t \le 0$ and \eqref{bornesup}, it follows that the densities $f_t$ are uniformly bounded.

Recall that $\pi$ is smooth so, in particular, it is locally trivial.  Therefore,  Theorem~\ref{H1} applies and we can now appeal to Theorem \ref{thm:uniform2} with $p=+\infty$ and $0<\a<\a(\Theta,{\mathcal X})$ and obtain
 $$
 -M_K \le \p_t-V_{\theta_t} \le 0.
 $$
 Note that one used here that the volumes $V_t$ stay away from zero.  The conclusion follows since $\p_t-\f_t$ is uniformly bounded by \eqref{bornesup}.
\end{proof}

\begin{rem}
\label{remgt}
Set 
$$
V_{\Theta}(x)=V_{\theta_{\pi(x)}}(x).
$$
and
$$
\phi(x):=\f_{\pi(x)}(x).
$$
It is tempting to compare $\phi$ to 
$$
\hat{V}_{\Theta}=\sup \{u  \in \PSH(\mathcal X,\Theta);  \,\, u \le 0 \}.
$$

\noindent
Clearly $\hat{V}_{\Theta} \le V_{\Theta}$ hence $\hat{V}_{\Theta} -M_K \le \phi$.
It follows from \cite[Thm.~A]{JHM1} that $\phi$ is $\Theta$-psh on ${\mathcal X}$, thus $\phi -\sup_{\pi^{-1}(K)} \phi \le \hat{V}_{\Theta} $ and
$$
-M_K \le \phi-\hat{V}_{\Theta}  \le M_K.
$$
\end{rem}

\begin{rem}
The same results can be proved if the family $\pi:\cX\to \mathbb D$ is replaced by a smooth family $\pi:(\mathcal X,B)\to \mathbb D$ of pairs $(X_t,B_t)$ of log general type, i.e. such that $(X_t,B_t)$ is klt and $K_{X_t}+B_t$ is big for all $t\in \mathbb D$. 
\end{rem}

\subsection{Stable varieties} \label{sec:stable}
A stable variety is a projective variety $X$ such that 
\begin{enumerate}
\item $X$ has semi-log-canonical singularities. 
\item The $\mathbb Q$-line bundle $K_X$ is ample. 
\end{enumerate}
We refer to \cite{KSB,Alex,Karu, Kov,Kollar} for a detailed account of these varieties and their connection to moduli theory. 

In \cite{BG}, it was proved that a stable variety admits a unique Kähler-Einstein metric $\omega$. There are several equivalent definitions for such an object, but the simplest is probably the following:

\begin{defi}
A Kähler-Einstein metric $\omega$  on a stable variety 
is a smooth Kähler metric on $X_{\reg}$ such that 
$$\Ric(\om)=-\om \quad \mbox{and} \quad \int_{X_{\rm reg}}\om^n=(K_X^n)$$
 if $n=\dim_{\mathbb C}X$.
\end{defi}

  It is proved in \textit{loc. cit.} that $\om$ extends canonically across $X_{\rm sing}$ to a closed, positive current in the class $c_1(K_X)$. It is desirable to understand the singularities of $\om$ near $X_{\rm sing}$. In \cite[Thm.~B]{GW}, it is proved that $\omega$ has cusp singularities near the double crossings of $X$. Moreover, it is proved in \cite{Song17} that the potential $\vp$ of $\om$ with respect to a given Kähler form $\om_X\in c_1(K_X)$, i.e. $\omega=\omega_X+dd^c \f$,  is locally bounded on the klt locus of $X$. More precisely, given any divisor $D=(s=0)\sim_{\mathbb Q} K_X$ containing the non-klt locus of $X$ and given any $\ep>0$, there exists a constant $C_{\ep}>0$ such that 
\begin{equation}
\label{logpoles}
\vp \ge \ep \log |s|^2-C_{\ep},
\end{equation} 
where $|\cdot |$ is some smooth hermitian metric on $\mathcal O_X(D)$. We wish to refine that estimate and obtain a version for families of canonically polarized manifolds degenerating to a stable variety. For that purpose, we introduce the following definition. 

   Given a log resolution of singularities $f:Y\to X$ write $K_Y=f^*K_X+\sum_{i\in I} a_iE_i$ where $E=\sum_{i\in I}E_i$ is a divisor with simple normal crossings whose components are either exceptional or sit above the divisorial components of $\mathrm{Sing}(X)$. We set $J:=\{i\in I; a_i=-1\}$ which we assume to be non-empty and define the non-klt locus of $X$ by $\mathrm{Nklt}(X)=\cup_{j\in J}f(E_j)$; it is independent of $f$. Given $K\subset I$, we set $E_K=\cap_{i\in K}E_i$ and introduce the invariants
   \begin{equation}
   \label{def nu}
   \nu(f):=\max\{|K|; K\subset J, E_K\neq \emptyset\}, \quad \mbox{and} \quad \nu:=\inf_f \{\nu(f)\}
   \end{equation} where the infimum is taken over all the log resolutions of $X$. We have $\nu \in \{1, \ldots, n\}$. 

\begin{prop}
\label{minoration}
Let $X$ be a stable variety of dimension $n$ and let $\om_X\in c_1(K_X)$ be a Kähler metric. Next, let $\om=\om_X+dd^c\f$ be the Kähler-Einstein metric of $X$. Let $D=(s=0)$ be a divisor containing the non-klt locus of $X$ and let $|\cdot |$ be some smooth hermitian metric on $\mathcal O_X(D)$. For any $\ep>0$, there is a constant $C_{\ep}$ such that 
\begin{equation}
\label{logpoles2}
\vp \ge -(n+\nu+\ep) \log (-\log |s|)-C_{\ep},
\end{equation}
where $\nu$ is defined in \eqref{def nu}.
\end{prop}

\begin{rem}
The estimate \eqref{logpoles2} is an important refinement of \eqref{logpoles}, as it insures that
$\vp$ belongs to the finite energy class ${\mathcal E}^1(X,\omega_X)$, cf \cite{GZ07} or \cite[Sect.~2]{BEGZ} for the definitions and main properties of these classes.
\end{rem}

This estimate is almost optimal as the following two examples show. 

\begin{exa}
\label{BBS1}
Let $X=\overline{\mathbb B^n/\Gamma}^{\rm BBS}$ be the Baily-Borel-Satake compactification of a ball quotient, where $\Gamma\subset\mathrm{PU}(n,1)$ is a discrete group acting freely with finite covolume. It is a normal stable variety and it admits a resolution $f: \overline X\to X$ which is a toroidal compactification of $\mathbb B^n/\Gamma$ obtained by adding disjoint abelian varieties $D_i$. That is, $\nu=\nu(f)=1$. If one sets $D=\sum D_i$, the potential $\varphi$ of the Kähler-Einstein metric on $(\overline X,D)$ with respect to a smooth form in $c_1(K_{\overline X}+D)$ satisfies 
 $$\varphi= -(n+1) \log (-\log |s_D|)+O(1)$$
 if $(s_D=0)=D$.  
\end{exa}

\begin{exa}
\label{BBS2}
Let $X=\overline{\Delta^2/\Gamma}^{\rm BBS}$ be Baily-Borel-Satake compactification of a Hilbert modular surface, where $\Gamma\subset \mathrm{PSL}(2,\mathbb R)$ is a discrete group whose diagonal action on the bi-disk $\Delta^2$ is free with finite covolume. It is a normal stable surface and the exceptional divisor $D$ of its minimal resolution $f: \overline X\to X$  consist of cycles of $\mathbb P^1$'s. In particular, we have $\nu=\nu(f)=2$. Moreover, it is known that  the potential $\varphi$ of the Kähler-Einstein metric on $(\overline X,D)$ with respect to a smooth form in $c_1(K_{\overline X}+D)$ satisfies 
 $$\varphi= -4 \log (-\log |s_D|)+O(1)$$
 if $(s_D=0)=D$, cf \cite[p. 57-58]{KobR84}. 
 \end{exa}

\begin{proof}
Let $f:Y\to X$ be a resolution of singularities of $X$ such that $f$ induces an isomorphism over $X_{\reg}$ and achieves $\nu=\nu(f)$. The complex Monge-Ampère equation satisfied by $\vp$ pulls back to $Y$ and reads
\begin{equation}
\label{mastable}
(f^*\om_X+dd^c f^*\f)^n= e^{f^*\vp}d\mu_Y 
\end{equation}
where 
$d\mu_Y:=\prod_{i=1}^r |t_i|^{2a_i}\om_Y^n$ is a positive measure with possibly infinite mass. Here, $\om_Y$ is a Kähler form  on $Y$, and $(t_i=0)$ are divisors sitting over $X_{\rm sing}$ (they need not be exceptional though, as $X$ may have singularities in codimension one).  Finally, one has $a_i\ge -1$ for all $i$, and any divisor $(t_i=0)$ such that $a_i=-1$ sits above the non-klt locus of $X$. 

Now, let $F$ be an effective divisor on $X$ and let $\sigma_X \in H^0(X, \cO_X(F))$ be a section cutting out $F$. Let $h$ be a smooth hermitian metric on $\mathcal O_X(F)$; there exists a constant $C_F$ such that $\Theta_h(F) \le C_F \om_X$. One can scale $h$ such that $|\sigma_X|^2_{h}<e^{-2(n+2)C_F}$ on $X$. Finally, let $\sigma_Y:=f^*\sigma_X$ and and let $\psi:=-\log (-\log |\sigma_Y|^2)$. We have
$$dd^c \psi = \frac{\langle D\sigma_Y, D\sigma_Y \rangle}{|\sigma_Y|^2(-\log |\sigma_Y|^2)}-\frac{1}{(-\log |\sigma_Y|^2)}\cdot f^*\Theta_h(F).$$
By our choice of scaling, the function $A\psi$ is $f^*\om_X$-psh for any $0\le A \le 2(n+2)$. Moreover, it belongs to the class $\mathcal E(Y,f^*\om_X)$ thanks to e.g. \cite[Prop.~2.3]{G12} and \cite[Thm.~1.1(ii)]{DDL18}.

We apply this construction to $F$ some (very ample, say) divisor containing the non-klt locus of $X$. This yields a section $\sigma_Y$ of $f^*F$ that vanishes at order at least one along the $(t_i=0)$ for which $a_i=-1$. As a result, the measure
$$e^{(n+\nu+2\ep)\psi} d\mu_Y\lesssim\frac{1}{ \prod_{a_i=-1}  |t_i|^{2}(-\log  \prod_{a_i=-1} |t_i|^2)^{n+\nu+2\ep}} \prod_{a_i>-1} |t_i|^{2a_i}\cdot \om_Y^n$$
has a density $g_\ep$ with respect to $\om_Y^n$ that satisfies
$$\int_{Y}g_\ep |\log g_\ep|^{n+\ep} \om_Y^n <+\infty$$
for any $\ep>0$ as we observed in the proof of Claim~\ref{intcontrol}, cf \eqref{integral} and the discussion below. By Theorem \ref{thm:uniform3}, 
this implies that the unique solution $u_\ep\in \mathcal E(Y,\frac 12 f^*\om_X)$ of the Monge-Ampère equation
$$ (\frac 12 f^*\om_X+dd^c u_\ep)^n= e^{u_\ep+(n+\nu+2\ep)\psi} d\mu_Y$$
is bounded, i.e. there exists a constant $C_\ep>0$ such that 
\begin{equation}
\label{bounded}
\|u_\ep\|_{L^{\infty}(Y)}\le C_\ep.
\end{equation}
Now, the function $v_\ep:=u_\ep+(n+\nu+2\ep)\psi\in \mathcal E(Y,f^*\om_X)$ satisfies the inequality
\begin{align*}
(  f^*\om_X+dd^cv_\ep)^n&\ge (\frac 12   f^*\om_X+dd^cu_\ep)^n\\
 &=e^{v_\ep} d\mu_Y,
 \end{align*}
i.e. $v_\ep$ is a subsolution of \eqref{mastable}. By the comparison principle, we obtain that $f^*\vp\ge v_\ep$ and it follows from \eqref{bounded} that
$$f^*\vp\ge (n+\nu+2\ep) \psi-C_\ep,$$
from which the conclusion follows.
\end{proof}

\subsection{Stable families}
Now one can establish a family version of the previous estimate, i.e. Proposition~\ref{minoration}. In Setting~\ref{fset}, let us assume additionally that $\Krel$ is ample. We let $h$ be a smooth hermitian metric on $\Krel$ whose curvature is a Kähler form $\om_\cX:=\Theta_h(\Krel)$; we set $\om_{X_t}:=\om_\cX|_{X_t}$. If $\Omega$ is a local trivialization of $m\Krel$, then the quantity
$$\mu_{\cX/\bD,h}:=\frac{i^{n^2}(\Omega\wedge \overline \Omega)^{1/m}}{|\Omega|_h^{2/m}}$$
is independent of $\Omega$ or $m$ (yet it depends on $h$) and for any $t\in \bD$, it restricts to $X_t^{\reg}$ as a positive measure 
$$\mu_{X_t,h}:=\mu_{\cX/\bD}|_{X_t^{\reg}}$$
which we extend by zero across $X_t^{\rm sing}$. For each $t\in \mathbb D$, there exists a unique Kähler-Einstein metric $\om_{{\rm KE},t} \in c_1(K_{X_t})$ which solves the Monge-Ampère equation
\begin{equation}
\label{KEt}
(\om_{X_t}+dd^c \f_t)^n = e^{\f_t}\mu_{X_t,h}
\end{equation}
on $X_t$. This is due to 
\cite{Aubin, Yau78} when $X_t$ is smooth and to \cite{BG} in general.

\begin{thm} \label{thm:stablefamily}
In Setting~\ref{fset}, assume that
\begin{enumerate}
\item[$\bullet$] The relative canonical bundle $\Krel$ is ample.
\item[$\bullet$] The central fiber $X_0$ is irreducible. 
\end{enumerate} 
Let $\om_{X_t}+dd^c \f_t$ be the Kähler-Einstein metric of $X_t$, solution of \eqref{KEt} and let $D=(s=0)\subset \cX$ be a divisor which contains $\mathrm{Nklt}(\cX,X_0)$, cf \eqref{nklt}.
Fix $|\cdot |$ a some smooth hermitian metric on $\mathcal O_\cX(D)$. 
Up to shrinking $\bD$, then for any $\ep>0$, there exists 
$C_{\ep}>0$  such that the inequality
\begin{equation}
\label{logpoles3}
\vp_t \ge -(n+\nu+\ep) \log (-\log |s|)-C_{\ep}
\end{equation}
holds on $X_t$ for any $t\in \bD$, where $\nu$ is defined in \eqref{nu2}. 
\end{thm}

This estimate improves an interesting control obtained previously by J.Song (see \cite[Lem.~4.2]{Song17}).

\begin{proof}
Let $f:\cX'\to \cX$ be a semi-stable model as in \eqref{semistable}, achieving $\nu(f)=\nu$.
The first observation is that the behavior of $f^*(\Omega_t\wedge \overline \Omega_t)^{1/m}$ and $f^*\mu_{X_t,h}$ on $X_t$ is the same, uniformly in $t$, because there exists a constant $C>0$ such that for any trivializing open set, one has $C\ge |\Omega|^2_h \ge C^{-1}$, where $\Omega$ ranges among the finitely many trivializations of $m\Krel$. This follows from the fact $h$ is a smooth hermitian metric on $m\Krel$. 

We set $\psi:=f^*(-\log (-\log |s|^2))$; it is a quasi-psh function on $\cX'$ satisfying $$\psi\le \psi_F+O(1)$$
where $\psi_F$ is defined in \eqref{psi}.  

By scaling the metric $|\cdot |$ on $\cO_{\cX}(D)$, one can assume that $A\psi$ is $f^*\om_X$-psh for any $0\le A \le 2(n+2)$. For any $t\in \bD^*$, the function $\psi_t:=\psi|_{X'_t}$ belongs to $ \mathcal E(X'_t,f^*\om_{X_t})$ by the same argument as in the proof of Proposition~\ref{minoration}.

Let $u_{\ep,t}\in \mathcal E(X'_t,\frac 12 f^*\om_{X_t})$ be the unique solution of the Monge-Ampère equation
\begin{equation}\label{eqJ}
 (\frac 12 f^*\om_{X_t}+dd^c u_{\ep,t})^n= e^{u_{\ep,t}+(n+\nu+2\ep)\psi_t} f^*\mu_{X_t,h}.
 \end{equation}

One can write $e^{(n+\nu+2\ep)\psi_t} f^*\mu_{X_t,h}=e^{\rho_t} f^*\om_{X_t}^n$ where $\rho_t$ is the restriction to $X'_t$ of the difference of quasi-psh functions on $\cX'$ with uniformly bounded $L^1$ norm on $X'_t$. Set $V:=\int_{X_t} \om_{X_t}^n$. Integrating both sides of \eqref{eqJ} and using Jensen inequality we have
\begin{align*}
\frac{V}{2^{n}} & =  \int_{X'_t}e^{u_{\ep,t}+(n+\nu+2\ep)\psi_t} f^*\mu_{X_t,h} \\
& = V \int_{X'_t}e^{u_{\ep,t}+\rho_t}\frac{f^*\om_{X_t}^n}{V} \\
& \ge V \cdot  e^{ \frac 1 V \int_{X'_t}(u_{\ep,t}+\rho_t) f^*\om_{X_t}^n}
\end{align*}
Since $\int_{X'_t} |\rho_t| f^*\om_{X_t}^n $ is uniformly bounded, we get that $\int_{X_t'}u_{\ep,t} f^*\om_{X_t}^n \le C$ for some $C>0$ independent of $\ep,t$. Since $u_{\ep,t}$ is $f^*(\frac 12 \om_{X_t})$-psh, it is the pull-back of a $\frac 12 \om_{X_t}$-psh function on $X_t$ to which one can apply Proposition~\ref{mainprop} since $\pi$ is projective. To summarize, we get an upper bound 
\begin{equation}
u_{\ep,t} \le C.
\end{equation}

 Next, we wish to apply Theorem~\ref{thm:uniform3}; in order to do so, one has to check that hypotheses (H1) and (H2') are satisfied in our situation. For (H1), it is a consequence of Theorem~\ref{H1} \--- recall that up to shrinking $\bD$, all fibers $X_t$ are irreducible since so is $X_0$. As for (H2'), it follows from Lemma~\ref{corintcontrol} that we pull back via $f$ to the smooth Kähler manifold $X_t'$. 
All in all, we can find $C_{\ep}>0$ independent of $t\in \bD$ such that
\begin{equation}
\label{bounded2}
\|u_{\ep,t}\|_{L^{\infty}(X'_t)}\le C_\ep.
\end{equation}
Now, the function $v_{\ep,t}:=u_{\ep,t}+(n+\nu+2\ep)\psi_t\in \mathcal E(X_t',f^*\om_{X_t})$ satisfies the inequality
\begin{align*}
(  f^*\om_{X_t}+dd^cv_{\ep,t})^n&\ge (\frac 12   f^*\om_{X_t}+dd^cu_{\ep,t})^n\\
 &=e^{v_{\ep,t}} f^*\mu_{X_t,h},
 \end{align*}
i.e. $v_{\ep,t}$ is a subsolution of \eqref{KEt}. By the comparison principle, we obtain that $f^*\vp_t\ge v_{\ep,t}$ and it follows from \eqref{bounded2} that
$$f^*\vp_t\ge (n+\nu+2\ep) \psi_t-C_\ep,$$
from which the conclusion follows.
\end{proof}

 \section{Log Calabi-Yau families} 
 \label{sec:calabi}
 
 \subsection{Families of Calabi-Yau varieties}
 \label{famCY}
 
 In Setting~\ref{fset}, let us assume additionally that $\Krel$ is relatively trivial and that $X_0$ has canonical singularities. For $t$ small enough, $X_t$ has canonical singularities as well and $K_{X_t}$ is linearly trivial.  
 
 Let $\alpha$ be a relative Kähler cohomology class on $\cX$ represented by a relative Kähler form $\om$. We set $\alpha_t:=\alpha|_{X_t}$, $\om_{X_t}:=\om|_{X_t}$ and $V:=\int_{X_t} \om_t^n$; it does not depend on $t$, cf Lemma~\ref{constant}. Let $\Omega$ be a  trivialization of $\Krel$, so that the quantity
$$\mu_{\cX/\bD}:=i^{n^2}\Omega\wedge \overline \Omega$$
 restricts to $X_t^{\reg}$ as a positive measure 
$$\mu_{X_t}:=\mu_{\cX/\bD}|_{X_t^{\reg}}$$
which we extend by zero across $X_t^{\rm sing}$. We set $c_t:=\log \int_{X_t}d\mu_{X_t}$. For each $t\in \mathbb D$, there exists a unique Kähler-Einstein metric $\om_{{\rm KE},t}=\om_t+dd^c \vp_t \in \alpha_t$ which solves the Monge-Ampère equation
\begin{equation}
\label{KEtt}
\frac{1}{V} (\om_{t}+dd^c \f_t)^n =e^{-c_t} \mu_{X_t}
\end{equation}
on $X_t$ and that we normalize by $\sup_{X_t} \vp_t=0$. This is due to 
\cite{Yau78} when $X_t$ is smooth and to \cite{EGZ} in general.

\begin{thm} \label{thm:CYfamily}
In Setting~\ref{fset}, assume that
\begin{enumerate}
\item[$\bullet$] The relative canonical bundle $\Krel$ is trivial.
\item[$\bullet$] The central fiber $X_0$ has canonical singularities. 
\item[$\bullet$] Assumption~\ref{aspt} is satisfied. 
\end{enumerate} 
Let $\om_{t}+dd^c \f_t$ be the Kähler-Einstein metric of $X_t$, solution of \eqref{KEtt}. Up to shrinking $\bD$, there exists 
$C>0$  such that one has
\begin{equation}
\label{borneCY}
\mathrm{osc}_{X_t} \vp_t \le C
\end{equation}
 for any $t\in \bD$, where $\mathrm{osc}_{X_t}(\vp_t)=\sup_{X_t}\vp_t-\inf_{X_t}\vp_t$. 
\end{thm}

A particular case of this result has been obtained previously by Rong-Zhang (see \cite[Lemma 3.1]{RZ}) by using Moser iteration process.

\begin{rem}
\label{cansing}
One can replace the first two assumptions in Theorem~\ref{thm:CYfamily} above by the following weaker ones: $\cX$ is normal, $\mathbb Q$-Gorenstein, $\Krel$ is trivial and $X_0$ has canonical singularities. Indeed, it follows from the inversion of adjunction \cite[Thm.~2.3]{Kol18} that $(X, X_t)$ is lc for $t$ close enough to $0$. Moreover, an easy computation relying on the adjunction formula shows that $X_t$ has canonical singularities for $t$ close to $0$.  
\end{rem}

\begin{proof}[Proof of Theorem~\ref{thm:CYfamily}]
A first observation is that the quantities $c_t$ remain bounded when $t$ varies thanks to Lemma~\ref{canonical}. The result now follows from Theorem~\ref{thm:uniform1}. Indeed, (H1) is satisfied thanks to Theorem~\ref{H1} while (H2) holds thanks to Lemma~\ref{canon} that we pull back to $X_t'$ via $f$, with the notation of the Lemma. 
\end{proof}

 \subsection{The log Calabi-Yau setting}
In the sequel we use the following setting. 

\begin{set}
\label{set2}
Let $X$ be an $n$-dimensional compact Kähler space and let $B=\sum b_iB_i$ be an effective $\mathbb R$-divisor such that the pair $(X,B)$ has klt singularities. We assume furthermore that the log Kodaira dimension of the pair $(X,B)$ vanishes, i.e. $$\kappa(K_X+B)=0.$$ 
\end{set}
In what follows, we denote by $E$ the (unique) effective $\mathbb R$-divisor in $c_1(K_X+B)$.  Thanks to log abundance in numerical dimension zero (see \cite[Cor.~1.18]{JHM2}), a particular instance of such pairs is provided by klt pairs $(X,B)$ with rational boundary such that the Chern class $c_1(K_X+B)\in H^2(X,\mathbb Q)$ vanishes.

\begin{defi}
In Setting~\ref{set2}, given a cohomology class $\alpha\in H^{1,1}(X,\mathbb R)$ that is nef and big, it follows from \cite{BEGZ} that there exists  a unique \textit{singular Ricci flat current} $T \in \a$, i.e. a closed,  positive current of bidegree
$(1,1)$ representing $\a$,  with the following properties:
\begin{itemize}
\item[$(i)$] $T$ has minimal singularities in $\a$;
\item[$(ii)$] $T$ is a K\"ahler form on the analytic open set $\Omega_\a:= (X_{\rm reg}\setminus \mathrm{Supp}(B+E)) \cap \Amp(\a)$;
\item[$(iii)$] $\Ric(T)=[B]-[E]$ on $X_{\rm reg}$.
\end{itemize}
\end{defi}
The current $T$ can be found by solving the Monge-Ampère equation
\begin{equation}
\label{eqMA}
\mathrm{vol}(\a)^{-1}(\theta+dd^c \varphi)^n=\mu_{(X,B)}
\end{equation}
where $\theta\in \alpha$ is a smooth representative, $\varphi\in \PSH(X,\theta)$ is the unknown function and 
$$
\mu_{(X,B)}= (s\wedge \bar s)^{\frac 1m} e^{-\phi_B}.
$$
 Here, $s\in H^0(X,m(K_X+B))$ is any non-zero section (for some $m\ge 1$) and $\phi_B$ is the unique singular psh weight on $\mathcal O_X(B)$ solving $dd^c \phi_B=[B]$ and normalized by 
 $$
 \int_X (s\wedge \bar s)^{\frac 1m} e^{-\phi_B}=1.
 $$

 We let ${\mathcal K}_X$ denote the K\"ahler cone, i.e. the set of cohomology classes
$\a \in H^{1,1}(X,\R)$ which can be represented by  a K\"ahler form. 
We fix $(\a_t)_{0 < t \le 1} \subset {\mathcal K}_X$ a path of K\"ahler classes
and assume that $\a_t \rightarrow \partial {\mathcal K}_X$ as $t \rightarrow 0$.

When $X$ is smooth and $B=0$,  the existence of a unique Ricci flat K\"ahler metric 
$\omega_t$ in $\a_t$ for each $0<t \le 1$ dates back to the celebrated work of Yau \cite{Yau78}.
A basic problem is to understand the asymptotic behavior of
the $\omega_t$'s, as $t \rightarrow 0$.
This problem has a long history, we refer the reader to \cite{GTZ13}
for references.
 
Despite  motivations coming from mirror symmetry, not much 
is known when the norm of $\a_t$ converges to $+\infty$
(this case is expected to be the mirror of a large complex structure limit,
see \cite{KS01} or the recent survey \cite{TosSurvey}). We thus only consider the case
when $\a_t \rightarrow \a_0 \in \partial {\mathcal K}_X$.
There are  two rather different settings, depending on whether
  $\a_0$ is big ($\vol(\a_0)>0$), or merely nef
 with $\vol(\a_0)=0$.

\subsection{The non-collapsing case}

We first consider the case when the volumes of the $\a_t$'s are non-collapsing, i.e. $\vol(\a_0)>0$.
Then, we have the following result, generalizing theorems of Tosatti \cite{Tos09} and Collins-Tosatti \cite{CT15}.

 \begin{thm}   \label{thm:non-collapsing}
Let $(X,B)$ be a pair as in Setting~\ref{set2} and let $(\a_t)_{0 < t \le 1} \subset {\mathcal K}_X$ be a smooth path of K\"ahler classes
such  that $\a_t \rightarrow \a_0 \in \partial {\mathcal K}_X$ as $t \rightarrow 0$,
with $\vol(\a_0)>0$.\\
Then, the singular Ricci-flat currents $T_t \in \a_t$ converge to $T_0$ as $t \rightarrow 0$ weakly on $X$, and locally smoothly on $\Omega_\a$.
\end{thm}


\begin{proof}
One can work in a desingularization $p:Y \rightarrow X$ of $X$.
The path $\a_t$ induces a path $\beta_t=p^*\a_t$ of semi-positive and big classes.
The currents $T_t$ can be decomposed as $T_t=\theta_t+dd^c \varphi_t$ where $\theta_t\in \beta_t$ is a smooth representative and  $\varphi_t$ are normalized by $\sup_{X_t} \varphi_t=0$ and solve the complex Monge-Amp\`ere equation
$$
\frac{1}{V_t} (\theta_t+dd^c \f_t)^n=\mu_Y=f \, dV_Y,
$$
where the volumes $V_t=\alpha_t^n$ are bounded away from zero and infinity, $C^{-1} \le V_t \le C$,
and $\mu_Y=f \, dV_Y$ is  a fixed volume form, with $f \in L^{p}(Y)$ for some $p>1$
(because $(X,B)$ has klt singularities, see \cite[Lem.~6.4]{EGZ}).

The hypothesis of Theorem \ref{thm:uniform1} $(H2)$ is thus trivially satisfied, while $(H1)$ follows if we initially bound from above
$\a_t \le \gamma_X$ by a fixed K\"ahler class.
The most delicate ${\mathcal C}^0$-estimate follows thus here from 
Theorem \ref{thm:uniform2}. When $X$ is smooth, the  ${\mathcal C}^0$-estimate in \cite{Tos09} is obtained by using a Moser iteration argument
as in Yau's celebrated paper \cite{Yau78}, but this argument can no longer be applied in the present more singular setting. \\
The rest of the proof is then roughly the same as in the case of smooth manifolds. It consists in adapting Yau's Laplacian estimate by using Tsuji's trick (first used in \cite{Tsuji88}), the remaining higher order estimates being local ones.
\end{proof}

\subsection{The collapsing case}

We now consider the case when the volumes of the $\a_t$'s are collapsing, i.e. $\vol(\a_0)=0$.
This case is more involved and only special cases are fully understood.

Suppose there is a surjective, holomorphic map with connected fibers $f:X \rightarrow Z$, where
$Z$ is a compact, normal K\"ahler space of positive dimension $m$. We denote by $k:=n-m=\dim X-\dim Z$ the relative dimension of the fiber space $f$. 
We let $S_Z$ denote the smallest proper analytic subset $\Sigma \subset Z$ such that 
\begin{enumerate}
\item[$\bullet$] $\Sigma$ contains the singular locus $Z_{\rm sing}$ of $Z$,
\item[$\bullet$] The map $f$ is smooth on $f^{-1}(Z\setminus \Sigma)$,
\item[$\bullet$] For any $z\in Z\setminus \Sigma$, $\mathrm{Supp}(B)$ intersects $X_z$ transversally,
\end{enumerate}
and we set $S_X=f^{-1}(S_Z)$. Finally, we set $Z^\circ:=Z\setminus S_Z$ and $X^\circ:=X\setminus S_X=f^{-1}(Z^\circ)$. By the last item, each component of $B|_{X^{\circ}}$ dominates $Z^\circ$. 

A general fiber $X_z$ satisfies $\kappa(K_{X_z}+B_z) \ge 0$, but the inequality may be strict. If $c_1(K_X+B)=0$, then log abundance implies that $K_{X_z}+B_z\sim_{\mathbb Q} \mathcal O_{X_z}$ for $z$ general. Moreover, Iitaka's conjecture predicts that $\kappa(K_{X_z}+B_z)$ vanishes as soon as $\kappa(Z) \ge 0$, which in turn should be equivalent to $Z$ not being uniruled.


Fix $\omega_Z$ a K\"ahler form on $Z$. For simplicity, we assume that $\int_Z\om_Z^m=1$. The form $f^* \omega_Z$  is a semi-positive form
such that $f^*\om_Z^p=0$ for any $p>m$. We also choose a Kähler form $\om_X$ on $X$. The quantity $\int_{X_z} \om_X^k = f_*\om_X^k$ is constant in $z\in Z$; up to renormalizing $\om_X$, we may assume that the constant is $1$. 
 
We assume that our path $(\alpha_t)_{t\ge 0}$ in $H^{1,1}(X,\mathbb R)$ is given by $\a_0=\{f^*\om_Z\}$ and $\a_t=\a_0+t \{\omega_X \}$. As a result, one has
\begin{equation}
\label{vt}
V_t:=\vol(\a_t)=\binom{n}{k} t^{k} \int_X f^*\om_Z^{m} \wedge \omega_X^{k}+o(t^{k})=\binom{n}{k} t^{k} +o(t^{k}) .
\end{equation}
We set $\omega_t:=f^*\om_Z+t \omega_X$ and let $\omf:=\omega_t+dd^c \f_t$ denote the singular Ricci-flat current in $\a_t$, normalized by $\int_{X}\f_t\om_X^n=0$. It satisfies
$$
 \omf^n=V_t\cdot \mu_{(X,B)},
$$
cf Eq.~\eqref{eqMA}. The probability measure $f_*\mu_{(X,B)}$ has $L^{1+\ep}$-density with respect to $\om_Z^m$ thanks to \cite[Lem.~2.3]{EGZ16}. Therefore, there exists a unique positive current $\om_{\infty}\in \{\om_Z\}$ with bounded potentials, solution of the Monge-Ampère equation 
$$\om_{\infty}^m= f_*\mu_{(X,B)},$$
cf \cite{EGZ}. In the case where $X$ is smooth, $B=0$ and $c_1(X)=0$, the Ricci curvature of $f_*\mu_X$ (or, equivalently, $\om_{\infty}$) coincides with the Weil-Petersson form of the fibration $f$ of Calabi-Yau manifolds. We propose the following problem. 

\begin{pb} 
\label{WP}
Let $f:X\to Z$ be a surjective holomorphic map with connected fibers between compact, normal Kähler spaces. Assume that there exists an effective divisor $B$ on $X$ such that $(X,B)$ is klt and $\kappa(K_X+B)=0$. Let $\omega_X$ (resp. $\om_Z$) be a Kähler form on $X$ (resp. $Z$) and let $\omf$ be the unique singular Ricci-flat current in $\{f^*\om_Z+t\om_X\}$ for $t>0$. 

\noindent
Then, the currents $\omf$ converge weakly to $f^*\om_{\infty}$ when $t\to 0$, where $\om_{\infty}\in \{\om_Z\}$ solves $\om_{\infty}^{\dim Z}= f_*\mu_{(X,B)}$.
\end{pb}

The Problem above is motivated by a string of papers (cf below) where the expected result is proved along with some additional information on the convergence. 
\begin{thm} \cite{Tos10,GTZ13,TWY18, HT18}
Assume that $X$ is smooth, $B=0$ and $c_1(K_X)=0$. Then, the metrics $\omf$ converge to $f^*\om_{\infty}$
in the ${\mathcal C}^{\alpha}_{\rm loc}$-sense on $X \setminus S_X$, for some $\alpha>0$.
\end{thm}

In this section, we aim at providing a positive answer to Problem~\ref{WP} whenever $X$ is smooth, $B$ has simple normal crossings support and $c_1(K_X+B)=0$. We will follow the strategy of Tosatti \cite{Tos10} rather closely. However, some adjustments need to be made, requiring the use of conical metrics and the results of the present paper. 

\begin{thm} 
\label{WP2}
In the Setting of Problem~\ref{WP}, assume furthermore that $X$ is smooth, $B$ has snc support and $c_1(K_X+B)=0$. 
Then, 
$\omf\rightarrow f^*\om_{\infty}$ as currents on $X$, when $t$ goes to $0$. 
\end{thm}

\begin{proof}
We will proceed in several steps, similarly to \cite{Tos10}. In order to simplify some computations to follow, one will assume that $S_Z$ is contained in a divisor $D_Z$, cut out by a section $\sigma_Z \in H^0(Z, \mathcal O_Z(D_Z))$. If $Z$ is projective, this is not a restriction. The general case 
requires to follow Tosatti's computations more closely but does not present significant additional difficulties. \\

\noindent
\textbf{Step 1.} \textit{Choice of some suitable conical metrics}

\noindent
We list in the Proposition below the properties of the conical metric that will be important for the following. It is mostly a recollection of well-known results, cf e.g. \cite{GP}. By abuse of notation, we will not distinguish between $B$ and $\mathrm{Supp}(B)$.

\begin{prop}
\label{conic}
There exists a Kähler current $\om_B\in \{\om_X\}$ on $X$ such that 
\begin{enumerate}
\item $\om_B$ is a smooth Kähler form on $\xmb $ and has conical singularities along $B$. 
\item There exists a constant $C>0$ and a quasi-psh function $\Psi\in \mathcal C^{\infty}(\xmb)\cap L^{\infty}(X)$ such that
the following inequalities of tensors hold in the sense of Griffiths on $\xmb$
$$-(C\om_B+dd^c \Psi) \otimes \mathrm{Id}_{T_X} \le \Theta_{\om_B}(T_X) \le C\om_B \otimes \mathrm{Id}_{T_X}.$$
\item Let $h:=\om_B^n/\om_X^n$. There exists $p>1$ such that for any $K\Subset Z^\circ$, one has 
$$\sup_{z\in K} \|h|_{X_z}\|_{L^p(\om_{X_z}^k)} <+\infty.$$
\end{enumerate}
\end{prop}

\begin{proof}[Sketch of proof of Proposition~\ref{conic}]
To construct such a metric $\om_B$, one first chooses smooth metrics $h_i$ on $B_i$, sections $s_i\in H^0(X,\mathcal O_X(B_i))$ cutting out $B_i$, and one sets $\om_B:=\om_X+dd^c \sum_i |s_i|^{2(1-b_i)}$. Up to scaling down the metrics $h_i$, one can easily achieve the first condition. The third condition also follows easily.  

The left-hand side inequality of 2 ("lower bound" on the holomorphic bisectional curvature) follows from \cite[(4.3)]{GP} with $\ep=0$. As for the right-hand side inequality (upper bound on the holomorphic bisectional curvature), a proof has been given in \cite[App.~A]{JMR} in the case where $B$ is smooth but a very simple argument has been found by Sturm, cf \cite[Lem.~3.14]{Rub14}. 
\end{proof}

\noindent
\textbf{Step 2.} \textit{Estimates}

\noindent
We list in the Proposition below various estimates on $\omf$ that will be useful for the last step. First, we define for $z\in Z^\circ$ the quantity $ \uf(z):=\int_{X_z}\f_t \om_{X_z}^k$. In the following, we will not distinguish between $\uf$ and $f^*\uf$. 
\begin{prop}
\label{estimates}
There exist a constant $C>0$ as well as a positive function $g \in \mathcal C^{\infty}(X^\circ)$, both independent of $t$,  such that 
\begin{enumerate}
\item $\|\f_t\|_{L^{\infty}(X)}\le C$.
\item $\omf \ge C^{-1} f^*\om_Z$.
\item $|\f_t-\uf| \le g \cdot t$.
\item $g^{-1}t \cdot \om_B\le  \omf \le g\cdot \om_B$. 
\item $g^{-1}t \cdot \om_{B_z} \le \omf|_{X_z} \le  gt \cdot \om_{B_z} \quad$ for all $z\in Z^\circ$.
\end{enumerate}
\end{prop}

\begin{proof}[Proof of Proposition~\ref{estimates}]
In this proof, $C$ will denote a constant that may change from line to line but is independent of $t$. In the same way, $g$ will be a smooth, positive function on $X^\circ$ that should be thought as blowing up to $+\infty$ near $S_X$; it can be assumed to come from $Z^\circ$ via $f$. \\

\textbf 1. This is a consequence of \cite[Thm.~A]{EGZ08} or \cite[p.~401]{DemPal}.\\

\textbf 2. Let us consider the holomorphic map $f:(\xmb, \omf)\to (Z,\om_Z)$. Given that $\Ric(\om_{\varphi_t})=0$ and that $\om_Z$ is a smooth Kähler metric on the compact space $Z$, Chern-Lu's formula \cite{Chern, Lu} provides a constant $C>0$ such that the non-negative function $u=\tr_{\omf}f^*\om_Z$ satisfies
$$\Delta_{\omf} \log u \ge -C (1+u)$$
on $\xmb$. Now, $$\Delta_{\omf}(-\f_t)=\tr_{\omf}(-\omf+f^*\om_Z+t\om_X) \ge u-n$$ so that setting $A=C+1$, one finds
$$\Delta_{\omf}( \log u-A\f_t) \ge u-C.$$
Let $\tau$ be a section of $\mathcal O_X(\lceil B\rceil)$ cutting out $B$ and let $h_B$ be a smooth hermitian metric on that line bundle. We set $\chi:= \log |\tau|^2_{h_B}$. As $\omf$ is a Kähler current and $\chi$ is quasi-psh, there exists a constant $C_t>0$ such that $dd^c \chi \ge -C_t \omf$. Therefore, for any $\delta \in (0,C_t^{-1})$, one has an inequality
$$\Delta_{\omf}( \log u-A\f_t+\delta \chi) \ge u-C.$$
As $\omf$ is a conical metric for $t>0$, the function $u$ is bounded above on $\xmb$ and therefore, $H_{t,\delta}:= \log u-A\f_t+\delta \chi$ attains its maximum at a point $x_{t,\delta}\in \xmb$ such that $u(x_{t,\delta})\le C$. As a result, the estimate obtained in 1. allows one to show that for any $x\in \xmb$, one has
\begin{align*}
\log u(x)&= H_{t,\delta}(x) +A\f_t(x)-\delta \chi(x)\\
& \le H_{t,\delta}(x_{t,\delta})+C-\delta \chi(x)\\
& \le C-\delta \chi.
\end{align*}
As this holds for any $\delta>0$ small enough, we can pass to the limit and conclude that $u\le e^C$ on $\xmb$, hence everywhere.  \\

\textbf 3. The equation solved by $\varphi_t$ can be rewritten as 
\begin{equation}
\label{eqremix}
(f^*\om_Z+t\om_X+dd^c \varphi_t)^n= t^k e^{F_t} \om_B^n
\end{equation}
where $F_t$ is uniformly bounded independently of $t$. Next, one has on $X_z$ ($z\in Z^\circ$)
\begin{equation}
\label{restr}
\frac{(\omf|_{X_z})^k}{\om_{X_z}^k} = \frac{\omf^k\wedge f^*\om_Z^m}{\om_X^k\wedge f^*\om_Z^m} \le Cg \cdot \frac{\omf^n}{\om_X^n}
\end{equation}
thanks to 2. Observing that $\omf|_{X_z}=(\omf-dd^c \uf)|_{X_z}$, one sees from Eq.~\eqref{eqremix} that $(\vp_t-\uf)|_{X_z}$ satisfies
$$(\om_{X_z}+dd^c (\frac 1t(\vp_t-\uf)|_{X_z}))^k  \le gh|_{X_z} \cdot \om_{X_z}^k$$
where $h=\om_B^n/\om_X^n$. Thanks to the third item of Proposition~\ref{conic}, Theorem~\ref{H1} and 
Theorem~\ref{thm:uniform1}, we can derive 3. Actually, we used a version of Theorem~\ref{H1} for higher-dimensional bases, but only for smooth morphisms, in which case the proofs in the one-dimensional case go through without any change. \\

\textbf {4.a} We first prove the right-hand side inequality. Let us start by writing $\om_B=\om_X+dd^c \psi_B$ where $\psi_B\in L^{\infty}(X) \cap \mathcal C^{\infty}(\xmb)$. From the second item of Proposition~\ref{conic} and Siu's Laplacian inequality (cf \cite[(2.2)]{GP}), one concludes that 
$$\Delta_{\omf}( \log \tr_{\om_B}\omf+\Psi) \ge -C(1+ \tr_{\omf} \om_B). $$
Next, one has 
 \begin{equation}
 \label{dphi}
 \Delta_{\omf}(-\f_t+t\psi_B)=\tr_{\omf}(-\omf+f^*\om_Z+t\om_B) \ge t\tr_{\omf} \om_B-n
 \end{equation}
 so that 
 \begin{equation}
 \label{delta1}
 \Delta_{\omf}( \log \tr_{\om_B}\omf+\Psi-\frac At \f_t+A\psi_B) \ge  \tr_{\omf} \om_B-\frac Ct. 
 \end{equation}
We want to bound from below the term $dd^c \uf$. In order to achieve this, one writes
\begin{align}
dd^c \uf &= dd^c f_*(\f_t \om_X^k) 
 = f_*(dd^c \f_t \wedge \om_X^k) \label{borneinf3}  \\
& \ge - f_*(f^*\om_Z\wedge \om_X^k + t\om_X^{k+1}) \nonumber\\
& \ge -\om_Z-t f_*\om_X^{k+1} \ge -g\cdot \om_Z\nonumber
\end{align}
given that $f_*\om_X^k=1$. In particular, one has 
\begin{equation}
\label{dphi2}
 \Delta_{\omf} \uf \ge -g
 \end{equation}
  thanks to 2. Combining that estimate with \eqref{delta1}, one finds
 \begin{equation}
 \Delta_{\omf}( \log \tr_{\om_B}\omf+\Psi-\frac At( \f_t-\uf)+A\psi_B) \ge  \tr_{\omf} \om_B-\frac gt. 
 \end{equation}
We now set $F:=\Psi-\frac At( \f_t-\uf)+A\psi_B$; it is a bounded function on $X$, smooth on $X^\circ \setminus B$ such that
\begin{equation}
\label{boundF}
|F| \le g
\end{equation}
thanks to 3. Next, we set  $\rho:=\chi+f^*\log |\sigma_Z|_{h_{D_Z}}^2$ where $\chi$ is defined in the proof of 2 and $h_{D_Z}$ is a smooth hermitian metric on the divisor $D_Z$ (containing $S_Z$). As $\rho$ is quasi-psh on $X$, there exists $C_t>0$ such that 
\begin{equation}
\label{rho}
dd^c \rho \ge -C_t \omf.
\end{equation}
 In particular, one has 
 \begin{equation}
 \label{delta2}
 \Delta_{\omf}( \log \tr_{\om_B}\omf+F+\delta \rho) \ge  \tr_{\omf} \om_B-\frac gt. 
 \end{equation}
as soon as $\delta\in (0,C_t^{-1})$. We choose such a $\delta$ for the following. As the quantity $\log \tr_{\om_B}\omf+F$ is globally bounded on $X$ and smooth on $X^\circ \setminus B$, the function $\log \tr_{\om_B}\omf+F+\delta \rho$ attains it maximum at a point $y_{t,\delta}\in X^\circ \setminus B$ such that 
\begin{equation*}
 \tr_{\omf} \om_B(y_{t,\delta}) \le \frac gt
\end{equation*}
thanks to the maximum principle. Combining this with 2, one finds 
\begin{equation}
\label{max2}
 \tr_{\omf} (f^*\om_Z+t\om_B)(y_{t,\delta}) \le g
\end{equation}
Using the standard inequality
$$\tr_{\om'}\om \le \frac{\om^n}{\om'^n} (\tr_{\om}\om')^{n-1}$$
valid for any two positive $(1,1)$-forms, one gets from \eqref{max2}
$$ \tr_{ f^*\om_Z+t\om_B}(\omf)(y_{t,\delta}) \le g$$
since $\omf^n \simeq t^k \om_B^n$ is uniformly comparable to $(f^*\om_Z+t\om_B)^n$ by Claim~\ref{claimm} below. As $\om_B$ dominates $f^*\om_Z+t\om_B$, we infer from the inequality above the following
\begin{equation}
\label{max4}
 \tr_{\om_B} \omf (y_{t,\delta}) \le g.
\end{equation}
Given the definition of $y_{t,\delta}$, the boundedness of $F$ and that $\delta>0$ is arbitrary, we find as in the proof of 2. above that \eqref{max4} actually implies $$ \tr_{\om_B} \omf \le g \quad \mbox{on } X^\circ \setminus B$$
 hence on the whole $X^\circ$. \\

To conclude the proof of the RHS inequality in 4., it remains to prove the following

\begin{claim}
\label{claimm}
We have
\begin{equation}
\label{max3}
g^{-1}t^k \cdot \om_B^n \le (f^*\om_Z+t\om_B)^n \le g t^k \cdot \om_B^n
\end{equation}
\end{claim}

\begin{proof}[Proof of Claim~\ref{claimm}]
The statement is local, so one can assume that  $f:\mathbb C^n\to \mathbb C^m$ is given by the projection onto the last $m$ factors and that $B=\sum_{i=1}^r b_i(z_i=0)$ for some $r\le k$. As the inequality is invariant under quasi-isometry, one can choose $\om_Z=\sum_{j=k+1}^n idz_j\wedge d\bar z_j $ to be the euclidean metric on $\mathbb C^m$ while 
$$\om_B=\sum_{j=1}^r \frac{idz_j\wedge d\bar z_j }{|z_j|^{2b_j}}+\sum_{j=r+1}^n idz_j\wedge d\bar z_j 
$$ 
is the standard cone metric. 
Setting $K:=\prod_{j=1}^r|z_j|^{-2b_j}$ and $\om_{\mathbb C^n}:=\sum_{j=1}^n  idz_j\wedge d\bar z_j $, one finds
$$\om_B^n = K \cdot \om_{\mathbb C^n}^n \quad \mbox{and} \quad (f^*\om_Z+t\om_B)^n= t^k(1+t)^m K \cdot  \om_{\mathbb C^n}^n$$
which gives the expected result. 
\end{proof}

\textbf{4.b} We now move on to the LHS inequality in 4. Let us set $v:=\tr_{\omf}(t\om_B)$. Remember from Proposition~\ref{conic} 2. that $\om_B$ has holomorphic bisectional curvature bounded from above on $\xmb$. By Chern-Lu's inequality, we get on $\xmb$
$$\Delta_{\omf} \log v \ge -Ct^{-1}v.$$
Combining that inequality with \eqref{dphi}-\eqref{dphi2} and \eqref{rho}, one finds, for $A=C+1$
$$\Delta_{\omf}( \log v-\frac At (\vp_t-\uf)+A\psi_B+\delta \rho) \ge \frac 1 t(v-g). $$
Applying the maximum principle and arguing as before, we eventually find $v\le g$ on $X^{\circ}\setminus B$, hence on $X^\circ$. \\

\textbf{5.} The LHS inequality is a direct consequence of 4, by restriction. As for the RHS, it follows easily from the LHS since
\begin{align*}
\tr_{\om_{B_z}}\omf|_{X_z} & \le \frac{(\omf|_{X_z})^k}{\om_{B_z}^k} \cdot (\tr_{\omf|_{X_z}} \om_{B_z})^{k-1}\\
& \le gt^{k-(k-1)}
\end{align*}
thanks to \eqref{restr}. This ends the proof of Proposition~\ref{estimates}. 
\end{proof}

\noindent
\textbf{Step 3.} \textit{Convergence}

\noindent
Thanks to Proposition~\ref{estimates} 1., the family $(\vp_t)_{0<t\le1}$ is relatively compact for the $L^1$-topology. All we have to do is showing that all its clusters values coincide. Let $\vpi$ be such a cluster value; it is an $f^*\om_Z$-psh function but $f$ has connected fibers so that $\vpi$ is necessarily constant on the fibers. Equivalently, one has $\vpi=f^*\uvpi$ for the (unique) $\om_Z$-psh function $\uvpi$ satisfying $\uvpi(z):=\int_{X_z} \vpi \om_X^k$ for $z\in Z^\circ$. We want to show that the following equality of measures
\begin{equation}
\label{omi}
(\om_Z+dd^c \uvpi)^m=f_*\mu_{(X,B)}
\end{equation}
holds on $Z$. Given that Eq.~\ref{omi} has a unique normalized bounded solution, this will prove the Theorem. As $\uvpi$ is globally bounded on $X$ thanks to Proposition~\ref{estimates} 1. and $f_*\mu_{(X,B)}$ does not charge any pluripolar set, it is actually enough to show that the equality of measures \eqref{omi} holds on $Z^\circ$.  
In order to prove \eqref{omi} on $Z^\circ$, given that $f_*\om_X^k=1$, it is enough to prove instead that for any function $u\in \mathcal C^{\infty}_0(Z^\circ)$, one has
\begin{equation}
\label{omi2}
\int_{X} f^*u \cdot (f^*\om_Z+dd^c \vpi)^m\wedge \om_X^k=\int_X f^*u \cdot d\mu_{(X,B)}.
\end{equation}

\bigskip

We start from the identity 
\begin{equation}
\label{identity}
\omf^n=(f^*\om_Z+t\om_X+dd^c \f_t)^n=V_t \cdot \mu_{(X,B)}
\end{equation}
where $V_t =\binom{n}{k}  t^k+o(t^k)$ when $t\to 0$, cf \eqref{vt}. Set $\psi_t:=\f_t-\uf$ and decompose $\omf$ as $$\omf=f^*(\om_Z+dd^c \uf)+(t\om_X+dd^c \psi_t).$$ 
By expanding, one obtains
$$\omf^n=\sum_{i=0}^m \binom{n}{i} \underbrace{f^*(\om_Z+dd^c \uf)^i\wedge (t\om_X+dd^c \psi_t)^{n-i}}_{=:\alpha_i}.$$

$\bullet$ \underline{Case $i=m$}. 

\noindent
We expand again 
$$\alpha_m=\sum_{j=0}^{k-1} \binom{k}{j} t^j \underbrace{f^*(\om_Z+dd^c \uf)^m\wedge \om_X^j \wedge (dd^c \psi_t)^{k-j}}_{=:\beta_j}+t^kf^*(\om_Z+dd^c \uf)^m\wedge \om_X^k .$$
Performing an integration by parts, one gets 
$$\int_{X}f^*u \cdot \beta_j= \int_X \psi_t \cdot f^*\underbrace{\Big(dd^c u \wedge (\om_Z+dd^c \uf)^m\Big)}_{=0}\wedge \om_X^j \wedge  (dd^c \psi_t)^{k-j-1}=0$$
for degree reasons.

 By dominated convergence theorem, we have that $\uf\to\uvpi$ in the $L^1_{\rm loc}(Z^\circ)$ topology. Moreover, as $B$ intersects the fibers of $f$ tranversally over $Z^\circ$, an easy argument relying on partition of unity shows that $f_*(\om_B\wedge \omega_X^{k})$ is a smooth $(1,1)$-form on $Z^\circ$. Combining this with Proposition~\ref{estimates} 4., we find $dd^c \uf = f_*(dd^c\f_t \wedge \om_X^k) \le  f_*(g \om_B\wedge \om_X^k) \le (f_*g) \cdot \om_Z$. Together with \eqref{borneinf3}, this implies
\begin{equation}
\label{push}
\pm dd^c\uf \le (f_*g) \cdot \om_Z.
\end{equation}
By standard result, this shows that $\uf\to \uvpi$ in $C^{1, \alpha}_{\rm loc}(Z^\circ)$ for any $\alpha<1$. In particular, the quasi-psh functions $\uf$ converge uniformly on $\mathrm{Supp}(u)$. By Bedford-Taylor theory, one deduces that 
\begin{equation*}
\label{conv}
\int_X f^*u \cdot f^*(\om_Z+dd^c \uf)^m\wedge \om_X^k \to \int_X f^*u \cdot f^*(\om_Z+dd^c \uvpi)^m\wedge \om_X^k.
\end{equation*}
In the end, one has showed that
\begin{equation}
\label{conv2}
\frac{\binom{n}{m} }{V_t}\int_X f^*u \cdot \alpha_m \to \int_X f^*u \cdot f^*(\om_Z+dd^c \uvpi)^m\wedge \om_X^k.
\end{equation}
since $V_t\sim\binom{n}{m}  t^k$.\\

$\bullet$ \underline{Case $i<m$}. 

\noindent
We expand  
$$\alpha_i=\sum_{j=0}^{n-i-1} \binom{n-i}{j} t^j \underbrace{f^*(\om_Z+dd^c \uf)^i\wedge \om_X^j \wedge (dd^c \psi_t)^{n-i-j}}_{=:\gamma_{ij}}+ t^{n-i} f^*(\om_Z+dd^c \uf)^i\wedge \om_X^{n-i}.$$
From \eqref{push}, we find 
\begin{equation}
\label{last}
\frac{t^{n-i}}{V_t} \int_X f^*u \cdot f^*(\om_Z+dd^c \uf)^i\wedge \om_X^{n-i} = O(t^{m-i})=o(1).
\end{equation}
 For the remaining terms, an integration by parts yields 
$$\int_{X}f^*u \cdot \gamma_{ij}= \int_X \psi_t \cdot f^*\big(dd^c u \wedge (\om_Z+dd^c \uf)^i\big)\wedge \om_X^j \wedge  (dd^c \psi_t)^{n-i-j-1}$$
From Proposition~\ref{estimates} 3., one has $|\psi_t| \le gt$. Moreover, among the $(n-i-j-1)$ eigenvalues of $dd^c \psi_t$ involved in the integral, at least $(n-i-j-1)-(m-(i+1))=k-j$ must come from the fiber. Given Proposition~\ref{estimates} 4-5., the integrand is a $O(t^{1+k-j})$. As a result, 
$$\frac{t^j}{V_t}\int_{X}f^*u \cdot \gamma_{ij}= O(t).$$
Combining that result with \eqref{last}, we see that for any $i>m$, one has  
\begin{equation}
\label{zero}
\lim_{t\to 0} \frac{1}{V_t} \int_X f^*u \cdot \alpha_i=0.
\end{equation}
Putting together \eqref{identity}, \eqref{conv2} and \eqref{zero}, we obtain
\begin{align*}
\int_X f^*u \cdot d\mu_{(X,B)} & = \frac{1}{V_t} \int_X f^*u \cdot  \omf^n \\
&= \lim_{t\to 0} \sum_{i=0}^m\binom{n}{i} \frac{1}{V_t} \int_X f^*u \cdot  \alpha_i \\
&= \lim_{t\to 0} \frac{\binom{n}{m} }{V_t} \int_X f^*u \cdot \alpha_m \\
&= \int_X f^*u \cdot f^*(\om_Z+dd^c \uvpi)^m\wedge \om_X^k.
\end{align*} 
In summary, \eqref{omi2} is proved, which concludes the proof of the Theorem.
\end{proof}

\bibliographystyle{smfalpha}
\bibliography{biblio}

\newcommand{\etalchar}[1]{$^{#1}$}
\providecommand{\bysame}{\leavevmode ---\ }
\providecommand{\og}{``}
\providecommand{\fg}{''}
\providecommand{\smfandname}{\&}
\providecommand{\smfedsname}{\'eds.}
\providecommand{\smfedname}{\'ed.}
\providecommand{\smfmastersthesisname}{M\'emoire}
\providecommand{\smfphdthesisname}{Th\`ese}
\begin{thebibliography}{KKMSD73}

\bibitem[AAZ20]{AAZ18}
{\scshape S.~Asserda, F.~Assila {\normalfont \smfandname} A.~Zeriahi} -- {\og
  {Projective Logarithmic Potentials}\fg}, \emph{Indiana Univ. Math. J.}
  \textbf{69} (2020), p.~487--516.

\bibitem[Ale96]{Alex}
{\scshape V.~Alexeev} -- {\og Log canonical singularities and complete moduli
  of stable pairs\fg}, Preprint
  \href{https://arxiv.org/abs/alg-geom/9608013}{arXiv:alg-geom/9608013}, 1996.

\bibitem[Aub78]{Aubin}
{\scshape T.~Aubin} -- {\og \'{E}quations du type {M}onge-{A}mp\`ere sur les
  vari\'et\'es k\"ahl\'eriennes compactes\fg}, \emph{Bull. Sci. Math. (2)}
  \textbf{102} (1978), no.~1, p.~63--95.

\bibitem[BBE{\etalchar{+}}19]{BBEGZ}
{\scshape R.~J. Berman, S.~Boucksom, P.~Eyssidieux, V.~Guedj {\normalfont
  \smfandname} A.~Zeriahi} -- {\og K\"{a}hler-{E}instein metrics and the
  {K}\"{a}hler-{R}icci flow on log {F}ano varieties\fg}, \emph{J. Reine Angew.
  Math.} \textbf{751} (2019), p.~27--89.

\bibitem[BCHM10]{BCHM}
{\scshape C.~Birkar, P.~Cascini, C.~Hacon {\normalfont \smfandname}
  J.~McKernan} -- {\og {Existence of minimal models for varieties of log
  general type}\fg}, \emph{J. Amer. Math. Soc.} \textbf{23} (2010),
  p.~405--468.

\bibitem[BEGZ10]{BEGZ}
{\scshape S.~Boucksom, P.~Eyssidieux, V.~Guedj {\normalfont \smfandname}
  A.~Zeriahi} -- {\og {Monge-Amp{\`e}re equations in big cohomology
  classes.}\fg}, \emph{Acta Math.} \textbf{205} (2010), no.~2, p.~199--262.

\bibitem[Bei19]{Bei}
{\scshape F.~Bei} -- {\og On the {L}aplace--{B}eltrami operator on compact
  complex spaces\fg}, \emph{Trans. Amer. Math. Soc.} \textbf{372} (2019),
  no.~12, p.~8477--8505.

\bibitem[BG14]{BG}
{\scshape R.~J. Berman {\normalfont \smfandname} H.~Guenancia} -- {\og
  {K{\"a}hler-Einstein metrics on stable varieties and log canonical
  pairs}\fg}, \emph{Geometric and Functional Analysis} \textbf{24} (2014),
  no.~6, p.~1683--1730.

\bibitem[Bou02]{Bou02}
{\scshape S.~Boucksom} -- {\og {On the volume of a line bundle.}\fg},
  \emph{Int. J. Math.} \textbf{13} (2002), no.~10, p.~1043--1063.

\bibitem[Bou04]{DZD}
\bysame , {\og Divisorial {Z}ariski decompositions on compact complex
  manifolds\fg}, \emph{Ann. Sci. {\'E}cole Norm. Sup. (4)} \textbf{37} (2004),
  no.~1, p.~45--76.

\bibitem[BT82]{BT82}
{\scshape E.~Bedford {\normalfont \smfandname} B.~Taylor} -- {\og A new
  capacity for plurisubharmonic functions\fg}, \emph{Acta Math.} \textbf{149}
  (1982), no.~1-2, p.~1--40.

\bibitem[CGP13]{CGP}
{\scshape F.~Campana, H.~Guenancia {\normalfont \smfandname} M.~P\u{a}un} --
  {\og {Metrics with cone singularities along normal crossing divisors and
  holomorphic tensor fields}\fg}, \emph{Ann. Scient. Éc. Norm. Sup.}
  \textbf{46} (2013), p.~879--916.

\bibitem[CGP21]{JHM1}
{\scshape J.~Cao, H.~Guenancia {\normalfont \smfandname} M.~P{\u{a}}un} -- {\og
  Variation of singular {K{\"a}hler}-{Einstein} metrics: positive {Kodaira}
  dimension\fg}, \emph{J. Reine Angew. Math.} \textbf{779} (2021), p.~1--36
  (English).

\bibitem[CGP23]{JHM2}
\bysame , {\og Variation of singular {K{\"a}hler}-{Einstein} metrics: {Kodaira}
  dimension zero (with an appendix by {Valentino} {Tosatti})\fg}, \emph{J. Eur.
  Math. Soc. (JEMS)} \textbf{25} (2023), no.~2, p.~633--679 (English).

\bibitem[CGZ13]{CGZ13}
{\scshape D.~Coman, V.~Guedj {\normalfont \smfandname} A.~Zeriahi} -- {\og
  Extension of plurisubharmonic functions with growth control\fg}, \emph{J.
  Reine Angew. Math.} \textbf{676} (2013), p.~33--49.

\bibitem[Cha84]{Chavel}
{\scshape I.~Chavel} -- \emph{Eigenvalues in {R}iemannian geometry}, Pure and
  Applied Mathematics, vol. 115, Academic Press, Inc., Orlando, FL, 1984,
  Including a chapter by Burton Randol, With an appendix by Jozef Dodziuk.

\bibitem[Che68]{Chern}
{\scshape S.-s. Chern} -- {\og On holomorphic mappings of hermitian manifolds
  of the same dimension\fg}, in \emph{Entire {F}unctions and {R}elated {P}arts
  of {A}nalysis ({P}roc. {S}ympos. {P}ure {M}ath., {L}a {J}olla, {C}alif.,
  1966)}, Amer. Math. Soc., Providence, R.I., 1968, p.~157--170.

\bibitem[CL81]{ChengLi}
{\scshape S.~Y. Cheng {\normalfont \smfandname} P.~Li} -- {\og Heat kernel
  estimates and lower bound of eigenvalues\fg}, \emph{Comment. Math. Helv.}
  \textbf{56} (1981), no.~3, p.~327--338.

\bibitem[CT15]{CT15}
{\scshape T.~Collins {\normalfont \smfandname} V.~Tosatti} -- {\og K\"{a}hler
  currents and null loci\fg}, \emph{Invent. Math.} \textbf{202} (2015), no.~3,
  p.~1167--1198.

\bibitem[DDNL18]{DDL18}
{\scshape T.~Darvas, E.~Di~Nezza {\normalfont \smfandname} C.~H. Lu} -- {\og On
  the singularity type of full mass currents in big cohomology classes\fg},
  \emph{Compos. Math.} \textbf{154} (2018), no.~2, p.~380--409.

\bibitem[Dem82]{Dem82}
{\scshape J.-P. Demailly} -- {\og Sur les nombres de {L}elong associ\'{e}s \`a
  l'image directe d'un courant positif ferm\'{e}\fg}, \emph{Ann. Inst. Fourier
  (Grenoble)} \textbf{32} (1982), no.~2, p.~ix, 37--66.

\bibitem[Dem85]{Dem85}
\bysame , {\og Mesures de {M}onge-{A}mp{\`e}re et caract{\'e}risation
  g{\'e}om{\'e}trique des vari{\'e}t{\'e}s alg{\'e}briques affines\fg},
  \emph{M{\'e}m. Soc. Math. France (N.S.)} (1985), no.~19, p.~124.

\bibitem[Dem92]{D2}
\bysame , {\og Regularization of closed positive currents and intersection
  theory\fg}, \emph{J. Algebraic Geom.} \textbf{1} (1992), no.~3, p.~361--409.

\bibitem[Din09]{Din09}
{\scshape S.~a. Dinew} -- {\og Uniqueness in {$\mathscr E(X,\omega)$}\fg},
  \emph{J. Funct. Anal.} \textbf{256} (2009), no.~7, p.~2113--2122.

\bibitem[DP04]{DP}
{\scshape J.-P. Demailly {\normalfont \smfandname} M.~P{\u{a}}un} -- {\og
  Numerical characterization of the {K}\"{a}hler cone of a compact {K}\"{a}hler
  manifold\fg}, \emph{Ann. of Math. (2)} \textbf{159} (2004), no.~3,
  p.~1247--1274.

\bibitem[DP10]{DemPal}
{\scshape J.-P. Demailly {\normalfont \smfandname} N.~Pali} -- {\og Degenerate
  complex {M}onge-{A}mp\`ere equations over compact {K}\"ahler manifolds\fg},
  \emph{Internat. J. Math.} \textbf{21} (2010), no.~3, p.~357--405.

\bibitem[EGZ08]{EGZ08}
{\scshape P.~Eyssidieux, V.~Guedj {\normalfont \smfandname} A.~Zeriahi} -- {\og
  A priori {$L^\infty$}-estimates for degenerate complex {M}onge-{A}mp{\`e}re
  equations\fg}, \emph{Int. Math. Res. Not.} (2008), p.~Art. ID rnn 070, 8.

\bibitem[EGZ09]{EGZ}
\bysame , {\og {Singular K{\"a}hler-Einstein metrics}\fg}, \emph{{J. Amer.
  Math. Soc.}} \textbf{22} (2009), p.~607--639.

\bibitem[EGZ18]{EGZ16}
\bysame , {\og Convergence of weak {K}\"{a}hler-{R}icci flows on minimal models
  of positive {K}odaira dimension\fg}, \emph{Comm. Math. Phys.} \textbf{357}
  (2018), no.~3, p.~1179--1214.

\bibitem[FN80]{FN}
{\scshape J.~E. Forn{\ae}ss {\normalfont \smfandname} R.~Narasimhan} -- {\og
  The {L}evi problem on complex spaces with singularities\fg}, \emph{Math.
  Ann.} \textbf{248} (1980), no.~1, p.~47--72.

\bibitem[Fuj16]{Fuj16}
{\scshape O.~Fujino} -- {\og Direct images of relative pluricanonical
  bundles\fg}, \emph{Algebr. Geom.} \textbf{3} (2016), no.~1, p.~50--62.

\bibitem[GP16]{GP}
{\scshape H.~Guenancia {\normalfont \smfandname} M.~P{\u{a}}un} -- {\og {Conic
  singularities metrics with prescribed Ricci curvature: the case of general
  cone angles along normal crossing divisors}\fg}, \emph{J. Differential Geom.}
  \textbf{103} (2016), no.~1, p.~15--57.

\bibitem[GTZ13]{GTZ13}
{\scshape M.~Gross, V.~Tosatti {\normalfont \smfandname} Y.~Zhang} -- {\og
  Collapsing of abelian fibered {C}alabi-{Y}au manifolds\fg}, \emph{Duke Math.
  J.} \textbf{162} (2013), no.~3, p.~517--551.

\bibitem[Gue14]{G12}
{\scshape H.~Guenancia} -- {\og {K{\"a}hler-Einstein metrics with mixed
  Poincar{\'e} and cone singularities along a normal crossing divisor}\fg},
  \emph{Ann. Inst. Fourier} \textbf{64} (2014), no.~6, p.~1291--1330.

\bibitem[GW16]{GW}
{\scshape H.~Guenancia {\normalfont \smfandname} D.~Wu} -- {\og On the boundary
  behavior of {K}{\"a}hler-{E}instein metrics on log canonical pairs\fg},
  \emph{Math. Annalen} \textbf{366} (2016), no.~1, p.~101--120.

\bibitem[GZ05]{GZ05}
{\scshape V.~Guedj {\normalfont \smfandname} A.~Zeriahi} -- {\og {Intrinsic
  capacities on compact K{\"a}hler manifolds.}\fg}, \emph{J. Geom. Anal.}
  \textbf{15} (2005), no.~4, p.~607--639.

\bibitem[GZ07]{GZ07}
\bysame , {\og {The weighted Monge-Amp{\`e}re energy of quasi plurisubharmonic
  functions}\fg}, \emph{J. Funct. An.} \textbf{250} (2007), p.~442--482.

\bibitem[GZ17]{GZ17}
\bysame , \emph{Degenerate complex {M}onge-{A}mp\`ere equations}, EMS Tracts in
  Mathematics, vol.~26, European Mathematical Society (EMS), Z\"{u}rich, 2017.

\bibitem[HT20]{HT18}
{\scshape H.-J. Hein {\normalfont \smfandname} V.~Tosatti} -- {\og Higher-order
  estimates for collapsing {Calabi}-{Yau} metrics\fg}, \emph{Camb. J. Math.}
  \textbf{8} (2020), no.~4, p.~683--773 (English).

\bibitem[JMR16]{JMR}
{\scshape T.~Jeffres, R.~Mazzeo {\normalfont \smfandname} Y.~A. Rubinstein} --
  {\og K\"ahler-{E}instein metrics with edge singularities\fg}, \emph{Ann. of
  Math. (2)} \textbf{183} (2016), no.~1, p.~95--176, with an Appendix by C. Li
  and Y. Rubinstein.

\bibitem[Kar00]{Karu}
{\scshape K.~Karu} -- {\og Minimal models and boundedness of stable
  varieties\fg}, \emph{J. Algebraic Geom.} \textbf{9} (2000), no.~1,
  p.~93--109.

\bibitem[KKMSD73]{KKMS}
{\scshape G.~Kempf, F.~F. Knudsen, D.~Mumford {\normalfont \smfandname}
  B.~Saint-Donat} -- \emph{Toroidal embeddings. {I}}, Lecture Notes in
  Mathematics, Vol. 339, Springer-Verlag, Berlin-New York, 1973.

\bibitem[KM98]{KM}
{\scshape J.~Koll{\'a}r {\normalfont \smfandname} S.~Mori} -- \emph{Birational
  geometry of algebraic varieties}, Cambridge Tracts in Mathematics, vol. 134,
  Cambridge University Press, Cambridge, 1998, With the collaboration of C. H.
  Clemens and A. Corti, Translated from the 1998 Japanese original.

\bibitem[Kob84]{KobR84}
{\scshape R.~Kobayashi} -- {\og Einstein-k{\"a}hler metrics on open algebraic
  surfaces of general type\fg}, \emph{T{\^o}hoku Math. J. (2)} \textbf{37}
  (1984), p.~43--77 (English).

\bibitem[Kol]{Kollar}
{\scshape J.~Koll{\'a}r} -- {\og Book on moduli of surfaces\fg}, ongoing
  project, avalaible at the author's webpage
  \href{https://web.math.princeton.edu/~kollar/book/chap3.pdf}{Book}.

\bibitem[Ko{\l}98]{Kolo}
{\scshape S.~Ko{\l}odziej} -- {\og {The complex Monge-Amp{\`e}re operator}\fg},
  \emph{Acta Math.} \textbf{180} (1998), no.~1, p.~69--117.

\bibitem[Kol13]{Kollar13}
{\scshape J.~Koll\'{a}r} -- \emph{Singularities of the minimal model program},
  Cambridge Tracts in Mathematics, vol. 200, Cambridge University Press,
  Cambridge, 2013, With a collaboration of S\'{a}ndor Kov\'{a}cs.

\bibitem[Kol18]{Kol18}
{\scshape J.~Koll{\'a}r} -- {\og Families of varieties of general type\fg},
  2018, Book in preparation, available at
  \url{https://web.math.princeton.edu/~kollar/}.

\bibitem[Kov13]{Kov}
{\scshape S.~J. Kov{\'a}cs} -- {\og Singularities of stable varieties\fg}, in
  \emph{Handbook of moduli. {V}ol. {II}}, Adv. Lect. Math. (ALM), vol.~25, Int.
  Press, Somerville, MA, 2013, p.~159--203.

\bibitem[KS01]{KS01}
{\scshape M.~Kontsevich {\normalfont \smfandname} Y.~Soibelman} -- {\og
  Homological mirror symmetry and torus fibrations\fg}, in \emph{Symplectic
  geometry and mirror symmetry ({S}eoul, 2000)}, World Sci. Publ., River Edge,
  NJ, 2001, p.~203--263.

\bibitem[KSB88]{KSB}
{\scshape J.~Koll{\'a}r {\normalfont \smfandname} N.~I. Shepherd-Barron} --
  {\og Threefolds and deformations of surface singularities\fg}, \emph{Invent.
  Math.} \textbf{91} (1988), no.~2, p.~299--338.

\bibitem[LT95]{LT95}
{\scshape P.~Li {\normalfont \smfandname} G.~Tian} -- {\og On the heat kernel
  of the {B}ergmann metric on algebraic varieties\fg}, \emph{J. Amer. Math.
  Soc.} \textbf{8} (1995), no.~4, p.~857--877.

\bibitem[Lu68]{Lu}
{\scshape Y.-C. Lu} -- {\og {On holomorphic mappings of complex
  manifolds.}\fg}, \emph{J. Diff. Geom.} \textbf{2} (1968), p.~299--312.

\bibitem[MS73]{MS}
{\scshape J.~H. Michael {\normalfont \smfandname} L.~M. Simon} -- {\og Sobolev
  and mean-value inequalities on generalized submanifolds of {$R^{n}$}\fg},
  \emph{Comm. Pure Appl. Math.} \textbf{26} (1973), p.~361--379.

\bibitem[P{\u{a}}u07]{Paun07}
{\scshape M.~P{\u{a}}un} -- {\og Siu's invariance of plurigenera: a one-tower
  proof\fg}, \emph{J. Differential Geom.} \textbf{76} (2007), no.~3,
  p.~485--493.

\bibitem[P{\u{a}}u08]{Paun}
{\scshape M.~P{\u{a}}un} -- {\og {Regularity properties of the degenerate
  Monge-Amp{\`e}re equations on compact K{\"a}hler manifolds.}\fg}, \emph{Chin.
  Ann. Math., Ser. B} \textbf{29} (2008), no.~6, p.~623--630.

\bibitem[Rub14]{Rub14}
{\scshape Y.~A. Rubinstein} -- {\og Smooth and singular {K}\"{a}hler-{E}instein
  metrics\fg}, in \emph{Geometric and spectral analysis}, Contemp. Math., vol.
  630, Amer. Math. Soc., Providence, RI, 2014, p.~45--138.

\bibitem[RZ11a]{RZ}
{\scshape X.~Rong {\normalfont \smfandname} Y.~Zhang} -- {\og Continuity of
  extremal transitions and flops for {C}alabi-{Y}au manifolds\fg}, \emph{J.
  Differential Geom.} \textbf{89} (2011), no.~2, p.~233--269, Appendix B by
  Mark Gross.

\bibitem[RZ11b]{RZ0}
{\scshape W.-D. Ruan {\normalfont \smfandname} Y.~Zhang} -- {\og Convergence of
  {C}alabi-{Y}au manifolds\fg}, \emph{Adv. Math.} \textbf{228} (2011), no.~3,
  p.~1543--1589.

\bibitem[Siu87]{Siu87}
{\scshape Y.-T. Siu} -- \emph{{Lectures on Hermitian-Einstein Metrics for
  Stable Bundles and K{\"a}hler-Einstein Metrics}}, Birkh{\"a}user, 1987.

\bibitem[Siu98]{Siu98}
{\scshape Y.-T. Siu} -- {\og Invariance of plurigenera\fg}, \emph{Invent.
  Math.} \textbf{134} (1998), no.~3, p.~661--673.

\bibitem[Sko72]{Sko72}
{\scshape H.~Skoda} -- {\og Sous-ensembles analytiques d'ordre fini ou infini
  dans {${\bf C}^{n}$}\fg}, \emph{Bull. Soc. Math. France} \textbf{100} (1972),
  p.~353--408.

\bibitem[Son17]{Song17}
{\scshape J.~Song} -- {\og {Degeneration of K{\"a}hler-Einstein manifolds of
  negative scalar curvature}\fg}, Preprint
  \href{http://arxiv.org/abs/1706.01518}{arXiv:1706.01518}, 2017.

\bibitem[SSW20]{SSW}
{\scshape J.~Song, J.~Sturm {\normalfont \smfandname} X.~Wang} -- {\og
  {Continuity of the Weil-Petersson potential}\fg}, Preprint
  \href{https://arxiv.org/abs/2008.11215}{arXiv:2008.11215}, 2020.

\bibitem[Tak07]{Taka07}
{\scshape S.~Takayama} -- {\og On the invariance and the lower semi-continuity
  of plurigenera of algebraic varieties\fg}, \emph{J. Algebraic Geom.}
  \textbf{16} (2007), no.~1, p.~1--18.

\bibitem[Tak19]{Taka19}
\bysame , {\og A filling-in problem and moderate degenerations of minimal
  algebraic varieties\fg}, \emph{Algebr. Geom.} \textbf{6} (2019), no.~1,
  p.~26--49.

\bibitem[Tos09]{Tos09}
{\scshape V.~Tosatti} -- {\og {Limits of Calabi-Yau metrics when the K{\"a}hler
  class degenerates}\fg}, \emph{J. Eur. Math. Soc.} \textbf{11} (2009), no.~4,
  p.~755--776.

\bibitem[Tos10]{Tos10}
\bysame , {\og Adiabatic limits of {R}icci-flat {K}\"ahler metrics\fg},
  \emph{J. Differential Geom.} \textbf{84} (2010), no.~2, p.~427--453.

\bibitem[Tos20]{TosSurvey}
\bysame , {\og Collapsing {Calabi}-{Yau} manifolds\fg}, in \emph{Differential
  geometry, Calabi-Yau theory, and general relativity. Lectures given at
  conferences celebrating the 70th birthday of Shing-Tung Yau at Harvard
  University, Cambridge, MA, USA, May 2019}, Somerville, MA: International
  Press, 2020, p.~305--337 (English).

\bibitem[Tsu88]{Tsuji88}
{\scshape H.~Tsuji} -- {\og Existence and degeneration of
  {K}{\"a}hler-{E}instein metrics on minimal algebraic varieties of general
  type\fg}, \emph{Math. Ann.} \textbf{281} (1988), no.~1, p.~123--133.

\bibitem[TWY18]{TWY18}
{\scshape V.~Tosatti, B.~Weinkove {\normalfont \smfandname} X.~Yang} -- {\og
  The {K}\"{a}hler-{R}icci flow, {R}icci-flat metrics and collapsing
  limits\fg}, \emph{Amer. J. Math.} \textbf{140} (2018), no.~3, p.~653--698.

\bibitem[Yau78]{Yau78}
{\scshape S.-T. Yau} -- {\og {On the Ricci curvature of a compact K{\"a}hler
  manifold and the complex Monge-Amp{\`e}re equation. I.}\fg}, \emph{Commun.
  Pure Appl. Math.} \textbf{31} (1978), p.~339--411.

\bibitem[Zer01]{Zer01}
{\scshape A.~Zeriahi} -- {\og Volume and capacity of sublevel sets of a
  {L}elong class of plurisubharmonic functions\fg}, \emph{Indiana Univ. Math.
  J.} \textbf{50} (2001), no.~1, p.~671--703.

\end{thebibliography}

\end{document}